\documentclass[aos]{imsart}

\RequirePackage{amsthm,amsmath,amsfonts,amssymb}
\RequirePackage[square,numbers]{natbib}
\RequirePackage[colorlinks,citecolor=blue,urlcolor=blue]{hyperref}
\usepackage{xr-hyper}
\RequirePackage[utf8]{inputenc}
\RequirePackage[OT1]{fontenc}

\usepackage{algorithm}
\usepackage[noend]{algpseudocode}

\RequirePackage{graphicx}
\RequirePackage[usenames,dvipsnames]{xcolor}

\definecolor{mypink1}{rgb}{0.858, 0.188, 0.478}
\definecolor{mypink2}{RGB}{219, 48, 122}
\definecolor{mypink3}{cmyk}{0, 0.7808, 0.4429, 0.1412}
\definecolor{mygray}{gray}{0.6}

\makeindex             

\startlocaldefs

\newcommand{\F}{\mathscr{F}}

\newcommand{\f}{\rightarrow}
\def\1{\mathbf{1}}

\newcommand{\data}{y^{(n)}} 
\newcommand{\Data}{Y^{(n)}} 

\newcommand{\vir}[1]{``#1''}

\newcommand{\pg}[1]{\left\{#1\right\}}
\newcommand{\pq}[1]{\left[#1\right]}
\newcommand{\pt}[1]{\left(#1\right)}
\newcommand{\abs}[1]{\left\vert#1\right\vert}

\newcommand{\di}{\mathrm{d}}

\numberwithin{equation}{section}
\theoremstyle{plain}
\newtheorem{thm}{Theorem}[section]
\newtheorem{prop}{Proposition}[section]
\newtheorem{rmk}{Remark}[section]

\newtheorem{lem}{Lemma}[section]

\newtheorem{ass}{Assumption}[section]

\usepackage{mathrsfs} 
\usepackage{dsfont}
\usepackage{bm}

\usepackage[sectionbib]{bibunits}
\usepackage{nameref}
\usepackage[nameinlink,noabbrev,capitalise]{cleveref}

\endlocaldefs

\numberwithin{equation}{section}

\begin{document}

\begin{frontmatter}
\title{Wasserstein convergence in Bayesian deconvolution models}
\runtitle{Wasserstein convergence in Bayesian deconvolution models}

\begin{aug}
\author[A]{\fnms{Judith} \snm{Rousseau}\ead[label=e1]{judith.rousseau@stats.ox.ac.uk}},
\author[B]{\fnms{Catia} \snm{Scricciolo}\ead[label=e2,mark]{catia.scricciolo@univr.it}}
\address[A]{University of Oxford, Department of Statistics, Oxford, UK
\printead{e1}}

\address[B]{Department of Economics, Università di Verona,  Verona, Italy
\printead{e2}}
\end{aug}

\begin{abstract}
We study the reknown deconvolution problem of recovering a distribution function from independent replicates (signal) additively contaminated with random errors (noise), whose distribution is known. 
We investigate whether a Bayesian nonparametric approach for modelling the latent distribution of the signal can yield inferences with asymptotic frequentist validity under the $L^1$-Wasserstein metric. 
When the error density is ordinary smooth, we develop two inversion inequalities relating   either the $L^1$  or the $L^1$-Wasserstein distance between two mixture densities (of the observations) to the $L^1$-Wasserstein distance between the corresponding distributions of the signal. This smoothing inequality improves on those in the literature.
We apply this general result to a Bayesian approach bayes on a Dirichlet process mixture of normal distributions as a prior on the mixing distribution (or distribution of the signal), with a Laplace or  Linnik noise.  In particular  we construct an \textit{adaptive} approximation of the density of the observations  by the convolution of a Laplace (or Linnik) with a well chosen  mixture of normal densities and show that the posterior concentrates at the minimax rate 
up to a logarithmic factor. The same prior law is shown to also adapt to the Sobolev regularity level of the mixing density, thus leading to a new Bayesian estimation method, relative to the Wasserstein distance, for distributions with smooth densities. 
\end{abstract}


\end{frontmatter}

\section{Introduction} \label{sec:intro}
In many applied problems of econometrics, biometrics, medical statistics, image reconstruction and signal deblurring
data are observed with some random error, see, \emph{e.g.}, \cite{Meister:2009} for relevant examples. One observes 
\begin{equation} \label{model}
Y_i = X_i + \varepsilon_i, \quad i\leq  n, \quad \varepsilon_i \stackrel{\mathrm{iid}}{\sim} \mu_\varepsilon, \quad X_i \stackrel{\mathrm{iid}}{\sim} \mu_X,
\end{equation}
where the noise $\varepsilon_i$ is independent of the corresponding signal of interest $X_i$, for $i=1,\,\ldots,\,n$. 
Estimation of the distribution of the $X_i$'s has received a lot of attention in the literature and is known as deconvolution problem, which is a prototypical linear inverse problem. In this paper, we consider Bayesian nonparametric estimation of $\mu_X$, when the distribution of $\varepsilon$ is known.
This is an extensively studied problem, from both the theoretical and methodological perspectives.
There is a rich literature on frequentist estimation of both the distribution $\mu_X$ of $X$ and its density $f_X$, with ground breaking papers of the early 90's based on Fourier inversion techniques to construct estimators of $f_X$, see \cite{caroll:hall:88, fan1993, diggle:hall:93, Delaigle2004BootstrapBS} or \cite{comte:etal:06} using penalized contrast estimators, just to mention a few. For instance, minimax rates for estimating $f_X$ have been studied in \cite{fan1993, butucea:tsybakov08}. 
Some recent interesting results have been obtained for the estimation of the distribution of $\mu_X$ in terms of the $L^1$-Wasserstein metric. State of the art results are provided in \cite{dedecker2015}, where minimax estimation rates are derived and a minimum distance estimator of $\mu_X$ is constructed, which achieves these rates.

While a wide range of frequentist estimators have been studied from a theoretical point of view, little is known about the theoretical properties of Bayesian nonparametric precedures. Contrariwise to frequentist kernel methods where the estimators are explicit, the posterior distribution of $\mu_X$ is not explicit, which 
renders the analysis much more difficult. A typical way to measure how well the posterior distribution recovers $\mu_X $ is to study posterior contraction rates, 
\emph{i.e.}, to determine rates $\epsilon_n = o(1)$ such that 
\begin{equation}
\Pi( d( \mu_X,\,\mu_{0X}) > \epsilon_n \mid Y^{(n)} ) =o_{\mathbb P}(1)
\end{equation}
under model \eqref{model}, where $\mu_{0X}$ is the true mixing distribution, $\Pi(\cdot \mid Y^{(n)})$ is the posterior distribution 
and $d(\cdot,\,\cdot)$ is a loss function on probability measures. 

If the approach to study posterior convergence rates developed in \cite{ghosal2000,ghosal:vdv:07} has proven to be highly successful for a wide range of models and prior distributions, it is not enough to derive sharp bounds on posterior convergence rates for $\mu_X$, since the latter is an inverse problem.
To derive posterior convergence rates for $\mu_X$, one typically first obtains posterior convergence rates in the direct problem, \emph{i.e.}, 
for $\mu_Y = \mu_\varepsilon\ast \mu_X$ or its density $f_Y$, and then combine it with an inversion inequality which translates an upper bound on a distance  between $\mu_Y $ and $\mu_Y'$ into an upper bound on $d(\mu_X,\,\mu_X')$. These two aspects, the direct posterior concentration rate and the inversion inequality, have an interest in themselves. This approach has been used successfully in other contexts  by \cite{Gine:Nickl:11,nickl:sohl} for instance. In this paper we consider the first approach which we believe sheds light on the relation between various interesting metrics in these models. 

In \cite{MR3706755} the inversion inequality strategy has been used to derive posterior $L^2$-norm convergence rates for the density $f_X$ from $L^2$-norm posterior convergence rates for the density $f_Y$. However, as mentioned in \cite{hall:lahiri08}, the distribution function of $\mu_X$ itself is also important in practice, so that considering weaker metrics, like Wasserstein metrics, is of interest. In this paper, we are interested in recovering $\mu_X$ in terms of the $L^1$-Wasserstein distance. 
Over the last decade there has been a growing interest in Wasserstein metrics in statistics and machine learning for both discrete and continuous distributions and we refer to \cite{weed2020minimax}, \cite{dedecker2015} and \cite{nguyen2013} for discussions on the use of Wasserstein metrics as loss functions for distributions.  

In the case of a Laplace noise, \cite{gao2016} derived posterior contraction rates in Wasserstein metrics under the assumption that $X$ has bounded support and \cite{scricciolo:18} extended and improved this result to the unbounded case. However, in both cases the authors obtain suboptimal rates. Recently, \cite{su2020nonparametric} proposed a Bayesian nonparametric method to recover $\mu_X$ under heteroscedastic noise and proved consistency of the method, but did not derive convergence rates. In a related problem, \cite{nguyen2013} has derived Wasserstein posterior contraction rates for the mixing distribution in smooth mixture models.  The author obtains the rate $n^{-1/4}$ in the $L^2$-Wasserstein metric for specific distributions $\mu_{0X}$, but this rate is suboptimal for general mixing distributions 
$\mu_{0X}$.

The construction of Bayesian minimax optimal methods for the estimation of $\mu_X$ in Wasserstein distances remains therefore an open issue. 
In this paper, we bridge the gap by studying Wasserstein posterior contraction rates under model \eqref{model}, 
when the noise $\varepsilon$ has ordinary smooth density $f_\varepsilon$ with regularity $\beta\geq 1/2$. 
Specifically, in Theorem \ref{theo:1}, we derive a general inversion inequality relating either the total variation or the $L^1$-Wasserstein distance 
between $\mu_Y$ and $\mu_{0Y}$ to the $L^1$-Wasserstein distance between $\mu_X$ and $\mu_{0X}$. 
This inversion inequality is sharper than the one obtained by \cite{gao2016} and \cite{scricciolo:18} and allows to derive 
nearly minimax $L^1$-Wasserstein posterior contraction rates in the case of Laplace noise, see Section \ref{sec:nonadaptive}. 
This inversion inequality is of interest in itself and can also be used to study frequentist estimators.

Another contribution of this paper is the construction of Bayesian adaptive estimation procedures for the density $f_Y$ of $\mu_Y$.
Posterior convergence rates for $f_Y$ have been widely studied in the Bayesian nonparametric mixture models literature, 
but mostly for Gaussian mixtures, see, \emph{e.g.}, \cite{ghosal2001} and \cite{scricciolo:12}. When the noise follows a Laplace distribution, \cite{gao2016} and \cite{scricciolo:18} have obtained the rate $n^{-3/8}$ (up to a $\log n $-term) in the Hellinger or $L^1$-distance for estimating $f_Y$ using a Dirichlet process mixture on $\mu_X$. As noted by \cite{gao2016}, this corresponds to the minimax estimation rate for densities belonging to Sobolev balls with smoothness $\beta =3/2$, which is the case for $\mu_Y$, as argued in Section \ref{subsec:Wrate}.
Under the additional assumption that $\mu_X$ has Lebesgue density $f_X$, we prove in Theorem \ref{thm:31} that this rate can be improved to $n^{-2/5}$ and, more generally, that the rate $n^{-\beta/(2\beta+1)} $ can be obtained when the noise follows a Linnik distribution with index $1<\beta\leq2$,
the Linnik distribution with $\beta =2$ being the Laplace distribution. We also study the case where $f_X$ is either H\"older or Sobolev $\alpha$-regular and obtain a Hellinger convergence rate for $f_Y$ of the order $O(n^{-(\alpha+\beta)/[2(\alpha+\beta)+1]})$ up to a $(\log n)$-term, see Theorem \ref{thm:4}. 
To obtain such results, we consider a Dirichlet process mixture of Gaussian densities as a prior on $f_X$. We believe that the approximation theory developed in Section \ref{subsec:adapt} to approximate the true density $f_{0Y}$ by $f_\varepsilon\ast f_X$, where $f_X$ is modelled as a mixture of Gaussian densities, is itself of interest. 

The paper is organized as follows. In Section \ref{sec:notation} we present the set-up and notation used throughout the paper. 
Section \ref{sec:general} contains the general posterior contraction rate theorem in terms of the $L^1$-Wasserstein 
distance relative to $\mu_X$ (Section \ref{subsec:Wrate}) together with an inversion inequality (Section \ref{subsec:inversion}). 
In Section \ref{sec:lap+exp} we apply the general theorem to the case where the noise has a Linnik distribution with index $1<\beta\leq2$ and 
the prior on $f_X$ is a Dirichlet process mixture of Gaussian densities. The proofs of Theorem \ref{thm:22} on posterior 
contraction rates for $L^1$-Wasserstein deconvolution and Theorem \ref{theo:1} on the inversion inequality are presented 
in Section \ref{sec:proofs}. Additional proofs are presented in the 
Supplement \cite{rousseau:scricciolo:supp} with lemmas, equations and sections referenced with a prefix S, 
to differentiate them from those of the main paper.

\section{Set-up and notation}\label{sec:notation}
We observe a sample  $Y^{(n)} = (Y_1,\,\dots,\, Y_n)$ from the model $Y_i = X_i + \varepsilon_i$ in \eqref{model},
where the random variables $X_i$ and $\varepsilon_i$ are independent, the noise $\varepsilon_i$ has known distribution 
$\mu_\varepsilon$ with Lebesgue density $f_\varepsilon$, which
is assumed to be ordinary smooth with parameter $\beta>0$, \emph{i.e.}, its Fourier transform $\hat f_{\varepsilon}$ verifies
\begin{equation}\label{eq:1}
d_0|t|^{-\beta}|\leq|\hat f_{\varepsilon}(t)|\leq d_1|t|^{-\beta}, \quad \mbox{as } t\rightarrow \infty,
\end{equation}
for constants $d_0,\,d_1>0$. Examples of ordinary smooth densities include the Laplace ($\beta=2$) distribution, 
all Linnik distributions with index $\beta \in (0,\,2]$ and the gamma distribution with shape parameter $\beta>0$.  

Let $\mathscr P$ stand for the set of all probability measures on $(\mathbb R,\,\mathcal B(\mathbb R))$ and $\mathscr P_0$ for
the subset of Lebesgue absolutely continuous distributions on $\mathbb R$. 
Denote by $\F$ the class of probability measures $\mu_Y=\mu_\varepsilon\ast \mu_X$ for $\mu_X\in\mathscr P_0$, with density $f_X$. 
Since $\mu_Y$ is Lebesgue absolutely continuous, we denote by $f_Y=f_\varepsilon\ast \mu_X$ its density.
For any $\mathscr P_1\subseteq \mathscr P_0$, let $\F(\mathscr P_1) $ stand for the set of probability measures $\mu_Y=\mu_\varepsilon\ast \mu_X$ with $\mu_X \in \mathscr P_1$. We consider a prior $\Pi$ on $\mathscr P$ and denote by $\Pi_n( \cdot | Y^{(n)})$ the resulting posterior distribution, with
$$ \Pi_n(B\mid Y^{(n)}) = \frac{ \int_B \prod_{i=1}^n f_Y(Y_i) \,\Pi(\di\mu_X) }
{\int_{\mathscr P}\prod_{i=1}^n f_Y(Y_i)\,\Pi(\di\mu_X)}, \quad B\in\mathcal B(\mathbb R).$$
Our aim is to assess the posterior concentration rate in $L^1$-Wasserstein distance for $\mu_X$, namely, to find a sequence
$\epsilon_n = o(1)$ such that, if $ Y^{(n)}$ is an $n$-sample from model \eqref{model} with true mixing distribution $\mu_{0X}$, then
$$ 
\Pi_n(W_1(\mu_X, \,\mu_{0X})\leq\epsilon_n \mid Y^{(n)})  \rightarrow 1
$$ 
in $P_{0Y}^n$-probability, with the $L^1$-Wasserstein distance $W_1(\mu_X, \,\mu_{0X})$ defined as 
$$ W_1(\mu_X,\,\mu_{0X}):=\inf_{\mu\text{-couplings}} \int |X-X'|\,\di\mu(X,X') = \int_{\mathbb R}  | F_X(x) - F_{0X}(x)|\,\di x,$$
where $\inf_{\mu\text{-couplings}}$ denotes the infimum over all couplings of $(X,\,X')$ having
$\mu_X $ and $\mu_{0X}$ as marginal distributions, see, \emph{e.g.}, \cite{villani2009}. 
The last identity, in which $F_X$ and $F_{0X}$ denote the distribution functions of $\mu_X$ and $\mu_{0X}$, respectively,
is valid only in dimension one.

We now introduce some notation that will be used throughout the paper. 
For functions $f,\,g\in L^1(\mathbb R)$, let $(f\ast g)(\cdot)=\int_{\mathbb R} f(\cdot-u)g(u)\,\di u$
be the convolution of $f$ and $g$. 
We denote by $\phi (x) =(2\pi)^{-1/2} e^{-x^2/2}$ the density of a standard Gaussian random variable and 
by $\phi_\sigma(x) = \phi(x/\sigma)/\sigma$ its rescaled version. When it exists, $f^{(k)}$ denotes the $k$th derivative of $f$, 
for any integer $k \geq 0$, with $f^{(0)}=f$.

Let $d_\mathrm{H}(f_1,\,f_2):=\| \sqrt{f_1} - \sqrt{f_2} \|_2$ be the Hellinger distance between densities $f_1$ and $f_2$, 
where $\|f\|_r$ is the $L^r$-norm of $f$, for $r\geq1$.
For probability measures $P_{0Y}$ and $P_Y$, let
$\mathrm{KL}(P_{0Y};\,P_Y):=\mathbb E_{0Y}[\log(f_{0Y}/f_Y)(Y)]$ be the Kullback-Leibler divergence 
of $P_Y$ from $P_{0Y}$ and, for $\epsilon>0$, let
$$B_{\mathrm{KL}}(P_{0Y};\,\epsilon^2)=
\pg{P_Y\in\mathscr P:\, \mathrm{KL}(P_{0Y};\,P_Y)\leq\epsilon^2,\,\,\,
\mathbb E_{0Y}\pt{\log\frac{f_{0Y}}{f_Y}}^2\leq\epsilon^2}$$
be the $\epsilon$-Kullback-Leibler type neighbourhood of $P_{0Y}$. 

Let $C_b(S)$ be the set of bounded, continuous real-valued functions on $S\subseteq \mathbb R$.
For any $\alpha>0$, let $\mathscr F_\alpha^{\mathcal S}=\{f:\,f\geq0,\,\|f\|_1=1 \mbox{ and }\int_{\mathbb R}|t|^{2\alpha}|\hat f(t)|^2\,\di t<\infty\}$ be the Sobolev space of densities of order $\alpha$ and let $\mathscr C_{\alpha}(L)$ be the H\"older ball with radius $L>0$, 
\emph{i.e.}, the set of functions  $f$ on $\mathbb R$ that are $\ell:=\lfloor \alpha\rfloor -1$ times continuously differentiable and such that the $\ell$th derivative 
satisfies $|f^{(\ell)}(x+\delta)-f^{(\ell)}(x)|\leq L|\delta|^{\alpha-\ell}$, for every $\delta,\,x\in\mathbb R$. In the above notation,
let $\lfloor x \rfloor=\max\, \{k\in\mathbb{Z} : k<x\}$ be the lower integer part of $x$. Similarly, we write
$\lceil x \rceil=\max\, \{k\in\mathbb{Z} : k>x\}$ for the upper integer part,
$[x]=\max\, \{k\in\mathbb{Z} : k\leq x\}$ for the integer part and $\{x\}=x-[x]$ for its fractional part when $x\in \mathbb{R}^+=\left\{x\in \mathbb {R}:\, x\geq 0\right\}$.  

For $f\in L^1(\mathbb R)$, let $\hat f (t):= \int_{\mathbb R} e^{\imath t x} f(x)\,\di x$, $t \in \mathbb R$, 
be its Fourier transform. For any function $f$ for which $\int_{\mathbb R}|t|^{\alpha}|\hat f(t)|\,\di t<\infty$, with $\alpha\geq 0$, define the 
$\alpha$th fractional derivative of $f$ as $D^{\alpha}\hspace*{-1pt}f(x):={(2\pi)}^{-1}\int_{\mathbb{R}}\exp{(-\imath t x)}(-\imath t)^\alpha\hat f(t)\,\di t$.
For $\alpha=0$, the convention $D^0f \equiv f$ holds.

For $\epsilon>0$, let $D(\epsilon,\,B,\,d)$ be the $\epsilon$-packing number of a set $B$ with 
metric $d$, that is, the maximal number of points in $B$ such that the $d$-distance between every pair is at least $\epsilon$,
where $d$ can be either the Hellinger or the $L^1$-metric.

We write  $a\vee b=\max\{a,\,b\}$,  $a\wedge b=\min\{a,\,b\}$ and
 $a_+=a\vee 0$.
Also  $a_n\lesssim b_n$  (resp. $a_n\gtrsim b_n$) means that  $a_n\leq Cb_n$ (resp. $a_n\geq Cb_n$) for some  $C>0$ that is universal or depends only on $P_{0Y}$ 
and  $a_n\asymp b_n$ means that both $a_n \lesssim b_n$ and $b_n \lesssim a_n$ hold.


\section{Posterior contraction rates for $L^1$-Wasserstein deconvolution}\label{sec:general}
In this section we present a general theorem on $L^1$-Wasserstein contraction rates for the posterior measure on the 
mixing distribution, which is based on properties of the prior law and the true data generating process. A key tool of the proof is an inversion inequality relating the $L^1$-Wasserstein distance $W_1(\mu_X,\,\mu_{0X})$ between the mixing distributions with the $L^1$-norm distance $\|f_Y - f_{0Y}\|_1$ between the corresponding mixed densities. The inequality is also of interest in itself and is given in Section \ref{subsec:inversion}.

\subsection{A general result on $L^1$-Wasserstein posterior contraction rates}\label{subsec:Wrate}
In order to obtain $L^1$-Wasserstein posterior contraction rates for the latent distribution $\mu_X$, 
we make assumptions on the \vir{true} mixing distribution $\mu_{0X}$ and the error distribution $\mu_\varepsilon$.  
If $\mu_\varepsilon$ has Lebesgue density $f_\varepsilon$, then
its characteristic function coincides with the Fourier transform of $f_\varepsilon$, denoted by $\hat f_\varepsilon$. 
If $|\hat f_{\varepsilon}(t)|\neq 0$, $t\in\mathbb R$, then the reciprocal of $\hat f_{\varepsilon}$,
\begin{equation}\label{eq:re}
r_\varepsilon(t):=\frac{1}{\hat f_\varepsilon(t)}, \quad t\in\mathbb{R},
\end{equation}
is well defined. For an $l$-times differentiable Fourier transform
$\hat f_\varepsilon$, with
$l\in\mathbb \{0\}\cup\mathbb N$, the $l$th derivative of $r_\varepsilon$ is denoted by $r^{(l)}_\varepsilon$, 
with $r_\varepsilon^{(0)}\equiv r_\varepsilon$.

\begin{ass}\label{ass:smoothXXX}
\emph{The probability measure $\mu_{0X}\in\mathscr P_0$ 
has finite first moment $\mathbb E_{0X}[|X|]<\infty$ and 
possesses Lebesgue density $f_{0X}$ verifying either one of the following conditions:
\begin{itemize}
\item[(i)]
there exist $\alpha>0$ and $L_0\in L^1(\mathbb R)$ such that the derivative $f^{(\ell)}_{0X}$ of order 
$\ell=\lfloor \alpha\rfloor$ exists and satisfies
\begin{equation}\label{eq:holder}
|f^{(\ell)}_{0X}(x+\delta)-f^{(\ell)}_{0X}(x)|\leq L_0(x)|\delta|^{\alpha-\ell}, \quad\mbox{for every }\delta,\,x\in\mathbb R,
\end{equation}
\end{itemize}
or 
\begin{itemize}
\item[(ii)]
there exists $\alpha>0$ such that
\begin{equation}\label{ass:smoothsob}
\int_{\mathbb R}|t|^\alpha|\hat f_{0X}(t)|\,\di t<\infty, \quad 
D^{\alpha}\hspace*{-1pt}f_{0X}\in L^1(\mathbb R).
\end{equation}
\end{itemize}
}
\end{ass}


\begin{ass}\label{ass:identifiability+error}
\emph{The error distribution $\mu_\varepsilon\in\mathscr P_0$ has finite first moment $\mathbb E[|\varepsilon|]<\infty$ 
and possesses Lebesgue density $f_\varepsilon$ with Fourier transform 
$|\hat f_{\varepsilon}(t)|\neq 0$, $t\in\mathbb{R}$. 
Furthermore, there exists $\beta>0$ such that, for $l=0,\,1$,
\begin{equation}\label{eq:deriv}
|r_\varepsilon^{(l)}(t)|\lesssim (1+|t|)^{\beta-l},\quad t\in\mathbb R.
\end{equation}
}
\end{ass}

\medskip
The discussion of Assumptions \ref{ass:smoothXXX} and \ref{ass:identifiability+error} is postponed to
Section \ref{subsec:inversion}.
To state the main theorem, whose proof is reported in Section \ref{sec:proofs}, we need to introduce some definitions. Depending on whether (i) or (ii) of Assumption \ref{ass:smoothXXX} holds true, we consider a kernel of order $(\lfloor\alpha\rfloor+1)$ or a supersmooth kernel, respectively. The kernels we consider are functions $K\in L^1(\mathbb R)\cap L^2(\mathbb R)$ with compactly supported Fourier transforms. More precisely,
\begin{itemize} 
\item[(a)] in the case of \emph{kernels of order $\ell$}, see \emph{e.g.}, \cite{Meister:2009}, pp. 38-39, used under (i) of Assumption \ref{ass:smoothXXX}, 
we have $\int_{\mathbb R} K(z)\,\di z=1$, while $\int_{\mathbb R}z^j K(z)\,\di z= 0$ for $j=1,\,\ldots,\,\ell$, with $\hat K$ supported on $[-1,\,1]$;
\item[(b)] in the case of supersmooth kernels, used under (ii) of Assumption \ref{ass:smoothXXX}, we have $K$ symmetric satisfying $\int_{\mathbb R}|z||K(z)|\,\di z<\infty$, with $\hat K$ supported on $[-2,\,2]$, while $\hat K\equiv 1$ on $[-1,\,1]$.
\end{itemize}
In case (b), a key property is that there exists $A:=1+\|K\|_1<\infty$ such that
\begin{equation*}
\sup_{|t|\neq0}\frac{|1-\hat K(t)|}{|t|^\alpha}\leq A, \quad\mbox{for all }\alpha>0.
\end{equation*}
We use the notation $K_h(\cdot):=(1/h)K(\cdot/h)$ for the rescaled kernel and
$b_{F_X}:= F_X\ast K_h - F_X$ for the \vir{bias} of the distribution function $F_X$
of a probability measure $\mu_X$, with the proviso that, for $\alpha>0$, the kernel $K$ is either of order
$(\lfloor\alpha\rfloor+1)$ or supersmooth, depending on which hypothesis between
(i) or (ii) holds true.

\begin{thm}\label{thm:22}
Let $\Pi_n$ be a prior distribution on $\mathscr P$. Let $\mu_{0X}\in\mathscr P$ have finite first moment $\mathbb E_{0X}[|X|]<\infty$ and 
$\mu_\varepsilon$ satisfy Assumption \ref{ass:identifiability+error} for $\beta>0$. 
Suppose that, for a positive sequence
$\tilde{\epsilon}_n\f0$, with
$n\tilde{\epsilon}^2_n\f\infty$,
constants $c_1,\,c_2,\,c_3,\,c_4>0$ and sets $\mathscr P_n\subseteq \mathscr P$, we have
\begin{equation}\label{con1}
\begin{split}
\log D(\tilde{\epsilon}_n,\,\F(\mathscr P_n),\,d) &\leq c_1n\tilde{\epsilon}^2_n,\\
 \Pi_n(\mathscr P\setminus \mathscr P_n )&\leq c_3\exp{(-(c_2+4)n\tilde{\epsilon}^2_n)},\\ 
\Pi_n(B_{\mathrm{KL}}(P_{0Y},\,\tilde\epsilon_n^2))&\geq c_4\exp{(-c_2n\tilde{\epsilon}^2_n)},
\end{split}
\end{equation}
with
\begin{equation}\label{eq:mfin}
\mathbb E_{\mu_X}[|X|]<\infty\mbox{ for $\mu_X \in \mathscr P_n$.}
\end{equation} 
Then, for
$\epsilon_n:=(\tilde\epsilon_n\log n)^{1/(\beta\vee1)}$ and a sufficiently large constant $\bar K$,
$$\mathbb E_{0Y}^n[\Pi_n(\mu_X:\,W_1(\mu_X,\,\mu_{0X})>\bar K\epsilon_n\mid Y^{(n)})]\rightarrow0.$$
If, in addition, $\mu_{0X}$ satisfies Assumption \ref{ass:smoothXXX} for $\alpha>0$ and
there exist constants $C_1,\,\bar h>0$ such that, for every $\mu_X\in\mathscr P_n$, 
\begin{equation}\label{eq:ass1}
\|b_{F_X}\|_1\leq C_1 h^{\alpha+1}\mbox{ for all $h\leq \bar h$,}
\end{equation}
then, for $\epsilon_{n,\alpha}:=(\tilde\epsilon_n\log n)^{(\alpha+1)/[\alpha+(\beta\vee1)]}$ and 
a constant $K_\alpha$ large enough,
\begin{equation}\label{eq:convsmooth}
\mathbb E_{0Y}^n[\Pi_n(\mu_X:\,W_1(\mu_X,\,\mu_{0X})>K_\alpha\epsilon_{n,\alpha}\mid Y^{(n)})]\rightarrow0.
\end{equation}
\end{thm}

\medskip

Theorem \ref{thm:22} provides sufficient conditions on the prior law and the data generating process 
so that the posterior measure asymptotically concentrates on $L^1$-Wasserstein balls centred at the sampling distribution.
Some remarks and comments on two main issues, namely, (i) relationship between rates in the direct and inverse problems,
(ii) rate optimality, are in order. Concerning the first issue,
Theorem \ref{thm:22} connects to existing results that give sufficient conditions
for assessing posterior convergence rates in the direct density estimation problem.
The conditions in \eqref{con1}, in fact, imply that, for sufficiently large constant $\bar M$, 
$$\mathbb E_{0Y}^n[\Pi_n(\mu_X:\,d(f_{Y},\,f_{0Y})> \bar M\tilde \epsilon_n\mid Y^{(n)})]=o(1),$$
see Theorem 2.1, p. 503, in \cite{ghosal2000}, which states that the posterior concentration rate in Hellinger or 
$L^1$-neighbourhoods of $f_{0Y}$ is $\tilde\epsilon_n$. Alternative conditions for assessing posterior contraction rates
in $L^r$-metrics, $1\leq r\leq \infty$, are given in \cite{Gine:Nickl:11}, see Theorems 2 and 3, pp. 2891--2892. 
A remarkable feature of Theorem \ref{thm:22} is the fact that, 
in order to obtain $L^1$-Wasserstein posterior convergence rates for $\mu_X$, 
which is an involved mildly ill-posed inverse problem, it is enough to derive 
posterior contraction rates in Hellinger or $L^1$-metric in the direct problem, 
which is more gestible. Granted Assumption \ref{ass:identifiability+error}, in fact,
the essential conditions to verify are those listed in \eqref{con1}, which are sufficient for the posterior 
law to contract at rate $\tilde\epsilon_n$ around $f_{0Y}$. This simplification is due to the inversion inequality of 
Theorem \ref{theo:1}, which holds true under Assumption \ref{ass:identifiability+error} only,
when no smoothness condition is imposed on $\mu_{0X}$, and jointly with
condition \eqref{eq:ass1}, when the regularity Assumption \ref{ass:smoothXXX} on $\mu_{0X}$ is in force.

Application of Theorem \ref{thm:22} to specific models gives further insight into this aspect. 
In Section \ref{sec:lap+exp} we consider a Dirichlet process mixture-of-Linnik-normals prior, 
with Linnik error distribution of index $1<\beta\leq2$, and we find the rate $n^{-1/(2\beta+1)}(\log n)^\nu$ 
when the latent distribution $\mu_{0X}$ is only known to have a density 
$f_{0X}$ and the rate $n^{-(\alpha +1)/[2(\alpha+\beta)+1]}(\log n)^{\tau}$ when a Sobolev regularity condition on
$f_{0X}$ holds true. These results are expected to be valid in greater generality. 
If, in fact, $|\hat f_\varepsilon(t)| \sim (1 + |t|)^{-\beta}$, as $|t|\rightarrow\infty$, 
and $f_{0X} \in  \mathscr F_\alpha^{\mathcal S}$, with $\alpha>0$, 
then $f_{0Y}\in\mathscr F^{\mathcal S}_{\alpha+\beta}$ and the minimax rate in Hellinger or $L^1$-metric for estimating 
densities with Sobolev regularity $(\alpha+\beta)$ is $n^{-(\alpha+\beta)/[2( \alpha+\beta)+1]}$.
This would lead to an $L^1$-Wasserstein posterior convergence rate 
for $\mu_X$ of the order $O(n^{-(\alpha +1)/[2(\alpha+\beta)+1]})$ (up to a $\log n$ term) when $\beta\ge1$,
which reduces to $O(n^{-1/(2\beta+1)})$ when $\mu_{0X}$ is only known to have a density. 
The latter upper bound matches the lower bound on the convergence rate for the $L^1$-Wasserstein risk 
obtained by \cite{dedecker2015} in Theorem 4.1, p. 243, requiring only a moment condition on the latent distribution.
The following proposition extends the lower bound result on the rate of convergence for the $L^p$-Wasserstein risk, 
$p\geq1$, to the case where the mixing density is either H\"older smooth or Sobolev regular. 
\begin{prop}\label{prop:lower bound}
Assume that there exists $\beta>0$ such that, for every $l=0,\,1,\,2$,
\begin{equation}\label{eq:cond}
|\hat f_\varepsilon^{(l)}(t)|\leq d_l (1+|t|)^{-(\beta+l)},\quad t\in\mathbb R,
\end{equation}
with a constant $d_l>0$. There, then, exists a constant $C>0$ such that, for any estimator $\hat \mu_n$,
\begin{equation}\label{eq:LB1}
\mathop{\underline{\lim}}_{n \to \infty}n^{p(\alpha+1)/[2(\alpha+\beta)+1]}
\sup_{\mu_X\in\mathscr D_p(M)\cap \mathscr C}\mathbb E[W_p^p(\hat\mu_n,\,\mu_X)]>C,
\end{equation}
where $\mathscr C$ stands for either a H\"older $\mathscr C_{\alpha}(L)$ or 
a Sobolev $\mathscr F_\alpha^{\mathcal S}(L)$ class of densities with $\alpha,\,L>0$ and $\mathscr D_p(M)$ is the class of probability measures $\mu_X\in\mathscr P$ with uniformly bounded $p$th absolute moment $\mathbb E_{\mu_X}[|X|]\leq M$
for some $M>0$.
\end{prop}
The proof develops along the lines of Theorem 4.1 in \cite{dedecker2015}, appealing to
intermediate results in \cite{fan1991} and \cite{fan1993}, and is not reported here. 
Condition \eqref{eq:cond} is stronger than condition (22) of Theorem 4.1, 
which requires that, for $l=0,\,1,\,2$, the $l$th derivative
$|r_\varepsilon^{(l)}(t)|\leq c(1+|t|)^{-\beta}$, $t\in\mathbb R$.
For $p=1$, the lower bound in \eqref{eq:LB1} matches the upper bound 
of Theorem \ref{thm:22} when $\beta\geq1$.

\medskip




\noindent\emph{$L^2$-minimax rates over logarithmic Sobolev classes of densities}

\noindent
When no assumption on the mixing distribution $\mu_{0X}$ is postulated,  
the $L^2$-minimax rate for estimating a convolution density $f_{0Y}=f_\varepsilon \ast \mu_{0X}$, with
\begin{equation}\label{eq:ordinarysmooth}
|\hat f_\varepsilon(t)|\lesssim (1+|t|)^{-\beta}, \quad t\in\mathbb R,
\end{equation}
for $\beta>1/2$ to ensure that $\hat f_{0Y}\in L^2(\mathbb R)$, is to our knowledge unknown,
even if $f_{0Y}$ belongs to a Sobolev type class.
Following \cite{Haroske:1998} and \cite{Haroske:2000}, for $\gamma>0$, $\delta>1$ and
$w_{\gamma,\delta}(t):=(1+|t|^2)^{\gamma/2}(\log(e+|t|))^{-\delta/2}$, $t\in\mathbb R$, we define
the logarithmic Sobolev class of densities as
\begin{equation}\label{eq:weighted}
\mathscr F^{\mathcal{LS}}_{\gamma,\delta}(L):=
\{f:\, f\geq0,\,\|f\|_1=1 \mbox{ and }\|w_{\gamma,\delta}\hat f\|_2^2\leq L^2\},\quad L>0.
\end{equation} 
The Sobolev class of densities
$\mathscr F_\gamma^{\mathcal S}(L):=\{f:\, f\geq0,\, \|f\|_1=1 \mbox{ and }\|(1+|\cdot|^2)^{\gamma/2}\hat f\|_2^2\leq L^2\}$
corresponds to $\mathscr F^{\mathcal{LS}}_{\gamma,0}(L)$. 
The following proposition assesses the order, up to a logarithmic factor, 
of the $L^2$-minimax rate for estimating densities in a logarithmic Sobolev class. 
Although the result seems to be a well-known fact, we could not find a proof for it and, for completeness, we 
prove it in Section \ref{sec:proof:minimax} of the Supplement \cite{rousseau:scricciolo:supp}. 


\begin{prop}\label{eq:minimax1}
For $\psi_{n,\gamma}:=n^{-\gamma/(2\gamma+1)}$,
\[
\begin{split}
\psi_{n,\gamma}^2\lesssim\inf_{\hat f_n}\sup_{f\in \mathscr F_{\gamma,\delta}^{\mathcal{LS}}(L)} \mathbb E^n_f [\|\hat f_n-f\|_2^2]\lesssim \psi_{n,\gamma}^2(\log n)^{\delta/(2\gamma+1)},
\end{split}
\]  
where the infimum is taken over all estimators $\hat f_n$ for densities 
$f$ in $\mathscr F_{\gamma,\delta}^{\mathcal{LS}}(L)$                    
based on $n$ observations and the expectation is with respect to the n-fold product measure of $P_f$.
\end{prop}

If $f_\varepsilon$ is such that $\hat f_\varepsilon$ satisfies condition \eqref{eq:ordinarysmooth}, then, 
for $\delta>1$, extending the definition of $w_{\gamma,\delta}$ to $\gamma<0$, we have
$$f_{0Y}\in\mathscr F_{\beta-1/2,\delta}^{\mathcal{LS}}(L), \quad \mbox{with }L = \|w_{{-1/2},\delta}\|_2.$$
In fact, 
$\|w_{\beta-1/2,\delta}\hat f_{0Y}\|_2^2 
\leq \|w_{\beta-1/2,\delta}\hat f_\varepsilon\|_2^2 \lesssim
\|w_{{-1/2},\delta}\|_2^2<\infty$.
Thus, from Proposition \ref{eq:minimax1}, the $L^2$-minimax rate over 
$\mathscr F^{\mathcal{LS}}_{\beta-1/2,\delta}(L)$, with $L=\|w_{-1/2,\delta}\|_2$, is
$\psi_{n,\beta-1/2}=n^{-(\beta-1/2)/(2\beta)}$ up to a logarithmic factor.
For a Laplace error distribution ($\beta=2$), the rate specializes to
$\psi_{n,3/2}=n^{-3/8}$, which is attained by the Bayes' estimator $f_{nY}^{\mathrm B}$,
defined as the posterior mean associated to a Dirichlet process mixture of Laplace densities, 
so that $\mathbb E_{0Y}^n[d^2_{\mathrm H}(f_{nY}^{\mathrm B},\,f_{0Y})]=O(n^{-3/4}(\log n))$. 
Since  $f_{nY}^{\mathrm B} = f_\varepsilon \ast \mu_{nX}^{\mathrm B}$, using   
the inversion inequality in \eqref{eq:t2}, we find the rate for the (squared) $L^1$-Wasserstein distance 
$\mathbb E_{0Y}^n[W_1^2( \mu_{nX}^{\mathrm B},\mu_{0X})]=O(n^{-3/8}(\log n)^{3/2})= o(n^{-1/4})$, 
where the $L^1$-Wasserstein rate $n^{-1/8}$ has been obtained by \cite{gao2016} and \cite{scricciolo:18}. 
Yet, this rate is larger than the $L^1$-Wasserstein minimax rate $n^{-1/5}$ pointed out by \cite{dedecker2015}. 
It is possible that a Dirichlet process mixture-of-Laplace prior may not achieve the $L^1$-Wasserstein lower bound rate $n^{-1/5}$, although we cannot rule out that either 
the Hellinger rate $n^{-3/8}(\log n)^{1/2}$ obtained by \cite{gao2016} in the direct problem is not sharp or that the inversion inequality is sharp only 
when the Hellinger or $L^1$-distance between mixture densities is of the order $O(n^{-\beta/(2\beta+1)})$. Note that this order is attained 
for Laplace mixtures when the mixing distribution possesses a density, see Section \ref{sec:lap+exp}. 

\subsection{Inversion inequality}\label{subsec:inversion}
In this section we establish, under general and minimal conditions, an inversion inequality relating 
the $L^1$-distance between mixture densities to the
$L^1$-Wasserstein distance between the corresponding mixing distributions, when the error density 
has Fourier transform decaying polynomially at infinity. This inequality is crucial
in the proof of Theorem \ref{thm:22}. Inequalities of this type have been previously obtained by 
\cite{nguyen:2013} for the $L^2$-Wasserstein distance and by \cite{gao2016} and \cite{scricciolo:18} for the $L^1$-Wasserstein distance, 
but they are not as sharp as the one given in Theorem \ref{theo:1}, whose proof is postponed to Section \ref{sec:rth1}. 
Starting from \cite{dedecker2015}, the idea is to use a suitable kernel
to smooth the mixing distributions $F_X$ and $F_{0X}$ and then to bound the $L^1$-Wasserstein distance between 
the smoothed versions, meanwhile controlling the bias induced by the smoothing.

\begin{thm}\label{theo:1}
Let $\mu_X,\,\mu_{0X}\in\mathscr P$ be probability measures with finite 
first moments $\mathbb E_{\mu_X}[|X|]<\infty$ and $\mathbb E_{0X}[|X|]<\infty$.
Let $\mu_\varepsilon\in\mathscr P_0$ satisfy Assumption \ref{ass:identifiability+error} for $\beta>0$.
Then, for probability measures 
$\mu_Y:=\mu_\varepsilon\ast\mu_X$, $\mu_{0Y}:=\mu_\varepsilon\ast\mu_{0X}$, 
with densities $f_Y$, $f_{0Y}$, and a sufficiently small $h>0$,
\begin{equation*}
W_1(\mu_X,\,\mu_{0X})\lesssim h+ T,
\end{equation*}
where
\begin{equation}\label{eq:t2}
T\lesssim W_1(\mu_Y,\,\mu_{0Y})+
\begin{cases}
h^{-(\beta-1/2)_+}|\log h|^{1+\1_{\{\beta=1/2\}}/2}\, W_1(\mu_Y,\,\mu_{0Y}), &\\[-6pt]
\hspace*{0.4cm}\mbox{or} &\\
h^{-(\beta-1)_+}|\log h|\, d(f_Y,\,f_{0Y}), &
\end{cases}
\end{equation}
for $d$ being either the Hellinger or the $L^1$-metric. If, in addition, 
$\mu_X$ verifies condition \eqref{eq:ass1} and 
$\mu_{0X}$ satisfies Assumption \ref{ass:smoothXXX} for $\alpha>0$, then 
\begin{equation}\label{eq:ineqw21}
W_1(\mu_X,\,\mu_{0X})\lesssim h^{\alpha+1}+T,
\end{equation}
for $T$ as in \eqref{eq:t2}.
\end{thm}



We comment on Assumptions \ref{ass:smoothXXX} and \ref{ass:identifiability+error}.  
With Assumption \ref{ass:smoothXXX} we are restricting attention to the set $\mathscr P_0$ of Lebesgue absolutely continuous probability measures 
on $\mathbb R$ as mixing distributions. 
Assumption \ref{ass:smoothXXX} is, in fact, a smoothness/regularity condition on $f_{0X}$.
It requires that either $f_{0X}$ is locally H\"older smooth, namely, it has $\ell$ derivatives, 
for $\ell$ the largest integer strictly smaller than $\alpha$, 
with the $\ell$th derivative being H\"older of order $\alpha-\ell$ and integrable envelope function $L_0$ to bound the $L^1$-norm of the bias,
or $f_{0X}$ has (global) Sobolev regularity $\alpha$. Indeed, requiring that $D^\alpha\hspace*{-1pt}f_{0X} \in L^2(\mathbb R)$ is equivalent to
impose that $f_{0X}\in\mathscr F^{\mathcal S}_\alpha$, the difference being that $D^\alpha\hspace*{-1pt}f_{0X}$ is here assumed 
to be in $L^1(\mathbb R)$, see condition \eqref{ass:smoothsob}.
H\"older and Sobolev classes of densities are common nonparametric families of smooth functions.
If $f_{0X}$ verifies Assumption \ref{ass:smoothXXX}, then $\|b_{F_{0X}}\|_1=O(h^{\alpha+1})$,
see Lemmas \ref{lem:der} and \ref{lem:sob}.
Note that if the random density is modelled as a Gaussian mixture, 
$f_X=\mu_H\ast \phi_\sigma$, then $\|b_{F_X}\|_1\leq C_1 h^{\alpha+1}\|\mu_{H}\ast\phi_{\sigma/\sqrt{2}}\|_1 = C_1h^{\alpha+1}$
for a constant $C_1>0$ not depending on $\mu_H$, see Lemma \ref{lem:biasmixgaus}, so that condition \eqref{eq:ass1} is verified. 

In Assumption \ref{ass:identifiability+error} it is required that $\hat f_\varepsilon$ is everywhere non-null.
This is a standard hypothesis in density deconvolution problems related to identifiability with respect to the $L^1(\mathbb R)$-metric, 
which is a necessary condition for the existence of (weak) consistent density estimators of $f_{0X}$ 
with respect to the $L^1(\mathbb R)$-metric, see \cite{Meister:2009}, pp. 23--26.
Finiteness of the first moment of $\varepsilon$ is a technical condition with a two-fold aim.
First, jointly with the existence of the first moment of $\mu_{0X}$, it implies
that also $\mu_{0Y}$ has finite expected value. This guarantees that $\mu_{0X}$ and $\mu_{0Y}$
have finite $L^1$-Wasserstein distances from $\mu_X$ and $\mu_Y$, respectively,
which, in turn, are required to possess finite expectations.
Secondly, it implies that $\hat f_\varepsilon$ is continuously differentiable on $\mathbb R$ and that the derivative
is $\hat f^{(1)}_\varepsilon(t)=\int_{\mathbb R}\exp{(\imath tu)}(\imath u) f_\varepsilon(u)\,\di u$, $t\in\mathbb R$. Then, 
$r_\varepsilon^{(1)}$ exists and is well defined. 
Note that, for $l=0$, condition \eqref{eq:deriv} is equivalent to $|\hat f_\varepsilon(t)|\gtrsim (1+|t|)^{-\beta}$, $t\in\mathbb R$. 
The requirement is satisfied for ordinary smooth distributions covering the following examples.
\begin{itemize}  
 \item[$\bullet$] 
The symmetric Linnik distribution with $\hat f_\varepsilon(t)=(1+|t|^\beta)^{-1}$, $t\in\mathbb R$, 
for index $0<\beta\leq 2$ and scale parameter equal to $1$.
The standard Laplace distribution corresponds to $\beta=2$, see $\S$ 4.3 in \cite{Kotz2001}, pp. 249--276.
\item[$\bullet$]
The gamma distribution with $\hat f_\varepsilon(t)=(1-\imath t)^{-\beta}$, $t\in\mathbb R$,
for shape parameter $\beta>0$ and scale parameter equal to $1$. 
The standard exponential distribution corresponds to $\beta=1$.
\item[$\bullet$]
An error distribution with characteristic function $\hat f_\varepsilon$ that is the reciprocal of a polynomial,
$r_\varepsilon(t)=\sum_{j=0}^m a_j t^{s_j}$, $t\in\mathbb R$, with $a_j\in \mathbb C$, $j=0,\,\ldots,\,m$, and exponents $0\leq s_0< s_1<\ldots<s_m=\beta$ for $\beta>0$. This example extends Example 1 in 
\cite{bissantz}, p. 487,
wherein the $s_j$'s are taken to be non-negative integers $s_j=j$, for $j=0,\,\ldots,\,\beta$.
\item[$\bullet$]
The error distribution in Example 2 of \cite{bissantz}, p. 487, with $f_\varepsilon(u)=\gamma[g_0(u-\mu)+g_0(u+\mu)]/2+(1-\gamma)g_0(u)$, 
$u\in\mathbb R$, for a density $g_0$, constants $0<\gamma<{1}/{2}$ and $\mu\neq 0$,
having $\hat f_\varepsilon(t)=[(1-\gamma)+\gamma\cos(\mu t)]\hat g_0(t)$, $t\in\mathbb R$, 
with $|\hat g_0(t)|\gtrsim (1+|t|)^{-\beta}$ for $\beta>0$.
\end{itemize}  
Location and/or scale transformations of random variables with distributions as in the previous examples, as well as their convolutions as in the last example, satisfy condition \eqref{eq:deriv}, whereas  the \emph{uniform}, \emph{triangular} and \emph{symmetric gamma} distributions do \emph{not} satisfy it. 
Differently from \cite{dedecker2015}, in condition \eqref{eq:deriv} we do not assume that 
$r_\varepsilon$ is at least twice continuously differentiable. Instead,
as in \cite{MR3449779}, we only assume the existence of the first derivative such that 
$|r_\varepsilon^{(1)}(t)|\lesssim (1+|t|)^{\beta-1}$, $t\in\,\mathbb R$.

The result of Theorem \ref{theo:1} falls within the scope of inversion inequalities, 
which translate an $L^p$-distance, $p\geq1$, between kernel mixtures into a proximity measure
between the corresponding mixing distributions.
A first inequality has been obtained by \cite{nguyen2013}, Theorem 2, p. 377, for ordinary and super-smooth
kernel densities in convolution mixtures, see also \cite{heinrich2018}.
A refined version for the ordinary smooth case translating the Hellinger or $L^2$/$L^1$-distance 
between mixtures, with kernel density having polynomially decaying Fourier transform, into the $L^1$-Wasserstein distance
between mixing distributions with finite Laplace transforms in a neighborhood of zero,
has been elaborated by \cite{gao2016} and \cite{scricciolo:18}.
A lower bound on the $L^p$-Wasserstein risk, $p\geq 1$, with only a moment condition
on the latent distribution, has been obtained by \cite{dedecker2015}, who have also 
shown that it is attained by a minimum distance estimator $\hat \mu_n$. 
Their proof, however, does not rely on an inversion inequality, but on the explicit expression of 
$W_1(\hat \mu_n,\,\mu_{0X})$.

When $\beta\geq1$, we show in Lemma \ref{lem:1} that the Kullback-Leibler condition in \eqref{con1}
implies a posterior concentration rate for the $L^1$-Wasserstein distance $W_1(\mu_Y,\,\mu_{0Y})$ 
of the order $O(\tilde\epsilon_n)$, which is equal to the order of $d(f_Y,\,f_{0Y})$, possibly up to a logarithmic factor.
By \eqref{eq:ineqw21}, this in turn implies that, with posterior probability tending to one,
$$W_1(\mu_X,\,\mu_{0X})\lesssim h^{\alpha +1} + \tilde\epsilon_n + h^{-(\beta -1)_+}|\log h|\tilde\epsilon_n.$$ 
Optimizing with respect to $h$ and neglecting logarithmic factors leads to 
$$W_1(\mu_X,\,\mu_{0X}) \lesssim d(f_Y,\,f_{0Y})^{(\alpha+1)/(\alpha+\beta)}.$$
Then, $d(f_Y,\,f_{0Y})=O(n^{-(\alpha+\beta)/[2(\alpha+\beta)+1]})$ under the posterior distribution, 
which yields the minimax-optimal rate $W_1(\mu_X,\,\mu_{0X})=O(n^{-(\alpha+1)/[2(\alpha+\beta)+1]})$,
see Proposition \ref{prop:lower bound} on the lower bound.
One deficit of Theorem \ref{theo:1} is that the $d$-version of inequality \eqref{eq:ineqw21} does not satisfactorily
cover the case where $0<\beta<1$. When applied to this case, in fact, it yields the rate $n^{-(\alpha+\beta)/[2(\alpha+\beta)+1]}$, 
times a logarithmic factor, rather than $n^{-(\alpha+1)/[2(\alpha+\beta)+1]}$. The $W_1$-version of \eqref{eq:ineqw21} 
would improve the rate to
\begin{equation}\label{eq:minimax}
n^{-(\alpha+1)/[2\alpha+(2\beta\vee 1)+1]},
\end{equation}
within (at most) a log-factor, provided that $W_1(\mu_Y,\,\mu_{0Y})=O(n^{-1/2})$, whose proof, however, remains elusive.
The existence of an \emph{elbow effect} with different rate \emph{r\'egimes} for $0<\beta\leq1/2$ and $\beta>1/2$ has already been pointed
out by \cite{MR2906875} and \cite{MR3449779} for distribution function estimation. 
In fact, because of the identity $W_1(\mu_X,\,\mu_{0X})=\|F_X-F_{0X}\|_1
=\|F_X^{-1}-F_{0X}^{-1}\|_1$, where $F_X^{-1}$ and $F_{0X}^{-1}$ denote the left-continuous inverse or quantile functions, 
see, \emph{e.g.}, \cite{shorack&wellner}, pp. 64--66, 
the $L^1$-Wasserstein posterior convergence rate for estimating a latent probability measure coincides with the
$L^1$-metric rate for estimating the corresponding distribution or quantile function. 
On the other hand, the sequence in \eqref{eq:minimax} is the minimax rate for estimating single
quantiles in deconvolution when the error density $f_\varepsilon$ satisfies Assumption \ref{ass:identifiability+error} for $\beta>0$ and the mixing density 
$f_{0X}$ is locally $\alpha$-H\"older continuous for $\alpha>0$, see Corollary 2.9 of \cite{MR3449779}, p. 151.

\smallskip

A further remark concerns potential application of the $W_1$-version of \eqref{eq:ineqw21} to 
bound the $L^1$-Wasserstein risk of a frequentist estimator for the latent distribution.

\begin{prop}
Let $\tilde\mu_n$ be an estimator based on $n$ iid observations 
from a probability measure $\mu_{0Y}=\mu_\varepsilon\ast\mu_{0X}$, with $\mu_{0X}$ and $\mu_\varepsilon$ satisfying
Assumptions \ref{ass:smoothXXX} and \ref{ass:identifiability+error}, respectively. 
If $\tilde\mu_n$ is a probability measure such that 
$\mathbb E[W_1(\mu_\varepsilon\ast\tilde\mu_n,\,\mu_{0Y})]=O(n^{-1/2})$
up to a logarithmic factor, then, within a log-factor,
$$\mathbb E[W_1(\tilde\mu_n,\,\mu_{0X})]=O(n^{-(\alpha+1)/[2\alpha+(2\beta\vee 1)+1]}).$$
\end{prop}


\section{Deconvolution by a Dirichlet mixture-of-Linnik-normals prior}\label{sec:lap+exp}
In this section we study the problem of density deconvolution for mixtures with
a Linnik error distribution, the Laplace distribution being  the most well known example. This family of symmetric distributions 
was introduced in 1953 by Yu. V. Linnik \cite{linnik}.
Being scale mixtures of normal distributions, see, \emph{e.g.}, \cite{korolev}, Linnik distributions
can serve as the one-dimensional distributions of a special subordinated Wiener process
used to model the evolution of stock prices and financial indexes. 
Also, generalized Linnik distributions are good candidates for modelling financial data which 
exhibit high kurtosis and heavy tails \cite{mittnik}.

We use a Dirichlet process mixture-of-normals prior on the mixing density 
$f_X=\mu_H\ast \phi_\sigma$, so that  $f_Y=f_\varepsilon \ast f_X=f_\varepsilon \ast (\mu_H\ast \phi_\sigma)$, with 
 $\mu_H \sim \mathscr D_{H_0}$, for some finite, positive measure $H_0$, 
and $\sigma\sim\Pi_\sigma$. We consider the following assumptions on $H_0$ and $\Pi_\sigma$.

\begin{ass}\label{ass:basemeasure1}
\emph
{
The base measure $H_0$ has a continuous and positive density $h_0$ on $\mathbb R$ such that,
for constants $b_0,\,c_0>0$ and $0<\delta\leq1$,  
$$h_0(u)=c_0\exp{(-b_0|u|^\delta)},\quad u\in\mathbb R.$$
}
\end{ass}

\begin{ass}\label{ass:priorscale1}
\emph
{The prior distribution $\Pi_\sigma$ for $\sigma$ has a continuous and positive density $\pi_\sigma$ on $(0,\,\infty)$ 
such that, for constants $D_1,\,D_2>0$, $s_1,\,s_2,\,t_1,\,t_2\geq0$ and $0<\gamma<\infty$,
\[\sigma^{-s_1}\exp{(-D_1\sigma^{-\gamma}|\log\sigma|^{t_1})}\lesssim  \pi_\sigma(\sigma) \lesssim \sigma^{-s_2}\exp{(-D_2\sigma^{-\gamma}|\log\sigma|^{t_2})}
\]
for all $\sigma$ in a neighborhood of $0$. Furthermore, for some $0<\varpi<\infty$, 
the tail probability $\Pi_{\sigma}([\sigma,\infty))\lesssim \sigma^{-\varpi}$ as $\sigma\rightarrow \infty$.
}
\end{ass}

Assumption \ref{ass:basemeasure1} on the base measure $H_0$ of the Dirichlet process is analogous to (4.8) in \cite{scricciolo:2011}, p. 288, and holds true, for example, when the density $h_0$ is proportional to an exponential power distribution with shape parameter $0<\delta\leq1$, the Laplace distribution, which corresponds to $\delta=1$, being the most popular case.
Assumption \ref{ass:priorscale1} on the scale parameter $\sigma$ of the Gaussian kernel has become common in the literature since the articles of \cite{vdV+vZanten}, \cite{dejonge} and 
\cite{kruijer2010}, when a full-support prior for $\sigma$ is allowed.
An inverse-gamma distribution $\mathrm{IG}(\nu,\,\gamma)$, 
with shape parameter $\nu>0$ and scale parameter $\gamma>0$, is an eligible prior on $\sigma$ for $s_1=s_2=\nu+1$, $t_1=t_2=0$ and 
$\gamma=1$. The condition on the tail behaviour at $\infty$ is satisfied for $\varpi=\nu$.

Hereafter, we first study the case in which the sampling density is a location mixture of standard 
Linnik densities, with mixing distribution that is only known to have density 
with exponentially decaying tails at infinity and then the case in which the mixing density has 
Sobolev regularity $\alpha$. In the latter case, the prior distribution on the mixing density does not depend on $\alpha$, 
yet leads to an adaptive posterior contraction rate. 
We call these two cases as non-adaptive and adaptive, respectively, and treat them separately.

\subsection{Non-adaptive case} \label{sec:nonadaptive}
Let $\Pi$ be the prior law induced by the product measure $\mathscr D_{H_0}\otimes \Pi_\sigma$ 
on the parameter $(\mu_H,\,\sigma)$ of the random density $f_Y=f_\varepsilon\ast(\mu_H\ast \phi_\sigma)$, 
for a standard Linnik error density $f_\varepsilon$ with index $1<\beta\leq2$. 
Let $f_{0Y}=f_\varepsilon\ast f_{0X}$ be the Linnik mixture,
with mixing density $f_{0X}$ satisfying the following exponential tail decay condition.

\begin{ass}\label{ass:twicwtailcond}
\emph{
The probability measure $\mu_{0X}\in\mathscr P_0$ 
has density $f_{0X}(x)\lesssim e^{-(1+C_0)|x|}$, $x\in\mathbb R$, for some constant $C_0>0$.
}
\end{ass}


We begin by assessing posterior contraction rates in $L^1$-metric for Linnik 
convolution mixtures with mixing distributions having exponentially decaying tails.

\begin{thm}\label{thm:31}
Let $Y_1,\,\ldots,\,Y_n$ be i.i.d. observations from
$f_{0Y}:=f_\varepsilon\ast f_{0X}$, where $f_\varepsilon$ is the density of a 
standard Linnik distribution with index $1<\beta\leq2$ and $f_{0X}$ satisfies Assumption \ref{ass:twicwtailcond}. 
Let $\Pi$ be the prior law induced by $\mathscr D_{H_0} \otimes \Pi_\sigma$, where 
${H_0}$ verifies Assumption \ref{ass:basemeasure1} and $\Pi_\sigma$ verifies
Assumption \ref{ass:priorscale1} for $\gamma=1$. 
There, then, exist constants $D$ large enough and $\kappa>0$ so that
\[
\Pi_n(\mu_Y:\,\|f_Y-f_{0Y}\|_1>D n^{-\beta/(2\beta+1)}(\log n)^{\kappa}
\mid \Data)\rightarrow0\mbox{ in $P_{0Y}^n$-probability}.
\]
\end{thm}

\begin{proof}
To prove Theorem \ref{thm:31}, we verify that assumption \eqref{con1} is 
satisfied for $\tilde\epsilon_n=n^{-\beta/(2\beta+1)}(\log n)^\tau$, with $\tau>0$. 
The Kullback-Leibler condition, third equation of assumption \ref{con1}, is verified in 
Lemma \ref{lem:KL}, which is based on the construction of an approximation of $f_{0Y}$ by 
$f_\varepsilon\ast(\mu_H\ast\phi_\sigma)$ for a carefully chosen probability measure $\mu_H$. 
This construction uses the representation of Linnik densities 
with index $0<\beta<2$ as scale mixtures of Laplace densities
and adapts the proof of Lemma 2 of \cite{gao2016}, pp. 615--616, to obtain an approximation error of the order $O(\tilde\epsilon_n)$. 
These results are detailed in Lemma \ref{lem:discrete}.  
The entropy and remaining mass conditions, first two equations of assumption \eqref{con1},
are a consequence of Theorem 5 of \cite{ghosal:shen:tokdar}, p. 631, because, for any pair of densities
 $f_1$ and $f_2$, we have $ \| f_\varepsilon \ast (f_1-f_2) \|_1 \leq \|f_1-f_2\|_1$. 
Finally, since, for $Z\sim N(0,\,1)$,
$$ \mathbb E_{\mu_X}[|X|]\leq \mathbb E_{\mu_H}[|U|]+\sigma \mathbb E[|Z|], \quad \text{where $\mu_X$ has density $\mu_H\ast\phi_\sigma$, with $U\sim \mu_H$,}$$
and $\Pi ( \mathbb E_{\mu_H}[|U|]=\infty ) = 0 $, condition \eqref{eq:mfin} holds true.
\end{proof}


\smallskip   

A rate of the order $O(n^{-\beta/(2\beta+1)})$, up to a logarithmic factor,
is achieved for estimating mixtures of Linnik densities with index $1<\beta\leq 2$, 
if a kernel mixture prior on the mixing density is constructed using a Gaussian kernel
with an inverse-gamma bandwidth and a Dirichlet process on the mixing distribution.
The result is new in Bayesian density estimation and is a preliminary step for 
$L^1$-Wasserstein density deconvolution.

\begin{thm}\label{thm:32}
Under the assumptions of Theorem \ref{thm:31},
there exist constants $K$ large enough and $\nu>0$ so that
\begin{equation*}\label{eq:69}
\Pi_n(\mu_X:\,W_1(\mu_X,\,\mu_{0X})>Kn^{-1/(2\beta+1)}(\log n)^\nu\mid\Data)
\rightarrow0\mbox{ in $P_{0Y}^n$-probability.}
\end{equation*}
\end{thm}

\begin{proof}
Under Assumption \ref{ass:basemeasure1}, we have $\int_{\mathbb R} u h_0(u)\,\di u<\infty$, which implies that, almost surely, 
the random variable $U$, with distribution $\mu_H\sim\mathscr D_{H_0}$, has finite first moment. 
Then, also $Y=(U+Z)+\varepsilon$, where $Z\sim N(0,\,1)$ and $\varepsilon$ is a standard Linnik random 
variable with index $1<\beta\leq2$, has finite first moment, see, \emph{e.g.}, Proposition 4.3.18 in
\cite{Kotz2001}, p. 267. Apply Theorem \ref{thm:31} and Theorem \ref{thm:22} to conclude.
\end{proof}

\subsection*{Extension to general scale mixtures of Laplace densities}
The method of proof used in Theorems \ref{thm:31} and \ref{thm:32}
for Linnik error densities with index $1<\beta<2$ uses the representation of a Linnik density 
as a scale mixture of Laplace densities of the form 
$$ f_\varepsilon(u)=\int_{\mathbb R^+} v e^{-v|u|}f_V(v;\,\beta)\,\di v, \quad u\neq0,\quad  f_V(v;\,\beta)\propto\frac{ v^{\beta-1} }{ 1 + v^{2\beta} +2 v^{\beta} \cos(\pi\beta/2)},\quad v>0.$$
We can extend the results to more general scale mixtures of Laplace densities
$f_\varepsilon(u)=\int_{\mathbb R^+} v e^{-v|u|}f_V(v;\,\beta)\,\di v$, $u\in\mathbb R$,
with mixing density $f_V(\cdot;\,\beta)$ which, for some constant $\beta>0$, satisfies the following conditions:
\vspace*{-0.1cm}
\begin{itemize}
\item[\textnormal{i)}] the random variable $V$ has finite expectation $\mathbb E[V]=\int_{\mathbb R^+} vf_V(v;\beta)\,\di v<\infty$ and
$$\int_0^1\frac{v^2}{\mathbb E[V\1_{\{V<v\}}]} f_V(v;\,\beta)\,\di v<\infty,$$
\item[\textnormal{ii)}] for every $t\in\mathbb R$,
$$\int_{\mathbb R^+}\frac{v^2}{v^2+t^2} f_V(v;\,\beta)\,\di v\lesssim \frac{1}{1+|t|^\beta},$$
\item[\textnormal{iii)}] there exist constants $C_\beta>0$ and $p_\beta>2$, possibly depending on $\beta$,
such that
$$f_\varepsilon(u)\sim C_\beta|u|^{-p_\beta},\quad |u|\rightarrow \infty,$$
\item[\textnormal{iv)}] the Fourier transform $\hat f_\varepsilon$ of $f_\varepsilon$ verifies condition \eqref{eq:deriv} of Assumption 
\ref{ass:identifiability+error}.
\end{itemize}
Conditions ii)--iv) are the essential ones. 
Condition i) is a technical requirement that may possibly be due to an artifact of the proof. 
Granted the assumptions of Theorem \ref{thm:31} on $f_{0X}$ and the prior $\Pi$, 
conditions i)--iii) are sufficient for extending the theorem to
cover error densities with \emph{heavy} tails. If, in addition, condition iv) holds true, 
then also Theorem \ref{thm:32} goes through. These results are not separately stated here. 

\medskip

For another application of Theorem \ref{thm:32}, consider the case where $Z$ has a monotone non-increasing density $f_Z$ on $\mathbb R^+$. 
Following \cite{williamson:56}, it is known that $Z$ is a scale mixture of uniforms so that $Y = \log Z = X + \varepsilon$, 
where $\varepsilon \sim \mathrm{Exp}(1)$ and is independent of $X$. If the distribution of $X$ satisfies the assumptions of Theorem \ref{thm:32},
then the posterior distribution of $\mu_X$ concentrates around the true $\mu_{0X}$ at rate $n^{-1/3}$, up to a $(\log n)$-term,
in the $L^1$-Wasserstein metric. Writing $f_Z(z) = \int_{\mathbb R^+} (\1_{[z,\,\infty)}/v)\, \di F(v)$ and 
$f_{0Z}(z) = \int_{\mathbb R^+}(\1_{[z,\,\infty)}/v)\, \di F_0(v)$, we have
$$ \Pi_n( d(F,\,F_0) > M n^{-1/3} (\log n)^\nu \mid Z^{(n)} ) = o_{\mathbb P}(1), \,\,\,\mbox{with } d(F,\,F_0):= \int_{\mathbb R^+}\frac{| F(v) - F_0(v) | }{v}\, \di v.$$

\subsection{Sobolev-regularity adaptive case}\label{subsec:adapt}
In this section we focus on the case in which the sampling density is a mixture of 
Laplace densities with a Sobolev regular mixing distribution. 
The Laplace case is worked out in detail, whereas the analogous results for Linnik mixtures
are briefly mentioned to avoid duplications.

We still consider $\Pi$ the prior law induced by $\mathscr D_{H_0}\otimes \Pi_\sigma$ 
on the parameter $(\mu_H,\,\sigma)$ of $f_Y=f_\varepsilon\ast(\mu_H\ast \phi_\sigma)$, 
for a standard Laplace error density $f_\varepsilon$, and $\Pi_n(\cdot\mid Y^{(n)})$ the posterior distribution 
based on i.i.d. observations $Y_1,\,\ldots,\,Y_n$ drawn from $f_Y$, with $f_Y\sim \Pi$. 
The sampling density 
$f_{0Y}=f_\varepsilon\ast f_{0X}$ is assumed to be a Laplace mixture, with mixing density $f_{0X}$ 
satisfying the following conditions. 
\begin{ass}\label{ass:sobolevcond}
\emph
{There exists $\alpha>0$ such that
\begin{equation*} \label{Sobolev:cond}
\forall\,b=\mp1/2,\quad\int_{\mathbb R} |t|^{2\alpha}|\widehat{e^{b\cdot}f_{0X}}(t)|^2\,\di t < \infty.
\end{equation*}
}
\end{ass}
\begin{ass}\label{ass:smoothf0}
\emph{There exist $0<\upsilon\leq1$, $L_0\in L^1(\mathbb R)$ and $R\geq m/\upsilon$, for the smallest integer $m\geq(\alpha+2)$, such that $f_{0X}$
satisfies
\begin{equation}\label{eq:holderf0}
|f_{0X}(x+\zeta)-f_{0X}(x)|\leq L_0(x)|\zeta|^{\upsilon}, \quad\mbox{for every }\zeta,\,x\in\mathbb R,
\end{equation}
and
\begin{equation}\label{eq:ratiofo}
\int_{\mathbb R}e^{|x|/2}f_{0X}(x)\pt{\frac{L_0(x)}{f_{0X}(x)}}^R\,\di x<\infty.
\end{equation}
}
\end{ass}
Assumption \ref{ass:sobolevcond} requires that, for every $b=\mp1/2$, the function $e^{b\cdot}f_{0X}$ belongs to a Sobolev space of order $\alpha$, while Assumption \ref{ass:smoothf0} requires that $f_{0X}$ is locally H\"older smooth of order $\upsilon$, with envelope function 
$L_0$ satisfying the integrability condition \eqref{eq:ratiofo}. 
The model $f_Y=f_\varepsilon\ast(\mu_H\ast \phi_\sigma)$ acts as an approximation scheme 
for automatic posterior rate adaptation to the global regularity of $f_{0Y}$, without any knowledge of 
the regularity of $f_{0X}$ being used in the prior specification. 
We show that a rate-adaptive estimation procedure for Laplace mixtures 
can be obtained if the prior is properly constructed, for instance, as a mixture of Laplace-normal convolutions, with
an inverse gamma bandwidth and a Dirichlet process on the mixing distribution.
\begin{thm}\label{thm:4}
Let $Y_1,\,\ldots,\,Y_n$ be i.i.d. observations from
$f_{0Y}:=f_\varepsilon\ast f_{0X}$, where $f_\varepsilon$ is the density of a 
standard Laplace distribution and $f_{0X}$ 
satisfies Assumption \ref{ass:twicwtailcond}, Assumption \ref{ass:sobolevcond} for $\alpha>0$ and Assumption \ref{ass:smoothf0}. 
Let $\Pi$ be the prior law induced by $\mathscr D_{H_0} \otimes \Pi_\sigma$, where 
${H_0}$ verifies Assumption \ref{ass:basemeasure1} and $\Pi_\sigma$ verifies
Assumption \ref{ass:priorscale1} for $\gamma=1$. 
There then exist constants $D'$ large enough and $\kappa'>0$ so that
\[
\Pi_n(\mu_Y:\,\|f_Y-f_{0Y}\|_1>D' n^{-(\alpha+2)/(2\alpha+5)}(\log n)^{\kappa'}
\mid \Data)\rightarrow0\mbox{ in $P_{0Y}^n$-probability}.
\]
\end{thm}
The key step of the proof of Theorem \ref{thm:4}, which is deferred to Section \ref{sec:proofs}, is 
a novel approximation result reported in Lemma \ref{lem:contapprox}.

\smallskip  

By combining Theorem \ref{thm:4} with Theorem \ref{thm:22} and Lemma \ref{lem:biasmixgaus}, we obtain 
adaptive posterior contraction rates for $L^1$-Wasserstein density deconvolution of Laplace mixtures.
\begin{thm}\label{thm:5}
Granted the assumptions of Theorem \ref{thm:4},
there exist constants $K'$ large enough and $\nu'>0$ so that
\begin{equation*}\label{eq:70}
\Pi(\mu_Y:\,W_1(\mu_X,\,\mu_{0X})>K'n^{-(\alpha+1)/(2\alpha+5)}(\log n)^{\nu'}\mid\Data)
\rightarrow0\mbox{ in $P_{0Y}^n$-probability.}
\end{equation*}
\end{thm}

Using the fact that every Linnik density with index $0<\beta<2$ admits a representation as a scale mixture of Laplace densities, together with 
the arguments used for proving Theorems \ref{thm:31} and \ref{thm:32}, the results of Theorems \ref{thm:4} and \ref{thm:5}
can be extended to mixtures of Linnik densities with index $1<\beta<2$ to obtain
adaptive rates $n^{-(\alpha+\beta)/[2(\alpha+\beta)+1]}$ and $n^{-(\alpha+1)/[2(\alpha+\beta)+1]}$, respectively, up to logarithmic factors.
Further extension to general scale mixtures of Laplace densities can be pursued using the conditions listed in Section \ref{sec:nonadaptive}.


\section{Proofs}\label{sec:proofs}
\subsection{Proof of Theorem \ref{thm:22}} \label{sec:prth22}
By the conditions in \eqref{con1}, Theorem 2.1 of \cite{ghosal2000}, p. 503, implies that, for sufficiently large $\bar M$,
$$\mathbb E_{0Y}^n[\Pi_n(\mu_X:\,d(f_Y,\,f_{0Y})>\bar M\tilde\epsilon_n \mid Y^{(n)})]\rightarrow 0$$
and, as a by-product, that $\mathbb E_{0Y}^n[\Pi_n(\mathscr P_n^c\mid Y^{(n)})]\rightarrow0$.
Next, the case where Assumption \ref{ass:smoothXXX} holds true is treated in details. 
By conditions \eqref{eq:mfin} and \eqref{eq:ass1}, Theorem \ref{theo:1} implies that 
$W_1(\mu_X,\,\mu_{0X})\lesssim h^{\alpha+1}+W_1(\mu_Y,\,\mu_{0Y})+h^{-(\beta-1)_+}\,d(f_Y,\,f_{0Y})\log(1/h)$
with constants that are uniform over $\mathscr P_n$. Minimizing with respect to $h$, we get
$$W_1(\mu_X,\,\mu_{0X})\lesssim W_1(\mu_Y,\,\mu_{0Y})+[d(f_Y,\,f_{0Y})\log(1/h)]^{(\alpha+1)/[\alpha+(\beta\vee1)]},$$
where $\alpha+(\beta\vee1)=\alpha+1+(\beta-1)_+$. 
For sufficiently large $M>0$, defined the set $\mathscr S_n:=\{\mu_X:\,W_1(\mu_Y,\,\mu_{0Y})\leq M \tilde \epsilon_n\}$,
under the third listed condition in \eqref{con1}, by Lemma \ref{lem:1}, we have 
$\mathbb E_{0Y}^n[\Pi_n(\mathscr S_n^c\mid Y^{(n)})]\rightarrow 0$.
Define $$\omega_n:=\sup_{\mu_X\in (\mathscr P_n\cap \mathscr S_n):\,d(f_Y,\,f_{0Y})\leq  \bar M\tilde\epsilon_n }W_1(\mu_X,\,\mu_{0X}),$$
we have
$\omega_n\lesssim \tilde\epsilon_n+\epsilon_{n,\alpha}\lesssim \epsilon_{n,\alpha}$. 
Reasoning as in Theorem 2.1 of \cite{knapik2018}, p. 2094, we have
$\mathbb E_{0Y}^n[\Pi_n(\mu_X:\,W_1(\mu_X,\,\mu_{0X})>K_\alpha\epsilon_{n,\alpha}\mid Y^{(n)})]
\leq \mathbb E_{0Y}^n[\Pi_n(\mu_X:\,W_1(\mu_X,\,\mu_{0X})>\omega_n\mid Y^{(n)})]\leq 
\mathbb E_{0Y}^n[\Pi_n(\mu_X\in(\mathscr P_n\cap \mathscr S_n):\,d(f_Y,\,f_{0Y})>\bar M\tilde\epsilon_n\mid Y^{(n)})]
+\mathbb E_{0Y}^n[\Pi_n(\mathscr P_n^c\mid Y^{(n)})]+\mathbb E_{0Y}^n[\Pi_n(\mathscr S_n^c\mid Y^{(n)})]\rightarrow0$ 
and the convergence in \eqref{eq:convsmooth} follows. 
The case where no assumption on $\mu_{0X}$ is required, except for the first moment condition, 
follows similarly from the inversion inequality with $h$ in lieu of $h^{\alpha+1}$.\qed

\subsection{Proof of Theorem \ref{theo:1}} \label{sec:rth1}
Because $\mu_X$ and $\mu_{0X}$ have finite first moments, $W_1(\mu_X,\,\mu_{0X})<\infty$,
see, \emph{e.g.}, \cite{villani2009}, p. 94.
Since  $\mathbb E[|\varepsilon|]<\infty$, also $\mu_Y$ and $\mu_{0Y}$ have finite first moments. 
Thus, $W_1(\mu_Y,\,\mu_{0Y})<\infty$. 

In what follows, we distinguish the case where no 
assumption, except for the first moment existence condition, 
is imposed on $\mu_{0X}$, from the case where Assumption \ref{ass:smoothXXX} is in force. 

\vspace{0.2cm}  

\noindent 
$\bullet$ \emph{Case 1: No smoothness assumption on $\mu_{0X}$}\\
The kernel $K$ can be taken to be a symmetric \emph{probability density}
with finite first moment $\int_{\mathbb R}|z|K(z)\,\di z<\infty$.
For $h>0$, let $F_{K_h}(z):=\int_{-\infty}^z K_h(s)\,\di s$, $z\in\mathbb R$, be the 
cumulative distribution function of the rescaled kernel $K_h$ and let $\mu_{K_h}$ be its probability law. 
Let $Z$ be a random variable with distribution $\mu_{K_h}$ and $X_1$ a random variable, independent of $Z$, with distribution $\mu_{0X}$.  
Then, $Z+X_1\sim \mu_{K_h}\ast\mu_{0X}$ and, since $\mathbb E[|Z|]=h\int_{\mathbb R}|z|K(z)\,\di z<\infty$, 
we have $\mathbb E[|Z+X_1|]\leq \mathbb E[|Z|]+\mathbb E[|X_1|]<\infty$ so that $W_1(\mu_{K_h}\ast\mu_{0X},\,\mu_{0X})<\infty$. 
By definition of the $L^1$-Wasserstein distance, $W_1(\mu_{K_h}\ast\mu_{0X},\,\mu_{0X}) \leq \mathbb E[|(Z+X_1)-X_1|]=\mathbb E[|Z|]\lesssim h$. 
Analogously, $W_1(\mu_X,\,\mu_{K_h}\ast\mu_X)\lesssim h$. Then,
\begin{eqnarray}\label{eq:distriang}
W_1(\mu_X,\,\mu_{0X})&\leq&
W_1(\mu_X,\,\mu_{K_h}\ast\mu_X) + W_1(\mu_{K_h}\ast\mu_X,\,\mu_{K_h}\ast\mu_{0X})+  W_1(\mu_{K_h}\ast\mu_{0X},\,\mu_{0X})\nonumber\\
&\lesssim& h + W_1(\mu_{K_h}\ast\mu_X,\,\mu_{K_h}\ast\mu_{0X}).
\end{eqnarray}

\noindent
$\bullet$ \emph{Case 2: Smoothness Assumption \ref{ass:smoothXXX} on $\mu_{0X}$ holds true}\\
If condition (i) of Assumption \ref{ass:smoothXXX} is in force, then
$K$ is taken to be an $(\lfloor\alpha\rfloor+1)$-order kernel satisfying, in addition, $\int_{\mathbb R}|z|^{\alpha+1}|K(z)|\,\di z<\infty$. For every $0\leq\alpha\leq 2$, the kernel can be a density and
the same reasoning as for Case 1, leading to an analogue of inequality \eqref{eq:distriang}, could be adopted. 
However, since, for $\alpha>2$, no non-negative function can be a higher-order kernel, 
some adjustments are required. We therefore lay out arguments that do not rely on the fact that $K$ is a density.
If, instead, condition (ii) of Assumption \ref{ass:smoothXXX} is in force, then $K$ is taken to be a superkernel such that  
$K\in L^1(\mathbb R)\cap L^2(\mathbb R)$ is symmetric, with 
$\int_{\mathbb R}|z||K(z)|\,\di z<\infty$ and 
$\hat K\equiv 1$ on $[-1,\,1]$, while $\hat K\equiv 0$ on $[-2,\,2]^c$.

Let $F_X$ and $F_{0X}$ denote the distribution functions of $\mu_X$ and $\mu_{0X}$, respectively.
Recalling that $W_1(\mu_X,\,\mu_{0X})=\|F_X-F_{0X}\|_1<\infty$, by the triangular inequality,
\begin{eqnarray*}
W_1(\mu_X,\,\mu_{0X})
&\leq&
\underbrace{\|F_X-F_{K_h}\ast \mu_X\|_1}_{=\|b_{F_X}\|_1}
+ \underbrace{\|F_{K_h}\ast(\mu_X-\mu_{0X})\|_1}_{=:T} + 
\underbrace{\|F_{K_h}\ast \mu_{0X}-F_{0X}\|_1}_{=\|b_{F_{0X}}\|_1},
\end{eqnarray*}
where all terms on the right-hand side of the last line are finite. 
In fact, for suitable $\gamma\in(0,\,1)$, we have $\|F_{K_h}\ast \mu_{0X}-F_{0X}\|_1=\|F_{0X}\ast K_h-F_{0X}\|_1=
\int_{\mathbb R}|z||K_h(z)|\int_{\mathbb R}f_{0X}(x-\gamma z)\,\di x\,\di z= \int_{\mathbb R}|z||K_h(z)|\,\di z=
h\int_{\mathbb R}|s||K(s)|\,\di s<\infty$. By the same reasoning, also
$\|F_X-F_{K_h}\ast \mu_X\|_1=\|F_X\ast K_h-F_X\|_1<\infty$.
Besides, by Young’s inequality $\|f\ast g\|_p\leq \|f\|_1\|g\|_p$
valid for every $f\in L^1(\mathbb{R})$ and $g\in L^p(\mathbb{R})$, $1\leq p\leq\infty$, we have
$\|F_{K_h}\ast(\mu_X-\mu_{0X})\|_1=\|(F_X-F_{0X})\ast K_h\|_1
\leq\|F_X-F_{0X}\|_1 \|K\|_1=\|K\|_1\times W_1(\mu_X,\,\mu_{0X})<\infty$.
From condition \eqref{eq:ass1} on $\|b_{F_X}\|_1$ 
and Assumption \ref{ass:smoothXXX} on $f_{0X}$, jointly with 
Lemma \ref{lem:der} or Lemma \ref{lem:sob} on $\|b_{F_{0X}}\|_1$, we obtain that 
\begin{eqnarray}\label{eq:Wi}
W_1(\mu_X,\,\mu_{0X}) \lesssim h^{\alpha+1}+T, \quad T = W_1(\mu_{K_h}\ast\mu_X,\,\mu_{K_h}\ast\mu_{0X}).
\end{eqnarray}

We show that, for sufficiently small $h$, inequality \eqref{eq:t2} 
holds true in Cases 1 and  2. Let $\chi:\,\mathbb R\rightarrow \mathbb R$ be a symmetric, continuously differentiable function,
equal to $1$ on $[-1,\,1]$ and to $0$ outside $[-2,\,2]$. 
For example, one such function can be defined as $\chi(t)=e\exp{\{-1/[1-(|t|-1)^2]\}}$ for $|t|\in(1,\,2)$.
For a unified treatment of the cases where $K$ is 
an $(\lfloor\alpha\rfloor+1)$-order kernel or a supersmooth kernel, we set
\[
\tilde h:=\begin{cases}
h, & \text{ if $K$ is an $(\lfloor\alpha\rfloor+1)$-order kernel},\\[-4pt]
h/2, &\text{ if $K$ is a supersmooth kernel}.
\end{cases}
\]
Analogously, we set $\tilde 1$ equal to $1$ if $K$ is an $(\lfloor\alpha\rfloor+1)$-order kernel 
and equal to $1/2$ if $K$ is a supersmooth kernel. 
For $h>0$, define
$$w_{1,h}(t):=\hat K(ht)\chi(t)r_{\varepsilon}(t),\,\,\,\, t\in\mathbb R,\qquad\mbox{ and }\quad w_{2,h}(t):=\hat K(ht)[1-\chi(t)]r_{\varepsilon}(t), \,\,\,\, t\in\mathbb R.$$
Note that $K\in L^1(\mathbb R)$ implies that $\hat K$ is well-defined and
$\|\hat K\|_\infty:=\sup_{t\in\mathbb R}|\hat K(t)|\leq \|K\|_1<\infty$. 
Since $\hat K(\tilde 1\cdot)\in C_b([-1,\,1])$, we have $\hat K\in L^1(\mathbb{R})$ and 
$K(\cdot)=(2\pi)^{-1}\int_{\mathbb R}\exp{(-\imath t \cdot)}\hat K(t)\,\di t$. 
If $h<1/2$, the function $w_{1,h}$ is equal to $0$ outside $[-2,\,2]$, while $ w_{2,h}$ is equal to $0$ on $[-1,\,1]$ and
outside $[-1/\tilde h,\,1/\tilde h]$. Thus, $w_{j,h}\in L^1(\mathbb R)$, for $j=1,\,2$. 
In fact, by inequality \eqref{eq:deriv} with $l=0$, 
we have $
\|w_{1,h}\|_1
\lesssim
\int_{|t|\leq2}|\hat K(ht)||\chi(t)| (1+|t|)^\beta\,\di t <\infty$ because 
the integrand is in $C_b([-2,\,2])$.
Analogously, $\|w_{2,h}\|_1\lesssim 
\int_{1<|t|\leq1/\tilde h}|\hat K(ht)||1-\chi(t)| (1+|t|)^\beta\,\di t
<\infty$. Then, the inverse Fourier transform of $w_{j,h}$,
\begin{equation*}\label{eq:45}
z\mapsto K_{j,h}(z):=\frac{1}{2\pi}\int_{\mathbb R}
\exp{(-\imath tz)}w_{j,h}(t)\,\di t,
\end{equation*}
is well defined and, as a consequence of Lemma \ref{lem:bigO} and Lemma \ref{lem:K2} in Section \ref{app:1},
is in $L^1(\mathbb R)$, for $j=1,\,2$.
We can then define the mappings
$$z\mapsto F_{j,h}(z):=\int_{-\infty}^z K_{j,h}(s)\,\di s, \quad j=1,\,2.$$ 
Using the decomposition $\hat K(ht)r_\varepsilon(t)=w_{1,h}(t)+w_{2,h}(t)$, $t\in\mathbb R$,
\[
\begin{split}
[K_h\ast (\mu_X-\mu_{0X})](y)
&=\frac{1}{2\pi}
\int_{\mathbb R}\exp{(-\imath ty)}\hat K(ht)r_\varepsilon(t) (\hat f_Y-\hat f_{0Y})(t)\,\di t\\
&=\frac{1}{2\pi}
\int_{\mathbb R}\exp{(-\imath ty)}[w_{1,h}(t)+w_{2,h}(t)](\hat f_Y-\hat f_{0Y})(t)\,\di t \\
&=[K_{1,h}\ast  (\mu_Y-\mu_{0Y})](y)+[K_{2,h}\ast  (\mu_Y-\mu_{0Y})](y),\quad y\in\mathbb R,
\end{split}
\]
and we can write 
$[F_{K_h}\ast (\mu_X-\mu_{0X})]=
[K_{1,h}\ast (F_Y-F_{0Y})]+[K_{2,h}\ast (F_Y-F_{0Y})]=
[K_{1,h}\ast (F_Y-F_{0Y})]+[F_{2,h}\ast (\mu_Y-\mu_{0Y})]$, 
where $F_Y$ and $F_{0Y}$ denote the distribution functions of $\mu_Y$ and $\mu_{0Y}$, respectively.
Then,
\[\begin{split}
T&\leq
\|K_{1,h}\ast(F_Y-F_{0Y})\|_1 + \|K_{2,h}\ast(F_Y-F_{0Y})\|_1\\
&=\|K_{1,h}\ast(F_Y-F_{0Y})\|_1 + \|F_{2,h}\ast(\mu_Y-\mu_{0Y})\|_1
=:T_1+T_2.
\end{split}
\]

\begin{itemize}
\item[$\bullet$] {\emph{Study of the term $T_1$}}\\[5pt]
By Young’s inequality,
\[\begin{split}
T_1:=\|K_{1,h}\ast(F_Y-F_{0Y})\|_1 &\leq \|K_{1,h}\|_1 \times \|F_Y-F_{0Y}\|_1\\
&=\|K_{1,h}\|_1 \times W_1(\mu_Y,\, \mu_{0Y})\lesssim W_1(\mu_Y,\, \mu_{0Y})
\end{split}
\]
because $\|K_{1,h}\|_1=O(1)$ in virtue of Lemma \ref{lem:bigO}.\\[1pt]
\item[$\bullet$] {\emph{Study of the term $T_2$}}\\[5pt]
For every $\beta>0$, by Young’s inequality and Lemma \ref{lem:K2}  in Appendix \ref{app:1},
\[\begin{split}
T_2:=\|K_{2,h}\ast(F_Y-F_{0Y})\|_1&\leq\|K_{2,h}\|_1
\times W_1(\mu_Y,\, \mu_{0Y})\\
&\lesssim h^{-(\beta-1/2)_+}|\log h|^{1+\1_{\{\beta=1/2\}}/2}\,W_1(\mu_Y,\, \mu_{0Y}).
\end{split}
\]
Analogously, for every $\beta>0$, using Lemma \ref{lem:F2} in Appendix \ref{app:1}, 
\[
T_2:=\|F_{2,h}\ast(\mu_Y-\mu_{0Y})\|_1\leq\|F_{2,h}\|_1
\times \|f_Y-f_{0Y}\|_1
\lesssim h^{-(\beta-1)_+}|\log h|\,\|f_Y-f_{0Y}\|_1.
\]
\end{itemize}
Combining the bounds on $T_1$ and $T_2$, we obtain inequality \eqref{eq:t2},
which, together with \eqref{eq:Wi}, proves the inversion inequality.
The statement for the Hellinger distance follows from LeCam's 
inequality $\|f_Y-f_{0Y}\|_1\leq 2 d_{\mathrm H}(f_Y,\,f_{0Y})$, see \cite{lecam1973}, p. 40. 
The proof is thus complete. \qed


\subsection{Proof of Theorem \ref{thm:4}} \label{sec:rth3}
To obtain a suitable approximation of $f_{0Y}=f_\varepsilon\ast f_{0X}$, 
we need an auxiliary result concerning the approximation of a smooth function by convolutions.
For $h>0$, let 
$$H(x):=\frac{1}{2\pi}\hat \tau(x) e^{-(h x)^2/2}, \quad x\in\mathbb R,$$
where $|\hat\tau(x)|\leq (16^2/15)e^{-\sqrt{|x|/15}}$, $x\in\mathbb R$, is the Fourier transform of 
$\tau:\,\mathbb R\rightarrow [0,\,1]$ defined in Theorem 25 of \cite{bourgain}, p. 29, such that
$$\tau (u) = \left\{ \begin{array}{ll}
1,&\text{if }  |u| < 1, \\[2pt]
0,&\text{if }  |u| > 17/15.
 \end{array} \right. $$
The function $\tau$ is such that
\begin{equation}\label{eq:poldec}
\mbox{for any $i\in\{0\}\cup\mathbb N$,}\quad|\hat\tau^{(i)}(x)|=O(|x|^{-\nu})\quad \mbox{for large $x\in\mathbb R$ and every $\nu>0$.}
\end{equation}
Given $m\in\mathbb N$, $b=\mp1/2$, $\delta,\,\sigma>0$ and a function $f:\,\mathbb R\rightarrow\mathbb R$, we define the transform 
$$T_{m,b,\sigma} f:=f+\sum_{k=1}^{m-1}\frac{ (-1)^k \sigma^{2k}}{ 2^k k!}
\sum_{j=0}^{2k} \binom{2k}{j}(-b)^{2k-j}[f\ast (e^{-b\cdot}D^jH_\delta)],$$
where $H_\delta (\cdot):=\delta^{-1}H(\cdot/\delta)$.
We have $M_{0X}(b)<\infty$. For 
\begin{equation}\label{eq:barh}
\bar h_{0,b}:=\frac{e^{b\cdot}f_{0X}}{M_{0X}(b)}
\end{equation} and $\gamma:=-(1-e^{-\sigma^2/8})$, let
\begin{equation}\label{hm}
h_{m,b,\sigma}:=\frac{1}{\gamma} \sum_{k=1}^{m-1}\frac{ (-1)^k \sigma^{2k}}{ 2^k k!}
\sum_{j=0}^{2k} \binom{2k}{j}(-b)^{2k-j}(\bar h_{0,b}\ast D^jH_\delta).
\end{equation}
 Note that $[M_{0X}(b)]^{-1}e^{b\cdot}(T_{m,b,\sigma}f_{0X})=\bar h_{0,b}+\gamma  h_{m,b,\sigma}$.

\begin{lem}\label{lem:contapprox}
Let $f_\varepsilon$ be the standard Laplace density.
Let $f_{0X}$ be a density such that $(e^{|\cdot|/2}f_{0X})\in L^1(\mathbb R)\cap L^2(\mathbb R)$ and which
satisfies Assumption \ref{ass:sobolevcond} for $\alpha>0$.
Then, for $m\geq (\alpha+2)$ and $\sigma>0$ small enough,
\begin{equation}\label{bound:tildehm12}
\sum_{b=\mp 1/2}\|e^{b\cdot}\{f_\varepsilon\ast [\phi_\sigma\ast (T_{m,b,\sigma}f_{0X})-f_{0X}]\}\|_2^2\lesssim\sigma^{2(\alpha + 2)}
\end{equation}
and
\begin{equation}\label{eq:bound}
\forall\, b=\mp1/2,\quad \int_{\mathbb R} h_{m,b,\sigma}(x)\,\di x = 1 + O(\sigma^{m-2}).
\end{equation}
\end{lem}

\begin{proof}
We first obtain an equivalent expression for the $L^2$-norm in \eqref{bound:tildehm12}.
Denoting the Fourier transform operator by $\mathcal F$, for any $f\in L^1(\mathbb R)$, we have $\mathcal F\{f\}:=\hat f$.
Recall that, for $f_\varepsilon(u)=e^{-|u|}/2$, $u\in\mathbb R$, we have $\mathcal F\{e^{b\cdot}f_\varepsilon\}(t)=[1/\varrho_b(t)]$, where
$\varrho_b(t):=[1-\psi_b^2(t)]$ and $\psi_b(t):=-(\imath t+b)$, $t\in\mathbb R$. Noting that 
$[M_{0X}(b)]^{-1}\mathcal F\{e^{b\cdot}(T_{m,b,\sigma}f_{0X})\}=\mathcal F\{\bar h_{0,b}\}+\gamma \mathcal F\{ h_{m,b,\sigma}\}$, 
where $M_{0X}(b)<\infty$ for $b=\mp1/2$ by the assumption $(e^{|\cdot|/2}f_{0X})\in L^1(\mathbb R)$,
we get that
\begin{eqnarray}\label{eq:equivalent}
\Delta_0&:=&\sum_{b=\mp 1/2}\|e^{b\cdot}\{f_\varepsilon\ast [\phi_\sigma\ast (T_{m,b,\sigma}f_{0X})-f_{0X}]\}\|_2^2\nonumber\\
&=&M^2_{0X}(b)
\sum_{b=\mp 1/2}\|(e^{b\cdot}f_\varepsilon)\ast
[(e^{b\cdot}\phi_\sigma)\ast \{[M_{0X}(b)]^{-1}e^{b\cdot}(T_{m,b,\sigma}f_{0X})\}-\bar h_{0,b}]\}\|_2^2\nonumber\\
&=&
\frac{M^2_{0X}(b)}{2\pi}
\sum_{b=\mp 1/2}\Big\|\frac{e^{\sigma^2 \psi_b^2/2}}{\varrho_b}
[(1-e^{-\sigma^2 \psi_b^2/2})\mathcal F\{\bar h_{0,b}\}+\gamma\mathcal F\{  h_{m,b,\sigma}\}]\Big\|_2^2.
\end{eqnarray}
Some facts are highlighted for later use. 
For every $\delta>0$, the function $\mathcal F\{H\}(\delta \cdot)$ is well defined because 
$\|H\|_1=(2\pi)^{-1}\|\hat \tau e^{-(h \cdot)^2/2}\|_1 <\infty$. Besides,
as $0\leq \tau\leq 1$,
\begin{equation}\label{eq:bilaplace1}
|\mathcal F\{H\}(\delta t)|=
|(\tau \ast \phi_{h})(-\delta t)|  \leq \|\phi_{h}(-\delta t-\cdot)\|_1=\|\phi_{-\delta t,h}\|_1=1,\quad t\in\mathbb R.
\end{equation}
Let $Z$ be a standard normal random variable. 
For constants $0<c_\delta,\,c_h<1$, take $\delta:=c_\delta\sigma$ and $h:=c_h|\log \sigma|^{-1/2}$.
Fix $u_0$ such that $0<c_\delta<u_0<1$. Then, for $\omega>0$ and 
$c_h$ such that $(1-u_0)\geq c_h\sqrt{2\omega}$, we have $\forall |t|\leq u_0/\delta,$
\begin{align}\label{eq:complement1}
|1-\mathcal F\{H\}(\delta t)|
&\leq2\int_{|u|\geq1}\phi_{-\delta t,h}(u)\,\di u  \leq2P(|Z|\geq (1-\delta |t|)/h) \nonumber\\
&\leq 2P(|Z|\geq(1-u_0) |\log \sigma|^{1/2}/c_h)\lesssim \sigma^{\omega}
\end{align}
as soon as $\sigma$ is small enough.
For every $j\in\{0\}\cup\mathbb N$, we have
$\mathcal F\{D^j H_\delta\}(t)=(-\imath t)^j \mathcal F \{H\} (\delta t)$, $t\in\mathbb R$. Then, recalling that $\psi_b(t)=-(\imath t +b)$,
\[
\forall\,b=\mp1/2,\quad
\mathcal F\{h_{m,b,\sigma}\} (t)
=\frac{1}{\gamma}\mathcal F\{\bar h_{0,b}\}(t)
\mathcal F\{H\}(\delta t) \sum_{k=1}^{m-1}
\frac{ (-1)^k [\sigma\psi_b(t)]^{2k}}{ 2^k k!},\quad t\in\mathbb R.
\] 
Decomposing $\mathcal F\{\bar h_{0,b}\}(t)$ by means of $\mathcal F\{H\}(\delta t)$ 
and $[1-\mathcal F\{H\}(\delta t)]$, 
the numerator of the integrand of $\Delta_0$ in \eqref{eq:equivalent} can be bounded above by
\[\begin{split}
\mathcal J_b^2(t)&:=|e^{\sigma^2 \psi_b^2(t) /2}|^2\big|(1-e^{-\sigma^2 \psi_b^2(t) /2})\mathcal F\{\bar h_{0,b}\}(t)
\mathcal F\{H\}(\delta t)  + \gamma\mathcal F\{h_{m,b,\sigma}\}(t)\big|^2\\ 
&\hspace*{5cm} + |e^{\sigma^2 \psi_b^2(t) /2}-1|^2\, |\mathcal F\{\bar h_{0,b}\}(t)|^2 \abs{1 - \mathcal F\{H\} (\delta t)}^2, \quad t\in\mathbb R.
\end{split}\]
Set $\Delta_{01}:= \sum_{b=\mp1/2}\int_{\delta|t|\leq u_0}[\mathcal J_b^2(t)/|\varrho_b(t)|^2]\,\di t$ and 
$\Delta_{02}:=\sum_{b=\mp1/2}\int_{\delta|t|> u_0}[\mathcal J_b^2(t)/|\varrho_b(t)|^2]\,\di t$,
we have that $\Delta_0\lesssim\Delta_{01}+\Delta_{02}$. We now prove that $\Delta_{0j}\lesssim \sigma^{2(\alpha+2)}$, for $j=1,\,2$. 
Taking into account that $|e^{\sigma^2\psi_b^2(t)/2}|^2=e^{-\sigma^2(t^2-b^2)}=e^{-\sigma^2(t^2-1/4)}$, 
for $\omega\geq m\geq(\alpha+2)$ and $\sigma>0$ small enough, by Lemma \ref{lem:diseg}, 
relationships \eqref{eq:bilaplace1} and \eqref{eq:complement1}, we have
\[\begin{split}
\Delta_{01}&\lesssim  \sum_{b=\mp1/2}\int_{\delta|t|\leq 
u_0} \frac{1}{|\varrho_b(t)|^2}
\big(
[\sigma^2(t^2+1/4)]^{m}\\&\hspace*{5cm}+\sigma^{2\omega}\min\{4,\,\sigma^4(t^2+1/4)^2/4\}
\big)
|\mathcal F\{\bar h_{0,b}\}(t)|^2\,\di t\\
&\lesssim \sigma^{2(\alpha+2)} \sum_{b=\mp 1/2}\int_{\delta|t|\leq u_0} (|t|^{2\alpha}+1)|\widehat{e^{b\cdot}f_{0X}}(t)|^2\,\di t
\lesssim \sigma^{2(\alpha+2)},
\end{split}
\]
because $\mathcal F\{\bar h_{0,b}\}(t)=[M_{0X}(b)]^{-1}\widehat{e^{b\cdot}f_{0X}}(t)$, $t\in\mathbb R$, and $\int_{\mathbb R}(|t|^{2\alpha}\vee 1)|\widehat{e^{b\cdot}f_{0X}}(t)|^2\,\di t<\infty$
by Assumption \ref{ass:smoothXXX} and $\|e^{b\cdot}f_{0X}\|_1\leq\|e^{|\cdot|/2}f_{0X}\|_1<\infty$. 
Analogously, for $\sigma|t|>(u_0/c_\delta)>1$,
\[\begin{split}
\Delta_{02}&\lesssim \sum_{b=\mp 1/2} \int_{\delta|t|>u_0} \frac{1}{ |\varrho_{b}(t)|^2}
\bigg(
\bigg|e^{\sigma^2 \psi_b^2(t) /2}\sum_{k=0}^{m-1}\frac{(-1)^k[\sigma\psi_b(t)]^{2k}}{2^k k!}-1\bigg|^2\\[-4pt]
&\hspace*{6cm}+|e^{\sigma^2 \psi_b^2(t) /2}-1|^2\bigg)
|\mathcal F\{\bar h_{0,b}\}(t)|^2\,\di t\\
&\lesssim\sum_{b=\mp 1/2}\int_{\delta|t|>u_0} \frac{1}{|\varrho_{b}(t)|^2}
\{e^{-\sigma^2(t^2-1/4)/2} [(\sigma|t|)^{2m}+1]\\[-4pt]
&\hspace*{6cm}+\min\{2,\,\sigma^2(t^2+1/4)/2\}\}^2 |\mathcal F\{\bar h_{0,b}\}(t)|^2 \,\di t \\ 
&\lesssim \sigma^{2(\alpha+2)}\sum_{b=\mp1/2}\int_{\delta|t|>u_0}\frac{t^4}{|\varrho_{b}(t)|^2}
[e^{-\sigma^2t^2}(\sigma|t|)^{4m-2(\alpha+2)}+1]
|t|^{2\alpha} |\mathcal F\{\bar h_{0,b}\}(t)|^2  \,\di t\\
&\lesssim \sigma^{2(\alpha+2)}
\sum_{b=\mp1/2}\int_{\delta|t|>u_0}
|t|^{2\alpha}|\widehat{e^{b\cdot}f_{0X}}(t)|^2 \,\di t \lesssim \sigma^{2(\alpha+2)}.
\end{split}\]
We now prove relationship \eqref{eq:bound}.
Since $\mathcal F\{\bar h_{0,b}\}(0)=1$,
$(1-e^{-\sigma^2/8})/\gamma=-1$ and $\sigma^2/8\leq e^{\sigma^2/8}|\gamma|$,
from previous computations for the term $\Delta_{01}$ we have
$$\frac{\sigma^2}{8}|\mathcal F\{h_{m,b,\sigma}\}(0)-1|\leq
e^{\sigma^2/8}|\gamma|\Big|\mathcal F \{h_{m,b,\sigma}\}(0)+\frac{(1-e^{-\sigma^2/8})}{\gamma}\Big|
\lesssim \mathcal J_b(0)\lesssim \sigma^m,$$
whence $\int_{\mathbb R}h_{m,b,\sigma}(x)\,\di x=\mathcal F \{h_{m,b,\sigma}\}(0)=1+O(\sigma^{m-2})$.
The proof is thus complete.
\end{proof}

\begin{proof}[Proof of Theorem \ref{thm:4}]
The entropy condition (2.8) and the small ball prior probability estimate
(2.10) of Theorem 2.1 in \cite{ghosal:2001}, p. 1239, are 
satisfied for $\bar\epsilon_n=n^{-(\alpha+2)/(2\alpha+5)}(\log n)^{\tau'}$ and 
$\tilde\epsilon_n=n^{-(\alpha+2)/(2\alpha+5)}(\log n)^{\tau_0'}$, with exponents $\tau'>\tau_0'>1$.
Then, the posterior rate is $\epsilon_n:=(\bar\epsilon_n\vee\tilde\epsilon_n)=\bar\epsilon_n=n^{-(\alpha+2)/(2\alpha+5)}(\log n)^{\tau'}$. 
For the details of the entropy and remaining mass conditions, 
see, \emph{e.g.}, Theorem 5 of \cite{ghosal:shen:tokdar}, p. 631, while
for the small-ball prior probability estimate apply Lemma \ref{lem:deltabound1} together with a modified version of Lemma \ref{lem:KL} 
with $(\alpha+2)$ in place of $\beta=2$.
\end{proof}


\section{Final remarks}
In this paper we have studied, from a Bayesian perspective, the problem of deconvolution, which, as described in section \ref{sec:intro}, is of primary importance in many applications. If some optimal results have been obtained using kernel estimators, no optimal Bayesian procedure has been studied so far. 
One of the key results of this work is an inversion/transportation inequality relating the $L^1$-Wasserstein distance between $\mu_X$ and $\mu_{0X}$ to either the total variation or the Wasserstein distance between the mixed distributions $\mu_Y$ and $\mu_{0Y}$. This inequality is derived under mild conditions on the noise density $f_\varepsilon$ and covers noise regularity $\beta$ ranging in $[1/2,\,\infty)$ showing the existence of different  \emph{r\'egimes}. The upper bound is expressed in terms of the total variation distance between the mixed distributions $\mu_Y$ and $\mu_{0Y}$ for $\beta\geq1$ and in terms of the $L^1$-Wasserstein distance for $\beta \geq 1/2$. The version expressed in terms of the $L^1$-Wasserstein metric would lead to the minimax posterior contraction rate for $W_1(\mu_X,\,\mu_{0X})$ if a posterior convergence rate for $W_1(\mu_Y,\,\mu_{0Y})$ of the order $O(1/\sqrt{n})$ were preliminarily obtained. However, deriving such a rate for nonparametric mixture models is a non trivial task and we are not aware of any such results in the literature. 

The inversion inequality can be used in various contexts nevertheless, in particular it can be used  outside the Bayesian inference or it can be used as a first step  to obtain Bernstein-von Mises type results on linear functionals of $\mu_Y$ or $\mu_X$ .

We have used the inversion inequality to derive a general result on $L^1$-Wasserstein posterior convergence rates for the mixing distribution. 
We have then studied the case where the error density is a Linnik distribution with true mixing density having exponentially decaying tails. 
In the special case of a Laplace error, we have further studied adaptive $L^1$-Wasserstein estimation of $\mu_X$ when $f_{0X}$ is Sobolev regular. 
Adaptation is obtained using as a prior distribution on the mixing density $f_X$ a mixture of Gaussian densities. 
This result is derived by constructing an approximation of $f_{0Y}$ 
of the form $f_\varepsilon\ast \phi_\sigma \ast f_{1}$ with a suitable function $f_1$. This approximation is significantly more involved than
the one deviced in \cite{kruijer:rousseau:vdv:10}, which would not lead to the correct error rate in the present context and is of interest in itself. 


\section*{Acknowledgements}
The authors gratefully acknowledge financial support from the Institut Henri Poincaré (IHP), Sorbonne Université (Paris), 
within the RIP program on \vir{Bayesian Wasserstein deconvolution} that has taken place in 2019 at the IHP-Centre Émile Borel, 
where part of this work was written. Catia Scricciolo has also been partially supported by Università di Verona. 
She wishes to dedicate this work to her mother and sister Emilia, with deep love and immense gratitude.

The project leading to this work has received funding from the European Research Council
(ERC) under the European Union’s Horizon 2020 research and innovation programme
(grant agreement No 834175).

\appendix

\section{Lemmas \ref{lem:K2} and \ref{lem:F2}  in the proof of Theorem \ref{theo:1}} \label{app:1}

For a unified treatment of the cases (i) and (ii) of Assumption \ref{ass:smoothXXX}, we set
\[\vspace*{-0.1cm}
\tilde h:=\begin{cases}
h, & \text{ if $K$ is an $(\lfloor\alpha\rfloor+1)$-order kernel},\\[-5pt]
h/2, &\text{ if $K$ is a supersmooth kernel}.
\end{cases}
\]
Analogously, $\tilde 1$ is equal to 1 if $K$ is an $(\lfloor\alpha\rfloor+1)$-order kernel and to $1/2$
if $K$ is a supersmooth kernel. Also, $K_{2,h}(\cdot):=(2\pi)^{-1}\int_{\mathbb R}\exp{(-\imath t\cdot)}w_{2,h}(t)\,\di t$ 
is the inverse Fourier transform of $w_{2,h}(t):=\hat K(ht)[1-\chi(t)]r_\varepsilon(t)$, $t\in\mathbb R$, 
while $F_{2,h}(z):=\int_{-\infty}^z K_{2,h}(s)\,\di s$, $z\in\mathbb R$,
is the \vir{distribution function} of $K_{2,h}$.

\begin{lem}\label{lem:K2}
If $\mu_\varepsilon\in\mathscr P_0$ satisfies Assumption \ref{ass:identifiability+error} for $\beta>0$, then, for $h>0$ small enough,
\begin{equation}\label{eq:alt_ineq2}
\|K_{2,h}\|_1
\lesssim h^{-(\beta-1/2)_+}|\log h|^{1+\1_{\{\beta=1/2\}}/2}.
\end{equation}
\end{lem}

\begin{proof}
Consider the decomposition
\[
\|K_{2,h}\|_1=
\bigg(\int_{|z|\leq h^{3/2}}+\int_{h^{3/2}<|z|\leq 1}+
\int_{|z|>1}\bigg)
|K_{2,h}(z)|\,\di z=:K_2^{(1)}+K_2^{(2)}+K_2^{(3)}.
\]
By condition \eqref{eq:deriv} with $l=0$ and any $\beta>0$, 
since $\hat K(\tilde 1\cdot)\in C_b([-1,\,1])$, we have
\[\begin{split}
K_2^{(1)} &:=
\int_{|z|\leq h^{3/2}}
\abs{K_{2,h}(z)}\di z < 2h^{3/2}
\int_{1<|t|\leq 1/\tilde h}
|\hat K(ht)| |1-\chi(t)||r_\varepsilon (t)|\,\di t \\
&\lesssim h^{3/2}
\int_{1<|t|\leq 1/\tilde h}
|\hat K(ht)| |t|^\beta\,\di t\\
&\lesssim   h^{-(\beta-1/2)}
\int_{1<|t|\leq 1/\tilde h}h|\hat K(ht)|\,\di t \lesssim h^{-(\beta-1/2)_+}\|\hat K\|_1.
\end{split}
\]  
Thus, $K_2^{(1)}=O(h^{-(\beta-1/2)_+})$. To bound $K_2^{(2)}$ and $K_2^{(3)}$, we apply
identity \eqref{eq:identity21} to $K_{2,h}$. Note that
\begin{equation*}\label{eq:K2der}
\mbox{$z\neq0$,} \quad  K_{2,h}(z)=\frac{1}{2\pi(\imath z)}\int_{\mathbb R}\exp{(-\imath tz)}w^{(1)}_{2,h}(t)\,\di t,
\end{equation*} 
provided that
$w^{(1)}_{2,h}(t)=  h\hat K^{(1)}(ht) [ 1 - \chi(t)] r_\varepsilon(t)- \hat K(ht)\{\chi^{(1)}(t)
r_\varepsilon(t)-[1 - \chi(t)]r^{(1)}_\varepsilon(t)\}$, $t\in\mathbb R$, is in $L^1(\mathbb R)$. 
We show that $w^{(1)}_{2,h}\in L^2(\mathbb R)$. Then, by the same arguments, 
also $w^{(1)}_{2,h}\in L^1(\mathbb R)$. We have
 \[
\begin{split}
\|w^{(1)}_{2,h}\|_2^2
&\lesssim\int_{\mathbb R}\big|h\hat K^{(1)}(ht) [1 - \chi(t)]r_\varepsilon(t)\big|^2\,\di t \\
&\hspace*{1.5cm}+ \int_{\mathbb R}
 |\hat K(ht)|^2\big|\chi^{(1)}(t)r_\varepsilon(t)-[1 - \chi(t)]
r^{(1)}_\varepsilon(t)\big|^2\,\di t\\&=:J_1^2+J_2^2.
\end{split}
 \]
For every $\beta>0$, since also $\hat K^{(1)}(\tilde 1\cdot)\in C_b([-1,\,1])$, 
\[
J_1^2\lesssim\int_{1<|t|\leq 1/\tilde h}h^2|\hat K^{(1)}(ht) |^2|1-\chi(t)|^2 (1 + |t|)^{2\beta}\,\di t \lesssim h^{-2(\beta-1/2)_+}\|\hat K^{(1)}\|_2^2
 \lesssim h^{-2(\beta-1/2)_+}
\]
so that $J_1=O(h^{-(\beta-1/2)_+})$. For every $h\leq1/2$,
\[\begin{split}
J_2^2&\lesssim
\int_{1<|t|<2}
 |\hat K(ht)|^2|\chi^{(1)}(t)|^2\abs{r_\varepsilon(t)}^2\,\di t
+\int_{1<|t|\leq 1/\tilde h}
 |\hat K(ht)|^2
|1-\chi(t)|^2
|r^{(1)}_\varepsilon(t)|^2\,\di t\\
&\lesssim \|\chi^{(1)}r_\varepsilon\|_2^2+\|(1-\chi)r^{(1)}_\varepsilon\|_2^2,
\end{split}
\]
where $\|\chi^{(1)}r_\varepsilon\|_2=O(1)$ and 
$
\|(1-\chi)r^{(1)}_\varepsilon\|_2^2
\lesssim   \int_{1<|t|\leq 1/\tilde h}|t|^{2(\beta-1)}\,\di t 
\lesssim 
h^{-2(\beta-1/2)_+}|\log h|^{\1_{\{\beta=1/2\}}}$. 
So, $J_2=O( h^{-(\beta-1/2)_+}|\log h|^{\1_{\{\beta=1/2\}}/2})$. It follows that
$\|w^{(1)}_{2,h}\|_2=O( h^{-(\beta-1/2)_+}|\log h|^{\1_{\{\beta=1/2\}}/2})$.
By the same arguments, also $\|w_{2,h}^{(1)}\|_1=
O( h^{-(\beta-1/2)_+}|\log h|^{\1_{\{\beta=1/2\}}/2})$. Then, 
\[\begin{split}
K_2^{(2)} &:=\int_{h^{3/2}<|z|\leq 1}
|K_{2,h}(z)|\,\di z \lesssim
\pt{\int_{h^{3/2}<|z|\leq 1}\frac{1}{|z|}\,\di z}\|w^{(1)}_{2,h}\|_1\\
&\lesssim h^{-(\beta-1/2)_+}|\log h|^{1+\1_{\{\beta=1/2\}}/2}.
\end{split}\]
By the Cauchy–Schwarz inequality,
\[\begin{split}
K_2^{(3)} &:=\int_{|z|>1}
\abs{K_{2,h}(z)}\,\di z =
 \int_{|z|>1}\frac{1}{|z|}\abs{\frac{1}{2\pi}\int_{\mathbb R}\exp{(-\imath t z)}
w_{2,h}^{(1)}(t)\,\di t}\,\di z\\
&\lesssim 
\pt{\int_{|z|>1 }\frac{1}{|z|^2}\,\di z}^{1/2}\|w_{2,h}^{(1)}\|_2
\lesssim \|w_{2,h}^{(1)}\|_2\lesssim h^{-(\beta-1/2)_+}|\log h|^{\1_{\{\beta=1/2\}}/2}.
\end{split}
\]
Inequality \eqref{eq:alt_ineq2} follows by combining the bounds on $K_2^{(1)}$,  $K_2^{(2)}$ and $K_2^{(3)}$.
\end{proof}

The following lemma is analogous to Lemma \ref{lem:K2} and gives the order,
in terms of the kernel bandwidth, of the $L^1$-norm of the \vir{distribution function} $F_{2,h}$ of $K_{2,h}$.

\begin{lem}\label{lem:F2}
If $\mu_\varepsilon\in\mathscr P_0$ satisfies Assumption \ref{ass:identifiability+error} for $\beta>0$, then, for $h>0$ small enough,
\begin{equation}\label{eq:alt_ineq}
\|F_{2,h}\|_1
\lesssim h^{-(\beta-1)_+}|\log h|.
\end{equation}
\end{lem}

\begin{proof}
By the same arguments used for the function $G_{2,h}$ in \cite{dedecker2015}, pp. 251--252, for $h<1$, 
we have
\begin{equation*}\label{eq:F2}
F_{2,h}(z)=
\frac{1}{2\pi}
\int_{\mathbb R}\exp{(-\imath t z)}\frac{w_{2,h}(t)}{(-\imath t)}\,\di t,\quad z\in\mathbb R,
\end{equation*}
where $t\mapsto [w_{2,h}(t)/t]$ is in $L^1(\mathbb R)$ because $\int_{1<|t|\leq1/\tilde h}[|w_{2,h}(t)|/|t|]\,\di t
\lesssim \|w_{2,h}\|_1<\infty$. 
Consider the integral decomposition
\[
\|F_{2,h}\|_1=
\bigg(\int_{|z|\leq h}+\int_{h<|z|\leq 1}+\int_{|z|>1}\bigg)
|F_{2,h}(z)|\,\di z=:F_2^{(1)}+F_2^{(2)}+F_2^{(3)}.
\]
By condition \eqref{eq:deriv} with $l=0$ and any $\beta>0$, since
$\hat K(\tilde 1\cdot)\in C_b([-1,\,1])$, we have
\[\begin{split}
F_2^{(1)} &:=
\int_{|z|\leq h}
\abs{F_{2,h}(z)}\,\di z<2h
\int_{1<|t|\leq 1/\tilde h}
|\hat K(ht)| |1-\chi(t)|\frac{|r_\varepsilon (t)|}{|t|}\,\di t \\&\lesssim h
\int_{1<|t|\leq 1/\tilde h}
|\hat K(ht)| |1-\chi(t)|\frac{(1+|t|)^\beta}{|t|}\,\di t\\
&\lesssim h^{-(\beta-1)} \times 
\begin{cases}
\displaystyle
\int_{1<|t|\leq 1/\tilde h}\dfrac{1}{|t|}\,\di t, & \text{ if } 0<\beta<1,\\
 \|\hat K\|_1, & \text{ if } \beta\geq1,
\end{cases}\\
&\lesssim h^{-(\beta-1)} \times
\begin{cases}
|\log h|, & \text{ if } 0<\beta<1,\\[-3pt]
\|\hat K\|_1, & \text{ if } \beta\geq1.
\end{cases}
\end{split}
\]  
Therefore, $F_2^{(1)}=O( h^{-(\beta-1)}|\log h|^{(1-\beta)_+})$. 
To bound $F_2^{(2)}$ and $F_2^{(3)}$, preliminarily note that, by applying identity \eqref{eq:identity21} to $F_{2,h}$, 
we have
\begin{equation}\label{eq:F2der}
\mbox{for $z\neq0$,} \quad  F_{2,h}(z)=\frac{1}{2\pi(\imath z)}\int_{\mathbb R}\exp{(-\imath tz)}\,\frac{\di }{\di t }\pt{\frac{w_{2,h}(t)}{-\imath t}}\,\di t,
\end{equation} 
where
$$  
\frac{\di }{\di t }\pt{\frac{w_{2,h}(t)}{t}}=  h\hat K^{(1)}(ht) [ 1 - \chi(t)] \frac{r_\varepsilon(t)}{t}- \hat K(ht) \pg{\chi^{(1)}(t)\frac{ r_\varepsilon(t)  }{t}-[1 - \chi(t)]
\frac{\di }{\di t }\pt{\frac{r_\varepsilon(t)}{t}}}.$$
Then, 
 \[
\begin{split}
I^2&:=\int_{\mathbb R}\bigg|\frac{\di}{\di t}\pt{\frac{w_{2,h}(t)}{t}}\bigg|^2\di t\\&
\lesssim\int_{\mathbb R}\abs{h\hat K^{(1)}(ht) [1 - \chi(t)]\frac{r_\varepsilon(t)}{t}}^2\di t \\ 
&\hspace*{1.5cm}+ \int_{\mathbb R}
 |\hat K(ht)|^2\abs{\chi^{(1)}(t)\frac{ r_\varepsilon(t)  }{t}-[1 - \chi(t)]
\frac{\di }{\di t }\pt{\frac{r_\varepsilon(t)}{t}}}^2\di t\\
&=:I_1^2+I_2^2.
\end{split}
 \]
For every $\beta>0$, since also $\hat K^{(1)}(\tilde 1\cdot)\in C_b([-1,\,1])$,
$$
I_1^2\lesssim\int_{1<|t|\leq 1/\tilde h}h^2 |\hat K^{(1)}(ht) |^2|1-\chi(t)|^2 \frac{(1 + |t|)^{2\beta}}{ |t|^2}\,\di t \lesssim h^{-2(\beta-1)_+}.
$$
For every $h\leq1/2$,
\begin{eqnarray}\label{eq:I2}
\qquad\qquad
I_2^2&:=&\int_{\mathbb R}
 |\hat K(ht)|^2\abs{\chi^{(1)}(t)\frac{r_\varepsilon(t)}{t}-[1 - \chi(t)]
\frac{\di}{\di t}\pt{\frac{r_\varepsilon(t)}{t}}}^2\di t \nonumber\\
&\lesssim&
\int_{1<|t|< 2}
 |\hat K(ht)|^2|\chi^{(1)}(t)|^2\abs{\frac{r_\varepsilon(t)}{t}}^2\di t \nonumber\\
&&\hspace*{1.5cm}+\int_{1<|t|\leq 1/\tilde h}
 |\hat K(ht)|^2
|1-\chi(t)|^2
\abs{\frac{\di }{\di t }\pt{\frac{r_\varepsilon(t)}{t}}}^2\di t \nonumber\\
&
\lesssim& \int_{1<|t|<2}\frac{(1+|t|)^{2\beta}  }{t^2}\,\di t
+
\int_{1<|t|\leq 1/\tilde h}
\abs{
\frac{\di }{\di t }\pt{\frac{r_\varepsilon(t)}{t}}}^2\di t, 
\end{eqnarray}
where the first integral in \eqref{eq:I2} is $O(1)$. Since
\[\begin{split}
\abs{
\frac{\di}{\di t}\pt{\frac{r_\varepsilon(t)}{t}}
}=\frac{|tr^{(1)}_\varepsilon(t)-r_\varepsilon(t)|}{|t|^2}\leq \pt{\frac{|r^{(1)}_\varepsilon(t)|}{|t|}+\frac{|r_\varepsilon(t)|}{t^2}}
\lesssim   \frac{(1+|t|)^{\beta-1}}{|t|}+\frac{(1+|t|)^{\beta}}{t^2},\quad t\in\mathbb R,
\end{split}
\]
the second integral in \eqref{eq:I2} can be bounded above as follows:
\[\begin{split}
\int_{1<|t|\leq 1/\tilde h}
\abs{
\frac{\di }{\di t }\pt{\frac{r_\varepsilon(t)}{t}}}^2\di t &
\lesssim   \int_{1<|t|\leq 1/\tilde h}|t|^{2(\beta-2)}\,\di t\\
& \lesssim
\begin{cases}
1, & \text{ if } 0<\beta<1,\\
\displaystyle \int_{1<|t|\leq 1/\tilde h}|t|^{2(\beta-3/2)}\,\di t, & \text{ if } \beta\geq1,
\end{cases}\\
&\lesssim \begin{cases}
1, & \text{ if } 0<\beta<1,\\[-3pt]
h^{-2(\beta-1)}, & \text{ if } \beta\geq1,
\end{cases} 
\end{split}\]
and $I_2\lesssim h^{-(\beta-1)_+}$.
By the same arguments used to bound $I$, we also have that
\[
\int_{\mathbb R}\abs{ \frac{\di}{\di t}\pt{\frac{w_{2,h}(t)}{t}}}\,\di t=
O(h^{-(\beta-1)_+}),
\]
so that, by virtue of identity \eqref{eq:F2der},
\[\begin{split}
F_2^{(2)}:=\int_{h<|z|\leq 1}
|F_{2,h}(z)|\,\di z&\lesssim
\pt{\int_{h<|z|\leq 1}\frac{1}{|z|}\,\di z}\pt{\int_{\mathbb R}\abs{ \frac{\di}{\di t}\pt{\frac{w_{2,h}(t)}{-\imath t}}}
\,\di t}\\
&\lesssim h^{-(\beta-1)_+}|\log h|.
\end{split}
\]
Applying Cauchy–Schwarz inequality to the expression of $F_{2,h}$ in \eqref{eq:F2der}, we have
\[\begin{split}
F_2^{(3)}:=\int_{|z|>1}
\abs{F_{2,h}(z)}\di z&\lesssim 
\pt{\int_{|z|>1 }\frac{1}{|z|^2}\,\di z}^{1/2} \pt{\int_{\mathbb R}\abs{ \frac{\di}{\di t}\pt{\frac{w_{2,h}(t)}{-\imath t}}}^2
\di t}^{1/2}\\
&\lesssim \pt{\int_{\mathbb R}\abs{ \frac{\di}{\di t}\pt{\frac{w_{2,h}(t)}{-\imath t}}}^2\di t}^{1/2}\lesssim h^{-(\beta-1)_+}.
\end{split}
\]
Inequality \eqref{eq:alt_ineq} follows by combining the bounds on $F_2^{(1)}$,  $F_2^{(2)}$ and $F_2^{(3)}$.
\end{proof}



\begin{frontmatter}
\title{Wasserstein convergence in Bayesian deconvolution models: Supplementary Material}
\runtitle{Wasserstein convergence in Bayesian deconvolution models: Supplement}

\begin{aug}
\author{\fnms{Judith} \snm{Rousseau} and \fnms{Catia} \snm{Scricciolo}}
\address{University of Oxford and University of Verona}
\end{aug}

\begin{abstract}
This supplement contains auxiliary results for proving
Theorem \ref{thm:22}, Theorem \ref{theo:1}, Theorem \ref{thm:31} and Theorem \ref{thm:4}
of the main document \cite{rousseau:scricciolo:main}, as well as the proof of Proposition \ref{eq:minimax1}. 
For the sake of simplicity, Sections, Equations, Lemmas, etc., of the supplementary material are denoted 
with the prefix S to distinguish them from Sections, Equations, Lemmas, etc., of the main text \cite{rousseau:scricciolo:main}.
\end{abstract}

\end{frontmatter}


\section{Lemma \ref{lem:1} and Proof of Proposition \ref{eq:minimax1}} \label{sec:section3}
In this section, we provide auxiliary results associated to Section \ref{sec:general}. 

\subsection{Lemma for Theorem \ref{thm:22} on posterior contraction rates for $L^1$-Wasserstein deconvolution} \label{sec:lem:gene}
We state a sufficient condition on the prior concentration rate on Kullback-Leibler
type neighborhoods of the \vir{true} probability measure for posterior $L^1$-Wasserstein 
contraction around the true distribution to take place at least as fast. The assertion is in 
the same spirit of Lemma 1 in \cite{scricciolo19}, pp. 123--125, providing sufficient conditions 
on the sampling distribution and the prior concentration rate for the posterior law
to contract at a nearly $\sqrt{n}$-rate on Kolmogorov neighborhoods.
The underlying idea is to construct tests with type I error probability controlled using a 
Dvoretzky-Kiefer-Wolfowitz (DKW) type inequality for the $L^1$-Wasserstein metric.

\begin{lem}\label{lem:1}
Let $\mu_{0Y}\in\mathscr P_0$ have 
finite first moment $\mathbb E_{0Y}[Y]<\infty$.
Let $\Pi$ be a prior law on probability measures $\mu_Y\in\mathscr P_0$, each one having
finite first moment $\mathbb E_{\mu_Y}[Y]<\infty$.
If, for a constant $C>0$ and a positive sequence $\tilde\epsilon_n\rightarrow 0$ such that $n\tilde\epsilon_n^2\rightarrow\infty$, 
we have
\begin{equation}\label{eq:33}
\Pi(B_{\mathrm{KL}}(P_{0Y};\,\tilde\epsilon_n^2))\gtrsim
\exp{(-Cn\tilde\epsilon_n^2)},
\end{equation}
then, for $M:=\xi(1-\theta)^{-1}(C+1/2)^{1/2}$, with $\theta\in(0,\,1)$ and $\xi>1$,
\begin{equation}\label{eq:66}
\Pi_n(\mu_Y:\,W_1(\mu_Y,\,\mu_{0Y})>M
\tilde\epsilon_n\mid\Data)\rightarrow0\mbox{ in $P_{0Y}^n$-probability.}
\end{equation}
\end{lem}

\begin{proof}
Because $\mu_Y$ and $\mu_{0Y}$ have finite first moments, $W_1(\mu_Y,\,\mu_{0Y})<\infty$,
see, \emph{e.g.}, \cite{villani2009}, p. 94. The posterior probability of
$A_n^c:=\{\mu_Y:\,W_1(\mu_Y,\,\mu_{0Y})>M\tilde\epsilon_n\}$  
is given by
\[\Pi_n(A_n^c\mid\Data)
=\frac{\int_{A_n^c}\prod_{i=1}^nf_Y(Y_i)\, \Pi_n(\di \mu_Y)}{\int_{\mathscr P_0}\prod_{i=1}^nf_Y(Y_i)\, \Pi_n(\di \mu_Y)}.
\]
We construct (a sequence of) tests $(\Psi_n)_{n\in\mathbb N}$ for the hypothesis
$H_0:\,P=P_{0Y}\equiv \mu_{0Y}$ \emph{versus} $H_1:\,P=P_Y\equiv \mu_Y,\, \mu_Y\in A_n^c$,
where $\Psi_n\equiv\Psi_n(\Data;\,P_{0Y}):\,\mathbb R^n\rightarrow\{0,\,1\}$
is the indicator function of the rejection region of $H_0$, such that
\[
\mathbb E_{0Y}^n[\Psi_n] = o(1) ,
\quad
\mbox{ and}\,\,\,\sup_{\mu_Y\in A_n^c}
\mathbb E_{\mu_Y}^n[1-\Psi_n]\leq 2\exp{(-2(M-K)^2n\tilde\epsilon_n^2)} \, \mbox{ for $n$ large enough},
\]
with $K:=\theta M$. Define $\Psi_n:=\1_{R_n}$, with rejection region $R_n:=\{\data:\,W_1(\mu_n,\,\mu_{0Y})>K\tilde\epsilon_n\}$, 
where $\mu_n$ is the empirical probability measure of the sample $\Data$. 
Since $\mu_{0Y}$ has \emph{continuous} distribution function $F_{0Y}$ by assumption, we have that
$P_{0Y}^n(\data:\,W_1(\mu_n,\,\mu_{0Y})>t)\leq 2e^{-2nt^2}$, see, \emph{e.g.}, \cite{boissard2011}, pp. 2304--2305,
which is based on the DKW inequality
\cite{dvoretzky1956}, with the tight universal constant in \cite{massart1990}.
Then, $\mathbb E_{0Y}^n[\Psi_n]=P_{0Y}^n(R_n)\leq 2\exp{(-2K^2n\tilde\epsilon_n^2)}$.
Therefore,
\begin{equation}\label{eq:56}
\mathbb E_{0Y}^n[\Pi_n(A_n^c\mid\Data)]
\leq 2\exp{(-2K^2n\tilde\epsilon_n^2)}+
\mathbb E_{0Y}^n[\Pi_n(A_n^c\mid\Data)(1-\Psi_n)].
\end{equation}
To control the second term in \eqref{eq:56}, defined
$$D_n:=\pg{
\data:\,\int_{\mathscr P_0}\prod_{i=1}^n \frac{f_Y}{f_{0Y}}(y_i)\,\Pi_n(\di \mu_Y)\leq \Pi(B_{\mathrm{KL}}(P_{0Y};\,\tilde\epsilon_n^2))
\exp{(-(C+1)n\tilde\epsilon_n^2)}
},$$
we consider the decomposition 
$\mathbb E_{0Y}^n[\Pi_n(A_n^c\mid\Data)(1-\Psi_n)(\1_{D_n}+\1_{D_n^c})]$.
From Lemma 8.1 of \cite{ghosal:ghosh:vdv:00}, p. 524, we have that
$P_{0Y}^n(D_n)\leq (C^2n\tilde\epsilon_n^2)^{-1}$. It follows that 
\begin{equation}\label{eq900}
\mathbb E_{0Y}^n[\Pi_n(A_n^c\mid\Data)(1-\Psi_n)\1_{D_n}]\leq
P_{0Y}^n(D_n)\leq(C^2n\tilde\epsilon_n^2)^{-1}.
\end{equation} 
By assumption \eqref{eq:33} and Fubini's theorem, 
\begin{equation}\label{eq:9}
\mathbb E_{0Y}^n[\Pi_n(A_n^c\mid\Data)(1-\Psi_n)\1_{D_n^c}
]\lesssim\exp{((2C+1)n\tilde\epsilon_n^2)}\int_{A_n^c}
\mathbb E_{\mu_Y}^n[1-\Psi_n]\,\Pi(\di \mu_Y).
\end{equation}
Next, we give an exponential upper bound on $\sup_{\mu_Y\in A_n^c}
\mathbb E_{\mu_Y}^n[1-\Psi_n]$.
Over the acceptance region $R_n^c$,
by the triangular inequality, for every $\mu_Y\in A_n^c$, we have
$
M\tilde\varepsilon_n<W_1(\mu_Y,\,\mu_{0Y})\leq
W_1(\mu_Y,\,\mu_n)+
W_1(\mu_n,\,\mu_{0Y})\leq
W_1(\mu_Y,\,\mu_n)+K\tilde\epsilon_n,
$  
which implies that
$W_1(\mu_Y,\,\mu_n)>(M-K)\tilde\epsilon_n$.
Since $F_Y$ is continuous, by the DKW type inequality for the $L^1$-Wasserstein metric, we have
\[\begin{split}
&\sup_{\mu_Y\in A_n^c}
\mathbb E_{\mu_Y}^n[1-\Psi_n]\\[-10pt]
&\qquad\qquad\quad\leq
\sup_{\mu_Y\in A_n^c}
P_Y^n(\data:\,W_1(\mu_n,\,\mu_Y)>(M-K)\tilde\epsilon_n)\leq 2\exp{(-2(M-K)^2n\tilde\epsilon_n^2)}.
\end{split}
\]
Combining the preceding inequality with \eqref{eq:9}, we have
\begin{equation}\label{eq300}
\mathbb E_{0Y}^n[\Pi_n(A_n^c\mid\Data)(1-\Psi_n)\1_{D_n^c}
]\lesssim \exp{(-2[(M-K)^2-(C+1/2)]n\tilde\epsilon_n^2)},
\end{equation}
where the right-hand side of \eqref{eq300} converges to zero because
$(M-K)=(1-\theta)M>(C+1/2)^{1/2}$. The convergence in 
\eqref{eq:66} follows by combining the bounds in 
\eqref{eq:56}, \eqref{eq900} and \eqref{eq300}. 
\end{proof}

\begin{rmk}
\emph{
If condition \eqref{eq:33} is replaced by
\begin{equation*}\label{eq:34}
\Pi(N_{\mathrm{KL}}(P_{0Y};\,\tilde\epsilon_n^2))\gtrsim
\exp{(-Cn\tilde\epsilon_n^2)},
\end{equation*}
where $N_{\mathrm{KL}}(P_{0Y};\,\tilde\epsilon_n^2):=
\{P_Y\in\mathscr P_0:\,\mathrm{KL}(P_{0Y};\,P_Y)\leq\tilde\varepsilon_n^2\}$
is a Kullback-Leibler neighborhood of $P_{0Y}$, 
then, by Lemma 6.26 of \cite{bookgvdv}, p. 145, with $P_{0Y}^n$-probability at least equal to $(1-L_n^{-1})$, 
for a sequence $L_n\rightarrow\infty$ such that $nL_n\tilde\epsilon_n^2\rightarrow\infty$, we have 
\begin{equation}\label{eq:LBKL}
\int_{\mathscr P_0}\prod_{i=1}^n \frac{f_Y}{f_{0Y}}(Y_i)\,\Pi(\di \mu_Y)\gtrsim
\exp{(-(C+2L_n)n\tilde\epsilon_n^2)}.
\end{equation}
Following the proof of Lemma \ref{lem:1} and applying the 
lower bound in \eqref{eq:LBKL}, the convergence statement in \eqref{eq:66} continues to
hold true with $M\tilde\varepsilon_n$ replaced by $M_n\tilde\varepsilon_n$, for 
$M_n:=\xi(1-\theta)^{-1}(C/2+L_n)^{1/2}$, with $\theta\in(0,\,1)$ and $\xi>1$. 
Taking $L_n$ to be a slowly varying sequence,
Kullback-Leibler type neighborhoods can be replaced by Kullback-Leibler neighborhoods 
at the cost of an additional log-factor in the rate, which thus becomes equal to $L_n^{1/2}\tilde \varepsilon_n$. 
An extra log-factor is common in convergence rates of posterior distributions. It arises from the
\vir{testing-prior mass} approach, which we also adopt here. We refer the reader to \cite{hoffmann&rousseau} and 
\cite{gao&zhou} for a better understanding of this phenomenon.}
\end{rmk}

\begin{rmk}
\emph{
Lemma \ref{lem:1} is stated for the case where $P_{0Y}$ is a convolution 
of probability measures, but the result holds true for every 
$P_{0Y}$ with continuous distribution function and finite first moment.}
\end{rmk}


\subsection{Proof of Proposition \ref{eq:minimax1} on $L^2$-minimax rates over logarithmic Sobolev classes} \label{sec:proof:minimax}
We first show that $\psi_{n,\gamma}^2$ is a lower bound on the $L^2$-minimax risk. 
Let $\delta>1$. For every $\gamma,\,L>0$, the inclusion $\mathscr F_\gamma^{\mathcal{S}}(L)\subseteq \mathscr F_{\gamma,\delta}^{\mathcal{LS}}(L)$ holds true.
Thus, for a constant $c>0$, possibly depending on $\gamma$ and $L$,
\[c  \psi_{n,\gamma}^2 \leq \inf_{\hat f_n}\sup_{f\in \mathscr F_\gamma^{\mathcal{S}}(L)}\mathbb E^n_f[ \|\hat f_n-f\|_2^2]\leq \inf_{\hat f_n}\sup_{f\in \mathscr F_{\gamma,\delta}^{\mathcal{LS}}(L)} \mathbb E^n_f [\|\hat f_n-f\|_2^2],\]
see Theorem 2.9 in \cite{Tsybakov:2009}, p. 107, for the left-hand side inequality. 
To show that $\psi_{n,\gamma}^2(\log n)^{\delta/(2\gamma+1)}$ is an upper bound on the $L^2$-minimax risk, 
we restrict the class of estimators to kernel density estimators 
$f_K(\cdot):=n^{-1}\sum_{i=1}^n K(Y_i-\cdot)$,
with a symmetric kernel $K\in L^2(\mathbb R)$.
Letting $\tilde w_{\gamma,\delta}(t):=|t|^\gamma
(\log(e+|t|))^{-\delta/2}$, $t\in\mathbb R$, and $S_K^2:=\sup_{|t|\neq 0} 
[|1-\hat K(t)|/\tilde w_{\gamma,\delta}(t)]^2$,
for every 
$f\in\mathscr F_{\gamma,\delta}^{\mathcal{LS}}(L)$
we have $\|(1-\hat K)\hat f\|_2^2  <  \| w_{\gamma,\delta}\hat f\|_2^2 S_K^2
\leq L^2 S_K^2$.
Using the decomposition (1.41) in Theorem 1.4 of \cite{Tsybakov:2009}, pp. 21--22, for the MISE of a 
kernel density estimator $f_K$, we have, writing $\sup_f$ for the supremum over $\mathscr F^{\mathcal{LS}}_{\gamma,\delta}(L)$, 
\[
\begin{split}
\inf_{\hat f_n} &\sup_{f}\mathbb E^n_f [\|\hat f_n-f\|_2^2]\leq 
\inf_K\sup_{f\in \mathscr F^{\mathcal{LS}}_{\gamma,\delta}(L)}\mathbb E^n_f[\|f_K-f\|_2^2]\\
&=
\inf_K\sup_{f}\mathbb E^n_f\frac{\|(1-\hat K)\hat f\|_2^2
+\frac{1}{n}(\|\hat K\|_2^2-\|\hat K\hat f\|_2^2)}{2\pi} 
\end{split} 
\]
Since the first term is bounded by $L^2S_K^2$, we have
\[
\begin{split}
\inf_{\hat f_n} &\sup_{f}\mathbb E^n_f [\|\hat f_n-f\|_2^2] <  \inf_K \bigg[L^2 S_K^2+ \sup_{f}
\frac{1}{n}\mathbb E^n_f\big(\|\hat K\|_2^2-\|\hat K\hat f\|_2^2\big)\bigg]\\
& \leq  \inf_{c>0} \inf_{\{K:\,|1-\hat K(t)|\leq c \tilde w_{\gamma,\delta}(t)\}} \Bigg[
L^2S_K^2 + 
\sup_{f}\frac{1}{n}\mathbb E^n_f(\|\hat K\|_2^2-\|\hat K\hat f\|_2^2)\Bigg] \\
&\leq  \inf_{c>0} \pq{c^2L^2+ \frac{2}{n} \inf_{\{K:\,|1-\hat K(t)|\leq c \tilde w_{\gamma,\delta}(t)\}} 
\|\hat K\|_2^2} \leq  \inf_{c>0} \pq{c^2 L^2+ 
\frac{2}{n}\int_{\mathbb R}[1-c \tilde w_{\gamma,\delta}(t)]_+^2\,\di t }\\
&=\pq{c_{\mathrm{min}}^2 L^2+ 
\frac{2}{ n}\int_{\mathbb R}[1- c_{\mathrm{min}}\tilde w_{\gamma,\delta}(t)]_+^2\,\di t }\lesssim \psi_{n,\gamma}^2(\log n)^{\delta/(2\gamma+1)},
\end{split} 
\]
where $\min_{x>0}Ax^2+Bn^{-1}x^{-1/\gamma}\log^{\delta/(2\gamma)}(1/x)$
is achieved at 
$x_{\mathrm{min}}\propto n^{-\gamma/(2\gamma+1)}(\log n)^{\delta/[2(2\gamma+1)]}$.
The assertion follows by taking $c_{\mathrm{min}}=x_{\mathrm{min}}$. 
\qed


\section{Auxiliary lemmas used in the proof of  Theorem \ref{theo:1} on the inversion inequality}\label{sec:lem:theo:1}
The following lemma assesses the order of the bias, in terms of the kernel bandwidth, of a distribution function 
having derivatives up to a certain order, with locally H\"older continuous derivative of the highest order.

\begin{lem}\label{lem:der}
Let $F_{0X}$ be the distribution function of $\mu_{0X}\in\mathscr P_0$ satisfying 
condition \eqref{eq:holder} of Assumption \ref{ass:smoothXXX} 
for $\alpha>0$. Let $K$ be a kernel of order $(\lfloor \alpha\rfloor+1)$ 
satisfying $\int_{\mathbb R}|z|^{\alpha+1}|K(z)|\,\di z<\infty$.
Then, for every $h>0$,
\begin{equation}\label{eq:derivative}
\|F_{0X}\ast K_h-F_{0X}\|_1=O(h^{\alpha+1}).
\end{equation}
\end{lem}

\begin{proof}
Recalled that $b_{F_{0X}}:=(F_{0X}\ast K_h-F_{0X})$, write
$b_{F_{0X}}(x)=\int_{\mathbb R}[F_{0X}(x-hu)-F_{0X}(x)]K(u)\,\di u$. 
Let $\ell=\lfloor \alpha\rfloor$. 
For any $x,\,u\in\mathbb R$ and $h>0$, by Taylor's expansion,
\[F_{0X}(x-hu)=F_{0X}(x)-huf_{0X}(x)+\,\ldots\,+\frac{(-hu)^{\ell+1}}{\ell!}\int_0^1(1-\tau)^{\ell}f_{0X}^{(\ell)}(x-\tau hu)\,\di \tau.\]
Since $K$ is a kernel of order $\ell+1=\lfloor\alpha\rfloor+1$, we have
\[
\begin{split}
b_{F_{0X}}(x)=\int_{\mathbb R} K(u)\frac{(-hu)^{\ell+1}}{\ell!}\int_0^1(1-\tau)^{\ell}
\big[f_{0X}^{(\ell)}(x-\tau hu)-f_{0X}^{(\ell)}(x)\big]\,\di \tau\,\di u.
\end{split}
\]
Condition \eqref{eq:holder} yields that
\[
\begin{split}
\|b_{F_{0X}}\|_1&\leq\int_{\mathbb R}\int_{\mathbb R}|K(u)|\frac{|hu|^{\ell+1}}{\ell!}\int_0^1(1-\tau)^{\ell}
\big|f_{0X}^{(\ell)}(x-\tau hu)-f_{0X}^{(\ell)}(x)\big|\,\di \tau\,\di u\,\di x\\
&\leq h^{\alpha+1}
\|L_0\|_1\frac{1}{\ell!}\pt{\int_{\mathbb R}|u|^{\alpha+1}|K(u)|\,\di u}
\int_0^1(1-\tau)^{\ell}\tau^{\alpha-\ell}\,\di \tau.
\end{split}
\]
By the assumptions $L_0\in L^1(\mathbb R)$ and $\int_{\mathbb R}|z|^{\alpha+1}|K(z)|\,\di z<\infty$,
we conclude that $\|b_{F_{0X}}\|_1=O(h^{\alpha+1})$.
\end{proof}


Analogously to Lemma \ref{lem:der}, the next lemma assesses the order of the bias, 
in terms of the kernel bandwidth, of a distribution function with density in a Sobolev type space.

\begin{lem}\label{lem:sob}
Let $F_{0X}$ be the distribution function of $\mu_{0X}\in\mathscr P_0$ satisfying condition \eqref{ass:smoothsob} of Assumption 
\ref{ass:smoothXXX} for $\alpha>0$. 
Let $K\in L^1(\mathbb R)\cap L^2(\mathbb R)$ be symmetric, with 
$\int_{\mathbb R}|z||K(z)|\,\di z<\infty$, and $\hat K\in L^1(\mathbb R)$ such that
$\hat K\equiv 1$ on $[-1,\,1]$.
Then, for every $h>0$,
\begin{equation}\label{eq:derivative212}
\|F_{0X}\ast K_h-F_{0X}\|_1=O(h^{\alpha+1}).
\end{equation}
\end{lem}

\begin{proof}
Recalled the notation $b_{F_{0X}}:=(F_{0X}\ast K_h-F_{0X})$,
by the same arguments used for the function $G_{2,h}$ in \cite{dedecker2015}, pp. 251--252, we have
\[\begin{split}
\|b_{F_{0X}}\|_1&=
\int_{\mathbb R}\bigg|\frac{1}{2\pi}\int_{|t|>1/h}
\exp{(-\imath t x)}\frac{1-\hat K(ht)}{(-\imath t)}\hat f_{0X}(t)\,\di t
\bigg|\di x
\end{split}\]
because $t\mapsto[1-\hat K(ht)][\hat f_{0X}(t)\1_{[-1,\,1]^c}(ht)/t]$ is in $L^1(\mathbb R)$ due to the first part of
condition \eqref{ass:smoothsob}. 
By Young's inequality and the second part of condition \eqref{ass:smoothsob} on 
$D^{\alpha}\hspace*{-1pt}f_{0X}\in L^1(\mathbb R)$, we have
 \[\begin{split}
\|b_{F_{0X}}\|_1&=\int_{\mathbb R}\bigg|\frac{1}{2\pi}\int_{|t|>1/h}
\exp{(-\imath t x)}\frac{1-\hat K(ht)}{(-\imath t)^{\alpha+1}}\underbrace{(-\imath t)^{\alpha}\hat f_{0X}(t)}_
{=\widehat{D^{\alpha}\hspace*{-1pt}f_{0X}}(t)}\,\di t
\bigg|\,\di x\\
&\leq\|D^{\alpha}\hspace*{-1pt}f_{0X}\|_1
\Bigg(
\underbrace{\int_{|x|\leq h}}_{=:B_1}+\underbrace{\int_{|x|>h}}_{=:B_2}
\Bigg)
\abs{\frac{1}{2\pi}\int_{|t|>1/h}
\exp{(-\imath t x)}\frac{1-\hat K(ht)}{(-\imath t)^{\alpha+1}}\,\di t
}\di x,
\end{split}
\]
where $B_1\lesssim h\int_{|t|>1/h}
[1+|\hat K(ht)|]|t|^{-(\alpha+1)}\,\di t\lesssim h^{\alpha+1}$ because $\|\hat K\|_\infty\leq \|K\|_1<\infty$.
Therefore, $B_1=O(h^{\alpha+1})$. To bound $B_2$, we recall that, for every $j\in \mathbb N$, letting
$\hat f^{(j)}$ denote the $j$th derivative of the Fourier transform $\hat f$ of a function $f:\,\mathbb R\rightarrow \mathbb C$, 
if $\hat f^{(j)}\in L^1(\mathbb R)$, then
\begin{equation}\label{eq:identity21}
\mbox{for $z\neq0$}, \quad f(z)=\frac{1}{2\pi(\imath z)^j}\int_{\mathbb R}\exp{(-\imath tz)}\hat f^{(j)}(t)\,\di t.
\end{equation}
Since $K\in L^1(\mathbb R)$ and $zK(z)\in L^1(\mathbb R)$ jointly imply that $\hat K$ is 
continuously differentiable with $|\hat K^{(1)}(t)|\rightarrow 0$ as $|t|\rightarrow \infty$, 
so that $\hat K^{(1)}\in C_b(\mathbb R)$, 
defined $\hat f(t):=[1-\hat K(ht)](\imath t)^{-(\alpha+1)}\1_{[-1,\,1]^c}(ht)$, $t\in\mathbb R$,
we have $\hat f^{(1)}(t)=\{-h\hat K^{(1)}(ht)(\imath t)^{-(\alpha+1)}-\imath(\alpha+1)[1-\hat K(ht)](\imath t)^{-(\alpha+2)
}\}\1_{[-1,\,1]^c}(ht)$. For $f(\cdot):=(2\pi)^{-1}\int_{|t|>1/h}\exp{(-\imath t \cdot)}\hat f(t)\,\di t$, which is well defined because
$\hat f\in L^1(\mathbb R)$, by the identity \eqref{eq:identity21} and the Cauchy–Schwarz inequality,
\[
\begin{split}
B_2:=  \int_{|x|>h}|f(x)|\,\di x &= \int_{|x|>h}\frac{1}{|x|}\abs{\frac{1}{2\pi}\int_{|t|>1/h}
\exp{(-\imath t x)}\hat f^{(1)}(t)\,\di t
}\di x\\ 
&\lesssim
\pt{ \int_{\mathbb R}\frac{1}{x^2}\1_{|x|>h}\,\di x}^{1/2}
\pt{\int_{|t|>1/h}|\hat f^{(1)}(t)|^2\,\di t}^{1/2}\\
&\lesssim h^{-1/2}h^{\alpha+3/2}\lesssim  h^{\alpha+1}.
\end{split}
\]
Conclude that $B_2=O( h^{\alpha+1})$.
The assertion follows by combining the bounds on $B_1$ and $B_2$.
\end{proof}

\begin{rmk}
\emph{
Constants in \eqref{eq:derivative} and \eqref{eq:derivative212}
may depend on the kernel $K$ and the distribution function $F_{0X}$.
}
\end{rmk}

The next  lemma  provides the order of the $L^1$-norm, in terms of the kernel bandwidth, 
of function arising when controlling the term $T_1$ in Theorem \ref{theo:1}.
We recall the notation. Let $\chi:\,\mathbb R\rightarrow \mathbb R$ 
be a symmetric, continuously differentiable function, equal to $1$ on $[-1,\,1]$ and to $0$ 
outside $[-2,\,2]$. Let $K$ be the kernel defined in Section \ref{subsec:Wrate}.
Recall that an $(\lfloor\alpha\rfloor+1)$-order kernel is used when $f_{0X}$ verifies (i) of Assumption \ref{ass:smoothXXX} as in Lemma \ref{lem:der},
whereas a superkernel is used when $f_{0X}$ verifies (ii) of Assumption \ref{ass:smoothXXX} as in Lemma \ref{lem:sob}.
Recall that the Fourier transform $\hat K$ of $K$ has compact support.
For $h>0$, let $w_{1,h}(t):=\hat K(ht)\chi(t)r_\varepsilon(t)$, $t\in\mathbb R$, where 
$r_\varepsilon$ is as defined in \eqref{eq:re} and satisfies Assumption \ref{ass:identifiability+error}.
The function $K_{1,h}(\cdot):=(2\pi)^{-1}\int_{\mathbb R}\exp{(-\imath t\cdot)}w_{1,h}(t)\,\di t$ is the inverse 
Fourier transform of $w_{1,h}$.

\begin{lem}\label{lem:bigO}
If $\mu_\varepsilon\in\mathscr P_0$ satisfies Assumption \ref{ass:identifiability+error} for $\beta>0$, then, for sufficiently small $h>0$,
$$\|K_{1,h}\|_1=O(1).$$
\end{lem}

\begin{proof}
Denoted by $w_{1,h}^{(1)}$ the derivative of $w_{1,h}$ with respect to $t$,
we have 
$\|K_{1,h}\|_1 \leq 2^{-1/2}(
\|w_{1,h}\|_2^2 +\|w^{(1)}_{1,h}\|_2^2)^{1/2}$,
see the proof of Theorem 4.2 in \cite{Bobkov:2016}, pp. 1030--1031.
For $h\leq1/2$, by condition \eqref{eq:deriv} with $l=0$, we have
$\|w_{1,h} \|_2^2\lesssim\int_{|t|\leq2}|\hat K(ht)|^2|\chi(t)|^2(1+|t|)^{2\beta}\,\di t\lesssim 
\|\chi\|_2^2<\infty$ as $\hat K $ is bounded on any compact set.  Analogously, for 
$w^{(1)}_{1,h}(t)=[h\hat K^{(1)}(ht)\chi(t) +\hat K(ht)\chi^{(1)}(t)] r_\varepsilon(t)+\hat K(ht)\chi(t) r^{(1)}_\varepsilon(t)$, $t\in\mathbb R$,
using condition \eqref{eq:deriv} with $l=1$ and $\beta>0$, we have
\[\begin{split}
\|w^{(1)}_{1,h}\|_2^2&\lesssim\int_{|t|\leq2}[h|\hat K^{(1)}(ht)||\chi(t)|+|\hat K(ht)||\chi^{(1)}(t)|]^2(1+|t|)^{2\beta}\,\di t \\
&\hspace*{6cm} +\int_{|t|\leq2} |\hat K(ht)|^2|\chi(t)|^2 (1+|t|)^{2(\beta-1)}\,\di t 
\\[3pt]
&\lesssim \|\chi\|_2^2 + \|\chi^{(1)}\|_2^2<\infty
\end{split}\]
because also $\hat K^{(1)}$ is bounded on any compact set by continuity. 
The assertion follows.
\end{proof}




\section{Lemmas for Theorem \ref{thm:31} on 
posterior contraction rates for Dirichlet Linnik-normal mixtures}\label{sec:th31:supp}

In the following lemmas we prove the existence of a compactly supported discrete 
mixing probability measure such that the corresponding Linnik-normal mixture
has Hellinger distance of the appropriate order from a Linnik mixture and, furthermore,
the prior law on Linnik-normal mixtures concentrates 
on Kullback-Leibler neighborhoods of the true density at optimal rate, up to a logarithmic factor.

\smallskip

The next lemma provides an upper bound on the remainder term (or truncation  error) 
associated  with  the $(r-1)$th order Taylor polynomial about zero of the complex exponential function, 
see, \emph{e.g.}, Lemma 10.1.5 in \cite{Athreya:2006}, pp. 320--321.

\begin{lem}\label{lem:diseg}
For every $r\in\mathbb{N}$, we have 
$$\abs{\exp{(\imath x)}-\sum_{k=0}^{r-1}\frac{(\imath x)^k}{k!}}\leq
\min\pg{\frac{|x|^{r}}{r!},\, \frac{2|x|^{r-1}}{(r-1)!}}, \quad x\in\mathbb R.$$
\end{lem}
\smallskip

For later use, we recall that the bilateral Laplace transform of a function 
$f:\,\mathbb R\rightarrow \mathbb C$ is defined as 
$\mathcal B\{f\}(s): =\int_{\mathbb R}\exp{(-sx)} f(x)\,\di x$ for all $s\in\mathbb C$ such that
$\int_{\mathbb R}|\exp{(-sx)}f(x)|\,\di x=\int_{\mathbb R}\exp{(-\mathrm{Re}(s)x)}|f(x)|\,\di x<\infty$, 
where $\mathrm{Re}(s)$ denotes the real part of $s$. 
With abuse of notation, for a probability measure $\mu$ on $\mathbb R$, 
we define $\mathcal B\{\mu\}(s):=\int_{\mathbb R}\exp{(-sx)}\mu(\di x)$, $s\in\mathbb C$.
For all $t\in\mathbb R$ such that $\int_{\mathbb R}\exp{(tx)}\mu(\di x)<\infty$, the mapping $t\mapsto
M_{\mu}(t):=\int_{\mathbb R}\exp{(tx)}\mu(\di x)$
is the moment generating function of $\mu$ and $M_\mu(t) = \mathcal B\{\mu\}(-t)$, $t\in\mathbb R$.

\smallskip

In the following lemma, we prove the existence of a compactly supported discrete mixing
probability measure, with a sufficiently small number of support points, 
such that the corresponding Linnik-normal mixture
has Hellinger distance of the order $O(\sigma^{\beta})$ from the sampling density $f_{0Y}$.

\smallskip

\begin{lem}\label{lem:discrete}
Let $f_\varepsilon$ be a standard Linnik density with index $0<\beta\leq 2$. 
Let  $\mu_{0X}\in\mathscr P_0$ be a probability measure supported on $[-a,\,a]$, with density $f_{0X}$
such that $(e^{|\cdot|/2}f_{0X})\in L^2(\mathbb R)$.
For $\sigma>0$ small enough, there exists a discrete probability measure $\mu_H$ on $[-a,\,a]$, with at most 
$N=O((a/\sigma)|\log\sigma|^{1/2})$ 
support points, such that, for $f_Y:=f_\varepsilon\ast (\mu_H\ast \phi_\sigma)$ and 
$f_{0Y}:=f_\varepsilon\ast f_{0X}$,
\begin{equation*}\label{eq:discrete}
d_{\mathrm{H}}(f_Y,\,f_{0Y})\lesssim \delta_0^{-1/2}e^{a_0/2}\sigma^\beta, 
\end{equation*}
as soon as $P_{0X}(|X|\leq a_0)\geq \delta_0$ for some $0<a_0<a$ and $0<\delta_0<1$.
\end{lem}


\begin{proof}
We begin by analysing the case of a Laplace error distribution, which corresponds to $\beta=2$.
We then use the fact that a Linnik density with index $0<\beta<2$ is a scale mixture of Laplace densities, see
\cite{kotz1996}, to complete the proof.

\begin{description}
\item[$\bullet$] \emph{Laplace error distribution \emph{($\beta=2$)}}: $f_\varepsilon(\cdot)=e^{-|\cdot|}/2$\\[2pt]
For $a_0$ and $\delta_0$ as in the statement, we have
\[\begin{split}
f_{0Y}(y)\geq\int_{|x|\leq a_0}f_\varepsilon(y-x)
f_{0X}(x)\,\di x \geq 
\frac{1}{2}e^{-(|y|+a_0)}P_{0X}(|X|\leq a_0)\geq\frac{\delta_0}{2}e^{-(|y|+a_0)},\quad y\in\mathbb R.
\end{split}
\]
Define 
\begin{equation}\label{eq:defU}
U(y):=e^{-y/2}+e^{y/2}.
\end{equation}
By the inequality $e^{|y|/2}\leq U(y)$ valid for every $y\in\mathbb R$,
we have 
$$d_{\mathrm H}^2(f_Y,\,f_{0Y}) \leq 
2\delta_0^{-1}e^{a_0}\int_{\mathbb R}[e^{|y|/2}(f_Y-f_{0Y})(y)]^2\,\di y\leq 2\delta_0^{-1}e^{a_0}
\|g_Y-g_{0Y}\|_2^2,$$
where $g_Y:=Uf_Y$ and $g_{0Y}:=Uf_{0Y}$.
For $b=\mp 1/2$, noting that $e^{b\cdot}f_{0Y}=(e^{b\cdot}f_\varepsilon)\ast(e^{b\cdot}f_{0X})$, 
where $e^{b\cdot}f_{0X}\in L^1(\mathbb R)$ for compactly supported $f_{0X}$ and
$e^{b\cdot}f_\varepsilon\in L^p(\mathbb R)$
for every $1\leq p\leq\infty$, we have 
$\|e^{b\cdot}f_{0Y}\|_p\leq \|e^{b\cdot}f_\varepsilon\|_p\times\|e^{b\cdot}f_{0X}\|_1<\infty$.
Analogously, $e^{b\cdot}f_Y=(e^{b\cdot}(f_\varepsilon\ast \phi_\sigma))\ast(e^{b\cdot}\mu_H)$, 
where 
$M_{\mu_H}(b)<\infty
$ 
for compactly supported $\mu_H$ and
$e^{b\cdot}(f_\varepsilon\ast \phi_\sigma)\in L^p(\mathbb R)$
for every $1\leq p\leq\infty$, we have 
$\|e^{b\cdot}f_Y\|_p\leq \|e^{b\cdot}(f_\varepsilon\ast \phi_\sigma)\|_p\times M_{\mu_H}(b)<\infty$.
Consequently, $g_Y,\,g_{0Y}\in L^1(\mathbb R)\cap L^2(\mathbb R)$ and 
the corresponding Fourier transforms $\hat g_Y(t):=\int_{\mathbb R}e^{\imath t y}g_{Y}(y)\,\di y$ and 
$\hat g_{0Y}(t):=\int_{\mathbb R}e^{\imath t y}g_{0Y}(y)\,\di y$, $t\in\mathbb R$, are well defined.
Also, $\|g_{0Y}\|_2^2=(2\pi)^{-1}\|\hat g_{0Y}\|_2^2$ and $\|g_Y\|_2^2=(2\pi)^{-1}\|\hat g_Y\|_2^2$.
For $\psi_b(t):=-(\imath t+b)$, let $\varrho_b(t):=[1-\psi_b^2(t)]$, $t\in\mathbb R$. Note that 
$\varrho_{-1/2}(t)=\overline{\varrho_{1/2}(t)}$ and $|\varrho_{-1/2}(t)|^2=|\varrho_{1/2}(t)|^2=(t^4+5t^2/2+9/16)$. 
Noting that 
$$\mathcal B\{f_\varepsilon(\cdot-x)\}(\psi_b(t))=
\frac{e^{-\psi_b(t)x}}{\varrho_b(t)},\quad t,\,x\in\mathbb R,$$
we have
\[
\begin{split}
r(t;\,x)&:=\int_{\mathbb R}e^{\imath t y}U(y)f_\varepsilon(y-x)\,\di y\\&=\sum_{b=\mp 1/2}\mathcal B\{f_\varepsilon(y-x)\}(\psi_b(t))
=\sum_{b=\mp 1/2}\frac{e^{-\psi_b(t)x}}{\varrho_{b}(t)}
,\quad t,\,x\in\mathbb R.
\end{split}
\]
Then,
$\hat g_{0Y}(t)
=\int_{|x|\leq a}r(t;\,x)f_{0X}(x)\,\di x
=\sum_{b=\mp 1/2}\mathcal B\{f_{0X}\}(\psi_b(t))/\varrho_{b}(t)$,
$t\in\mathbb R$.   
We derive the expression of $\hat g_Y$.
Since
\begin{equation}\label{eq:lapnorm}
\mathcal B\{\phi_\sigma(\cdot-u)\}(\psi_b(t))
=\exp{(-\psi_b(t)u+\sigma^2\psi_b^2(t)/2)}, \quad t,\,u\in\mathbb R,
\end{equation} 
we have
\[
\hat g_Y(t)
=\int_{|u|\leq a}\pt{\int_{\mathbb R}r(t;\,x)\phi_\sigma(x-u)\,\di x}\mu_H(\di u)
=\sum_{b=\mp 1/2}\frac{e^{\sigma^2\psi_b^2(t)/2}}{\varrho_{b}(t)}
\mathcal B\{\mu_H\}(\psi_b(t)), 
\quad t\in\mathbb R.
\] 
For ease of notation, we introduce the integrals
\[I_b:=\int_{\mathbb R}\frac{1}{|\varrho_b(t)|^2}
\bigg|e^{\sigma^2\psi_b^2(t)/2}
\mathcal B\{\mu_H\}(\psi_b(t))-\mathcal B\{f_{0X}\}(\psi_b(t))
\bigg|^2\di t,\quad b=\mp 1/2.\]
By Plancherel's theorem and the triangular inequality, 
$2\pi\|g_Y-g_{0Y}\|_2^2=\|\hat g_Y-\hat g_{0Y}\|_2^2\leq 2(I_{-1/2}+I_{1/2})$.
Both terms $I_{-1/2}$ and $I_{1/2}$ can be controlled using the same arguments, we therefore 
consider a unified treatment for $I_b$. For $M>0$, we have
\[
\begin{split}  
I_b
&\leq
\pt{\int_{|t|\leq M}+\int_{|t|>M}}\frac{|e^{\sigma^2\psi_b^2(t)/2}|^2}{{|\varrho_b(t)|^2}}
\big|(\mathcal B\{\mu_H\}-\mathcal B\{f_{0X}\})(\psi_b(t))\big|^2\di t\\
&\hspace*{3.5cm}+
\int_{\mathbb R}\frac{1}{{|\varrho_b(t)|^2}}
|e^{\sigma^2\psi_b^2(t)/2}-1|^2
\abs{\mathcal B\{f_{0X}\}(\psi_b(t))
}^2\di t=:\sum_{k=1}^3 I_b^{(k)}.
\end{split}
\]  
\noindent
\emph{Study of the term $I_b^{(1)}$}\\[5pt]
The term $I_b^{(1)}$ can be bounded similarly to $I_1$ in Lemma 2 of \cite{gao2016}, p. 616. 
Preliminarily note that,
for $\sigma<1/|b|=2$, we have
$|e^{\sigma^2\psi_b^2(t)/2}|^2=|e^{\sigma^2(-t^2+2\imath bt+b^2)/2}|^2
=e^{-\sigma^2(t^2-b^2)}=e^{-\sigma^2(t^2-1/4)}< e$.
Let $\mu_H$ be a discrete probability measure on $[-a,\,a]$ satisfying the constraints
\begin{equation}\label{eq:c2}
\begin{split}
\int u^j\mu_H(\di u) &=\int u^j f_{0X}(u)\,\di u, \quad\mbox{for } j=0,\,\ldots,\,J-1, \\
\int e^{bu}\mu_H(\di u)& =\int e^{bu}f_{0X}(u)\,\di u,\quad  b= \pm 1/2,
\end{split}
\end{equation}
where $J=\lceil\eta ea M\rceil$ for some $\eta>1$. Note that the last constraints can be written as $M_{\mu_H}(b)=M_{0X}(b)$, $b=\pm1/2$.
Using Lemma \ref{lem:diseg} with $r=J$,
by the inequality $J!\geq (J/e)^J$, we have 
\[\begin{split}
I_b^{(1)}&\lesssim \int_{|t|\leq M}\frac{1}{{|\varrho_b(t)|^2}}
\bigg|\int_{|u|\leq a}  \bigg[e^{\psi_b(t) u}-\sum_{j=0}^{J-1}\frac{[\psi_b(t)u]^j}{j!}\bigg](\mu_H-\mu_{0X})(\di u)
\bigg|^2\,\di t\\
&\lesssim[M_{\mu_H}(b)+M_{0X}(b)]^2\frac{1}{(J!)^2}
\int_{|t|\leq M}\frac{(a|t|)^{2J}}{{|\varrho_b(t)|^2}}\,\di t\\
&\lesssim \frac{a^{2J}}{(J!)^2}
\int_{0}^M t^{2(J-2)}\,\di t \\
&\lesssim \frac{a^{2J}}{(J!)^2}\times
\frac{M^{2J-3}}{2J-3}\lesssim M^{-4} \frac{a^{2J}}{(J!)^2}\times
\frac{M^{2J+1}}{2J-3}\lesssim  M^{-4} \pt{\frac{eaM}{J}}^{2J+1}\lesssim M^{-4}.
\end{split}\]   
\noindent
\emph{Study of the term $I_b^{(2)}$}\\[5pt]
Choosing $M$ so that  
$(\sigma M)^2\geq|\log\sigma|$, equivalently, $M\geq \sigma^{-1}|\log\sigma|^{1/2}
$, and using the facts that $|e^{\sigma^2\psi_b^2(t)/2}|^2=O(e^{-\sigma^2t^2})$ and $|\mathcal B\{\mu_H\}(\psi_b(t))|\leq M_{\mu_H}(b)=M_{0X}(b)$,
we have
$$I_b^{(2)}
\lesssim M_{0X}(b)
e^{-(\sigma M)^2}
\int_{|t|>M}\frac{1}{t^4}
\,\di t
\lesssim
e^{-(\sigma M)^2}M^{-3}\lesssim 
\sigma M^{-3}
\lesssim \sigma^4.$$
\noindent
\emph{Study of the term $I_b^{(3)}$}\\[5pt]
Noting that 
\begin{equation}\label{eq:disff}
\mathcal B\{f_{0X}\}(\psi_b(t))=(\widehat{e^{b\cdot}f_{0X}})(t), \quad t\in\mathbb R,
\end{equation} 
and, by Lemma \ref{lem:diseg},
\begin{equation}\label{eq:disf}
|e^{\sigma^2\psi_b^2(t)/2}-1|\leq \min\{2,\,\sigma^2(t^2+b^2)/2\}\leq\sigma^2(t^2+1/4)/2,
\end{equation}
we have
\[  
\begin{split}
I_b^{(3)}
&\leq \frac{\sigma^4}{2}\int_{\mathbb R}\frac{(t^4+b^4)}{{|\varrho_b(t)|^2}}
\abs{\mathcal B\{f_{0X}\}(\psi_b(t))
}^2\,\di t
\lesssim
\sigma^4
\int_{\mathbb R}
|(\widehat{e^{b\cdot}f_{0X}})(t)|^2\,\di t\lesssim
\sigma^4,
\end{split}   
\]  
where, by Plancherel's theorem,
$(2\pi)^{-1}\|\widehat{e^{b\cdot}f_{0X}}\|_2^2=\|e^{b\cdot}f_{0X}\|_2^2<\infty
$ by the assumption that $e^{b\cdot}f_{0X}\in L^2(\mathbb R)$.

\vspace*{-0.1cm}
\smallskip
The existence of a discrete probability measure $\mu_H$ supported on $[-a,\,a]$, with at most 
$2(J+1)\propto (a M)\gtrsim (a/\sigma)|\log\sigma|^{1/2}$ support 
points, is guaranteed by Lemma A.1 of \cite{ghosal:2001}, p. 1260.
Combining the bounds on $I_b^{(k)}$, for $k=1,\,2,\,3$, we conclude that 
$\|g_Y-g_{0Y}\|_2^2\lesssim 
\sigma^4$.
\end{description}


\begin{description}
\item[$\quad\bullet\,\,$] \emph{Linnik error distribution with index $0<\beta<2$}\\[2pt]
Every Linnik density $f_\varepsilon$ with index $0<\beta<2$ admits a
representation as a scale mixture of Laplace densities $f_{1/v}(\cdot):=ve^{-v|\cdot|}/2$, $v>0$,
\[
f_\varepsilon(u)=\int_0^\infty f_{1/v}(u)f_V(v;\,\beta)\,\di v,\quad u\neq0,
\]
with mixing density
\begin{equation}\label{eq:mixing}
f_V(v;\,\beta):=\pt{\frac{2}{\pi}\sin\frac{\pi\beta}{2}}\frac{v^{\beta-1}}
{1+v^{2\beta}+2v^\beta\cos(\pi\beta/2)},\quad v>0.
\end{equation}
Let $V$ be a random variable with the density in \eqref{eq:mixing}. Note that 
$\mathbb E[V]=\int_0^\infty v f(v;\,\beta)\,\di v<\infty$ for $1<\beta<2$ and infinite
for $0<\beta\leq1$, in which case $f_\varepsilon$ has an infinite peak at $u=0$.
Writing $(f_Y-f_{0Y})(\cdot)=[f_\varepsilon\ast (\mu_X-\mu_{0X})](\cdot)=\int_0^\infty
[f_{1/v}\ast(\mu_X-\mu_{0X})](\cdot)f_V(v;\,\beta)\,\di v$, we have
\[\begin{split}
d^2_{\mathrm{H}}(f_Y,\,f_{0Y})&<
\int_{\mathbb R}\bigg\{\pt{\int_0^1+\int_1^\infty}[f_{1/v}\ast(\mu_X-\mu_{0X})](y)
\frac{f_V(v;\,\beta)}{\sqrt{f_{0Y}(y)}}
\,\di v\bigg\}^2\di y\\
&\leq 2 
\int_{\mathbb R}\bigg\{\int_0^1[f_{1/v}\ast(\mu_X-\mu_{0X})](y)
\frac{f_V(v;\,\beta)}{\sqrt{f_{0Y}(y)}}
\,\di v\bigg\}^2\di y \\
&\hspace*{1cm} + 2\int_{\mathbb R}\bigg\{\int_1^\infty [f_{1/v}\ast(\mu_X-\mu_{0X})](y)
\frac{f_V(v;\,\beta)}{\sqrt{f_{0Y}(y)}}
\,\di v\bigg\}^2\di y
=:I_1 + I_2.
\end{split}
\]

\bigskip
\noindent
\emph{Study of the term $I_1$}\\[5pt]
For $z>0$, we define the function $z\mapsto E(z):=\mathbb E[V\1_{\{V<z\}}]=\int_0^z vf(v;\,\beta)\,\di v$.
Let $v>0$ be fixed and $a_0,\,\delta_0$ as in the statement. For every $y\in\mathbb R$,
\begin{equation}\label{eq:LB}
f_{0Y}(y)\geq P_{0X}(|X|\leq a_0)\int_0^{v}\frac{u}{2} e^{-u(|y|+a_0)} f_V(u;\,\beta) \,\di u \geq
\frac{\delta_0}{2}e^{- v(|y|+a_0)} E(v).
\end{equation}
By the Cauchy–Schwarz inequality,
\[\begin{split}
\frac{I_1}{2}&\leq\frac{2e^{a_0}}{\delta_0}
\int_{\mathbb R}\bigg[\int_0^1 e^{v|y|/2}
\bigg(\int_{|x|\leq a}e^{-v|y-x|}(\mu_X-\mu_{0X})(x)\,\di x\bigg)
\frac{vf_V(v;\,\beta)}{\sqrt{E(v)}}
\,\di v\bigg]^2\di y\\
&\leq \frac{2e^{a_0}}{\delta_0}\pt{\int_0^1f_V(v;\,\beta)\,\di v}\\
&\qquad\qquad\quad\times
\int_0^1\underbrace{
\int_{\mathbb R}
\bigg(e^{v|y|/2}\int_{|x|\leq a}e^{-v|y-x|}(\mu_X-\mu_{0X})(x)\,\di x\bigg)^2\,\di y}_{=:\mathcal I(v)}
\frac{v^2f_V(v;\,\beta)}{E(v)}
\,\di v\\
&\lesssim
\int_0^1\mathcal I(v)
\frac{v^2f_V(v;\,\beta)}{E(v)}
\,\di v.
\end{split}
\]
Since the function $v\mapsto \mathcal I(v)$ is continuous, by the mean value theorem, there exists $\bar v\in(0,\,1)$ such that 
$\int_0^1\mathcal I(v)
[{v^2f_V(v;\,\beta)}/{E(v)}]
\,\di v=\mathcal I(\bar v)\int_0^1[{v^2f_V(v;\,\beta)}/{E(v)}]\,\di v\lesssim \mathcal I(\bar v)$
because, for every $0<\gamma<1$,
$$E(v)\gtrsim\int_{(\gamma v)^\beta}^{v^\beta}\frac{z^{1/\beta}}{1+z^2+2z\cos(\pi \beta/2)}\,\di z\gtrsim 
v\int_{(\gamma v)^\beta}^{v^\beta}\frac{1}{(1+z)^2}\,\di z\gtrsim v^{1+\beta}$$
and, consequently,
$
\int_0^1[{v^2f_V(v;\,\beta)}/{E(v)}]
\,\di v\lesssim
\int_0^1
\{{1}/[{1+v^{2\beta}+2v^\beta \cos(\pi\beta/2)}]\}\,\di v<\infty$. The integral $\mathcal I(\bar v)$ can be bounded above by $\sigma^4$ 
using the same arguments as for the standard Laplace error distribution, 
with $\mp\bar v/2$ playing the role of $b$. 
Thus,
$I_1\lesssim \sigma^4\lesssim \sigma^{2\beta}$.

\bigskip
\noindent
\emph{Study of the term $I_2$}\\[5pt]
Taking $v=1$ in \eqref{eq:LB}, we have $
f_{0Y}(y)\geq\delta_0e^{-(|y|+a_0)}E(1)/2$, $y\in\mathbb R$. For 
$$\bar f_\varepsilon(u):=\int_1^\infty f_{1/v}(u)f_V(v;\,\beta)\,\di v,\quad u\neq0,$$
let $\bar g_{0Y}:=U (\bar f_\varepsilon\ast \mu_{0X})$ and $\bar g_Y:=U(\bar f_\varepsilon \ast \mu_X)$.
By the same arguments laid down for the standard Laplace error distribution case, we have
$\bar g_Y,\,\bar g_{0Y}\in L^1(\mathbb R)\cap L^2(\mathbb R)$ and 
the corresponding Fourier transforms $\hat {\bar g}_Y(t):=\int_{\mathbb R}e^{\imath t y}{\bar g}_Y(y)\,\di y$ and 
$\hat {\bar g}_{0Y}(y):=\int_{\mathbb R}e^{\imath t y}{\bar g}_{0Y}(y)\,\di y$, $t\in\mathbb R$, are well defined.
We have
\[\begin{split}
\frac{I_2}{2}
&\leq\frac{2e^{a_0}}{\delta_0E(1)}
\int_{\mathbb R}\pg{e^{|y|/2}\int_1^\infty 
[f_{1/v}\ast(\mu_X-\mu_{0X})](y)f_V(v;\,\beta) 
\,\di v}^2\,\di y\\&=\frac{2e^{a_0}}{\delta_0E(1)}
\int_{\mathbb R}\{e^{|y|/2}[\bar f_\varepsilon\ast(\mu_X-\mu_{0X})](y)\}
^2\,\di y\\&\leq\frac{2e^{a_0}}{\delta_0E(1)}\|\bar g_Y-\bar g_{0Y}\|_2^2=\frac{e^{a_0} }{\pi\delta_0 E(1)}\|\hat {\bar g}_Y-\hat {\bar g}_{0Y}\|_2^2.
\end{split}
\]
We derive the expressions of $\hat {\bar g}_{0Y}$ and $\hat {\bar g}_Y$:
\[
\begin{split}
\hat {\bar g}_{0Y}(t)&=\int_{|x|\leq a}\pt{\int_1^\infty f_V(v;\,\beta)\int_{\mathbb R}e^{\imath t y}U(y)f_{1/v}(y-x)\,\di y\,\di v}f_{0X}(x)\,\di x\\
&=\sum_{b=\mp1/2}\pt{\int_1^\infty \frac{vf_V(v;\,\beta)}{v^2-\psi_b^2(t)}\,\di v}\mathcal B\{f_{0X}\}(\psi_b(t)),\quad t\in\mathbb R,
\end{split}
\]
and, using the expression of $\mathcal B\{\phi_\sigma(\cdot-u)\}(\psi_b(t))$ in \eqref{eq:lapnorm},
\[
\hat {\bar g}_Y(t)=\sum_{b=\mp 1/2}
\pt{\int_1^\infty \frac{v f_V(v;\,\beta)}{v^2-\psi_b^2(t)}\,\di v}
{e^{\sigma^2\psi_b^2(t)/2}}\mathcal B\{\mu_H\}(\psi_b(t)), 
\quad t\in\mathbb R.
\] 
Thus, $\|\hat{\bar g}_Y-\hat{\bar g}_{0Y}\|_2^2\leq 2(\bar I_{-1/2}+\bar I_{1/2})$, where, for $b=\mp 1/2$,
\[\begin{split}
\bar I_b&:=\int_{\mathbb R}\abs{\int_1^\infty \frac{v f_V(v;\,\beta)}{v^2-\psi_b^2(t)}\,\di v}^2
\Big|e^{\sigma^2\psi_b^2(t)/2}
\mathcal B\{\mu_H\}(\psi_b(t))-\mathcal B\{f_{0X}\}(\psi_b(t))
\Big|^2\di t.
\end{split}\]
By identity (2) in \cite{kotz1996}, p. 63, (with their $\beta=2$), 
\[\begin{split}
\abs{\int_1^\infty \frac{v f_V(v;\,\beta)}{v^2-\psi_b^2(t)}\,\di v}&\leq \int_1^\infty \frac{v f_V(v;\,\beta)}{|v^2-\psi_b^2(t)|}\,\di v\lesssim\int_0^\infty \frac{v^2 f_V(v;\,\beta)}{v^2+t^2}\,\di v=\frac{1}{1+|t|^\beta},\quad
t\in\mathbb R.
\end{split}\]
For $M>0$, we have
\[
\begin{split}  
\bar I_b
&\leq
\pt{\int_{|t|\leq M}+\int_{|t|>M}}\frac{|e^{\sigma^2\psi_b^2(t)/2}|^2}{{(1+|t|^\beta)^2}}
\Big|(\mathcal B\{\mu_H\}-\mathcal B\{f_{0X}\})(\psi_b(t))\Big|^2\di t\\
&\hspace*{1.5cm}+
\pt{\int_{|t|\leq 1/\sigma}+\int_{|t|>1/\sigma}}\frac{1}{{(1+|t|^\beta)^2}}
|e^{\sigma^2\psi_b^2(t)/2}-1|^2
\abs{\mathcal B\{f_{0X}\}(\psi_b(t))
}^2\di t\\
&=:\sum_{k=1}^4 \bar I_b^{(k)}.
\end{split}
\] 
Using the same construction of $\mu_H$ on $[-a,a]$ as for the standard Laplace error distribution case, with at most
$2(J+1)\propto (a M)\gtrsim (a/\sigma)|\log\sigma|^{1/2}$, for $M\geq\sigma^{-1}|\log \sigma|^{1/2}$, 
support points such that the constraints in \eqref{eq:c2} hold true. 
Let $\sigma<1/|b|$. By the same arguments 
laid down for $I_b^{(1)}$ and  $I_b^{(2)}$ of the standard Laplace error distribution case,
\[
\bar I_b^{(1)}\lesssim 
\frac{a^{2J}}{(J!)^2}\int_0^M t^{2(J-\beta)}\,\di t \lesssim 
M^{-2\beta}\frac{a^{2J}}{(J!)^2}\times \frac{M^{2J+1}}{2J-2\beta+1}\lesssim M^{-2\beta}\lesssim \sigma^{2\beta}
\]
and
\[\begin{split}
\bar I_b^{(2)}\lesssim \int_{|t|>M}\frac{e^{-(\sigma t)^2}}{(1+|t|^{\beta})^2}\,\di t\lesssim M^{-2\beta}
\int_{|t|>M}e^{-(\sigma t)^2}\,\di t\lesssim
\sigma^{2\beta}.
\end{split}\]
Recalling the identity in \eqref{eq:disff} and the inequality in \eqref{eq:disf}, since $e^{b\cdot}f_{0X}\in L^2(\mathbb R)$ by assumption,
\[\begin{split}
\bar I_b^{(3)}&\lesssim \sigma^4\int_{|t|\leq 1/\sigma}\frac{(t^4+b^4)}{(1+|t|^\beta)^2}\abs{\mathcal B\{f_{0X}\}(\psi_b(t))}^2\,\di t\\
&\lesssim
\sigma^4\int_{|t|\leq 1/\sigma}|t|^{2(2-\beta)}\abs{\mathcal B\{f_{0X}\}(\psi_b(t))}^2\,\di t \lesssim \sigma^{2\beta}
\int_{\mathbb R}|(\widehat{e^{b\cdot}f_{0X}})(t)|^2\,\di t \lesssim \sigma^{2\beta}.
\end{split}\]
Also, 
\[
\bar I_b^{(4)}\lesssim \int_{|t|>1/\sigma}\frac{1}{|t|^{2\beta}}\abs{\mathcal B\{f_{0X}\}(\psi_b(t))}^2\,\di t \lesssim\sigma^{2\beta}
\int_{\mathbb R}|(\widehat{e^{b\cdot}f_{0X}})(t)|^2\,\di t \lesssim \sigma^{2\beta}.
\]

\smallskip

Combining the bounds on $\bar I_b^{(k)}$, for $k=1,\,\ldots,\,4$ and $b=\mp 1/2$,
we have 
$I_2\lesssim\|\hat{\bar g}_Y-\hat{\bar g}_{0Y}\|_2^2\lesssim \sum_{b=\mp1/2}\bar I_b\lesssim 
\sum_{b=\mp1/2}\sum_{k=1}^4\bar I_b^{(k)}\lesssim \sigma^{2\beta}$.
\end{description}
Conclude that $d^2_{\mathrm H}(f_Y,\,f_{0Y})\lesssim(I_1+I_2)\lesssim \delta_0^{-1} e^{a_0}\sigma^{2\beta}$, which completes the proof.
\end{proof}


The next lemma gives sufficient conditions on the prior law $\mathscr D_{H_0}\otimes \Pi_\sigma$ so that 
the induced probability measure on Linnik-normal mixtures 
$f_Y=f_\varepsilon\ast(\mu_H\ast \phi_\sigma)$ concentrates on Kullback-Leibler neighborhoods 
of a Linnik mixture $f_{0Y}=f_\varepsilon\ast f_{0X}$, with index $1<\beta\leq2$
and mixing density $f_{0X}$ having exponentially decaying tails, at a rate of the order $O(n^{-\beta/(2\beta+1)}(\log n)^\tau)$
for a suitable $\tau>0$.

\begin{lem}\label{lem:KL}
Let $f_{0Y}:=f_\varepsilon\ast f_{0X}$, where $f_\varepsilon$ is the density of a standard 
Linnik distribution with index $1<\beta\leq 2$ and $f_{0X}$ satisfies Assumption \ref{ass:twicwtailcond}. 
Let the model be $f_Y:=f_\varepsilon\ast (\phi_\sigma\ast \mu_H)$, with $\mu_H\in\mathscr P$.
If the base measure $H_0$ of the Dirichlet process prior $\mathscr D_{H_0}$ for $\mu_H$ 
satisfies Assumption \ref{ass:basemeasure1} and the prior $\Pi_\sigma$ for $\sigma$ satisfies 
Assumption \ref{ass:priorscale1} with $0<\gamma\leq1$, then
$(\mathscr D_{H_0}\otimes \Pi_\sigma)
(N_{\mathrm{KL}}(P_{0Y};\,\tilde\epsilon_n^2)
)\gtrsim \exp{(-C n\tilde\epsilon_n^2)}$ for $\tilde\epsilon_n=n^{-\beta/(2\beta+1)}(\log n)^{1/2+\beta(t_1\vee 3/2)/(2\beta+1)}$.
\end{lem}

\begin{proof} 
We show that, for some constant $C>0$, the prior probability of a 
Kullback-Leibler neighborhood of $P_{0Y}$ of radius $\tilde\epsilon_n^2$ is at least $\exp{(-C n\tilde\epsilon_n^2)}$.
We apply Lemma B2 of \cite{ghosal:shen:tokdar}, pp. 638--639, to relate 
$N_{\mathrm{KL}}(P_{0Y};\,\xi^2)$ to a Hellinger ball of appropriate radius. 
By Assumption \ref{ass:twicwtailcond}, there exists $C_0 >0$ such that  
$\mu_{0X}([-a,\,a]^c) \leq  e^{-(1+C_0)a}$ for $a$ large enough. 
Set $a_\eta:= a_0\log(1/\eta)$, with $a_0\geq2/(1+C_0)$ and $\eta>0$ small enough,
we have $\mu_{0X}([-a_\eta,\,a_\eta]^c) \leq\eta^2$. Then,
Lemma A.3 of \cite{ghosal2001}, p. 1261, shows that the $L^1$-distance between $f_{0Y}$ and 
$f_{0Y}^\ast:=f_\varepsilon\ast f_{0X}^\ast$, where $f_{0X}^\ast$ is the density of the renormalized restriction of $\mu_{0X}$ to $[-a_\eta,\,a_
\eta]$, denoted by $\mu_{0X}^\ast$, is bounded above by $2\eta^2$. From
$d_{\mathrm{H}}^2(f_{0Y},\,f^\ast_{0Y})\leq\|f_{0Y}-f^\ast_{0Y}\|_1\leq 2\eta^2$, we have $d_{\mathrm{H}}(f_{0Y},\,f^\ast_{0Y})\lesssim\eta$.
Lemma \ref{lem:discrete} applied to $\mu^\ast_{0X}$ (which plays the role of $\mu_{0X}$ in the statement) shows that, for $\sigma>0$ small enough, there exists a discrete probability measure $\mu_H^\ast$ supported on $[-a_\eta,\,a_\eta]$, with at most $N=O((a_\eta/\sigma)|\log\sigma|^{1/2})$ support points, such that $f_Y^\ast:=f_\varepsilon\ast (\mu_H^\ast\ast \phi_\sigma)$ satisfies 
$$d_{\mathrm{H}}(f_Y^\ast,\,f^\ast_{0Y})\lesssim \sigma^\beta.$$
An analogue of Corollary B1 in \cite{ghosal:shen:tokdar}, p. 16, shows that 
$\mu_H^\ast=\sum_{j=1}^Np_j\delta_{u_j}$ has support points inside $[-a_\eta,\,a_\eta]$, with at least 
$\sigma^{1+2\beta}$-separation between every pair of points $u_i\neq u_j$, 
and that $d_{\mathrm{H}}(f_Y^\ast,\,f^\ast_{0Y})\lesssim \sigma^{\beta}$.
Consider disjoint intervals $U_j$ centred at $u_1,\,\ldots,\,u_N$ with length $\sigma^{1+2\beta}$ each.
Extend $\{U_1,\,\ldots,\,U_N\}$ to a partition $\{U_1,\,\ldots,\,U_K\}$ of $[-a_\eta,\,a_\eta]$ such that each 
$U_j$, for $j=N+1,\,\ldots,\,K$, has length at most $\sigma$.
Further extend this to a partition $U_1,\,\ldots,\, U_M$ of $\mathbb R$ such that, for some constant $a_1>0$, 
we have $a_1\sigma\leq H_0(U_j)\leq 1$ 
for all $j=1,\,\ldots,\,M$. 
The whole process can be done with a total number $M$ of intervals of the same order as $N$.
Define $p_j=0$ for $j=N+1,\,\ldots,\,M$. Let $\mathscr P_\sigma$ be the set of probability measures $\mu_H\in\mathscr P$ 
with 
\begin{equation}\label{eq:diffprob}
\sum_{j=1}^K|\mu_H(U_j)-p_j|\leq 2\sigma^{2\beta+1} \quad \mbox{ and }\quad \min_{1\leq j\leq K}\mu_H(U_j)\geq \sigma^{2(2\beta+1)}/2.
\end{equation} 
Note that $\sigma^{2\beta+1}K<1$.
By Lemma 5 in \cite{ghosal2007}, p. 711, or 
Lemma B1 in \cite{ghosal:shen:tokdar}, p. 16, with $V_0:=\cup _{j>N}U_j$ and $V_j\equiv U_j$, $j=1,\,\ldots,\,N$, for any $\mu_H\in\mathscr P_\sigma$ 
we have $d_{\mathrm{H}}^2(f_Y,\,f^\ast_Y)\leq\|f_Y-f^\ast_Y\|_1\lesssim \sigma^{2\beta}$. Then, for 
$\eta=O(\sigma^\beta)$, we have
$d_{\mathrm{H}}^2(f_Y,\,f_{0Y}) \lesssim
d_{\mathrm{H}}^2(f_Y,\,f_Y^\ast) +d_{\mathrm{H}}^2(f_Y^\ast,\,f_{0Y}^*) + d_{\mathrm{H}}^2(f_{0Y}^*,\,f_{0Y})\lesssim \sigma^{2\beta}$.
To apply Lemma B2 of \cite{ghosal:shen:tokdar}, pp. 16--17, we study the quotient $(f_Y/f_{0Y})$.
Let $\mu_H\in\mathscr P_\sigma$.

\smallskip

\item[$\bullet$] \emph{Linnik error distribution with index $1<\beta<2$}

\noindent
For every $1<\beta<2$, we have $\|f_{0Y}\|_\infty<\mathbb E[V]<\infty$. 
Let $b_\sigma:=\sigma^{-2\beta/(\beta-1)}$. Since $a_\eta<b_\sigma$ and 
$f_\varepsilon(u)\sim\Gamma(1+\beta)\sin(\pi\beta/2)|u|^{-(1+\beta)}/\pi$, as $|u|\rightarrow\infty$, see 
the asymptotic expansion (4.3.39) in \cite{kotz1996}, p. 262,
for $|y|<b_\sigma$ with $\sigma$ small enough,
\[\begin{split}
\frac{f_Y}{f_{0Y}}(y)&\gtrsim \int_{|x|\leq a_\eta} f_\varepsilon (y-x)\int_{|x-u|\leq\sigma}\phi_\sigma(x-u)\mu_H(\di u)\,\di x\\
&\gtrsim \frac{a_\eta}{\sigma} f_\varepsilon(2b_\sigma)\mu_H(U_{J(x)})
\gtrsim a_\eta b_\sigma^{-(1+\beta)}\sigma^{4\beta+1},
\end{split}
\]
where $J(x)$ denotes the index $j\in\{1,\,\ldots,\,K\}$ for which $U_j\ni x$, because the interval $U_{J(x)}$ 
with length at most $\sigma$ is a subset of an interval of radius $\sigma$ centred at $x$. 
Analogously, for $|y|\geq b_\sigma$,
\[\begin{split}
\frac{f_Y}{f_{0Y}}(y)
&\gtrsim \int_{|x|\leq a_\eta} f_\varepsilon (y-x)\int_{|x-u|\leq\sigma}\phi_\sigma(x-u)\mu_H(\di u)\,\di x\\
&\gtrsim \frac{a_\eta}{\sigma}e^{-(|y|+a_\eta)}\mu_H(U_{J(x)})
\gtrsim a_\eta e^{-2|y|}\sigma^{4\beta+1}.
\end{split}
\]
For $\lambda=a_\eta b_\sigma^{-(1+\beta)}\sigma^{4\beta+1}$, 
we have $\log(1/\lambda)\lesssim \log(1/\sigma)$. 
Since $\{y:\,(f_Y/f_{0Y})(y)\leq\lambda\}\subseteq \{y:\,|y|\geq b_\sigma\}$, 
\begin{align*}
P_{0Y } \pt{ \left(\log\frac{f_{0Y}}{f_Y} \right)\1_{\big\{\frac{f_Y}{f_{0Y}}\leq \lambda\big\}}}
& \lesssim\int_{|y|\geq b_\sigma}
 \log\frac{f_{0Y}}{f_Y}(y) f_{0Y}(y)\,\di y\\&\lesssim\log(1/\sigma)\int_{|y|\geq b_\sigma}|y|f_{0Y}(y)\,\di y,
\end{align*}
where, by Assumption \ref{ass:twicwtailcond} on $f_{0X}$ that guarantees that
$\int_{\mathbb R}e^{|x|}f_{0X}(x)\,\di x<\infty$ and by Proposition 4.3.13 in \cite{Kotz2001}, p. 262, with $n=1$,
\[
\begin{split}
&\int_{|y|\geq b_\sigma}|y|f_{0Y}(y)\,\di y\\
&\hspace*{1.5cm}\leq \int_{|y|\geq b_\sigma}|y|
\int_{\mathbb R}\pq{\pt{e^{|x|}\int_0^1e^{-v|y|}+\int_1^\infty e^{-|y-x|}}vf_V(v;\,\beta)\,\di v}f_{0X}(x)\,\di x\,\di y\\
&\hspace*{1.5cm}\leq\pt{\int_{\mathbb R}e^{|x|}f_{0X}(x)\,\di x}
\int_{|y|\geq b_\sigma}|y|f_\varepsilon(y)\,\di y \lesssim
\int_{|y|\geq b_\sigma}|y|^{-\beta}\,\di y\lesssim b_\sigma^{-(\beta-1)}\lesssim \sigma^{2\beta}.
\end{split}
\]
It follows that 
$P_{0Y}(\log(f_{0Y}/f_Y)\1_{\{(f_Y/f_{0Y})\leq\lambda\}})\lesssim \sigma^{2\beta}$. 


\smallskip

\item[$\bullet$] \emph{Laplace error distribution \emph{($\beta=2$)}} 

\noindent 
Since $\|f_{0Y}\|_\infty\leq 1/2$, for $|y|<a_\eta$,
\begin{align*}
\frac{f_Y}{f_{0Y}}(y)
&\gtrsim \int_{|x|\leq a_\eta} f_\varepsilon (y-x)\int_{|x-u|\leq\sigma}\phi_\sigma(x-u)\mu_H(\di u)\,\di x\\
&\gtrsim \frac{a_\eta}{\sigma} e^{-2a_\eta}\mu_H(U_{J(x)}) 
\gtrsim a_\eta e^{-2a_\eta}\sigma^{4\beta+1},
\end{align*}
while, for $|y|\geq a_\eta$,
\[\begin{split}
\frac{f_Y}{f_{0Y}}(y)
&\gtrsim \int_{|x|\leq a_\eta} f_\varepsilon(y-x)\int_{|u|\leq a_\eta}\phi_\sigma(x-u)\mu_H(\di u)\,\di x
\gtrsim \frac{a_\eta}{\sigma}e^{-|y|}e^{-a_\eta}e^{-2(a_\eta/\sigma)^2},
\end{split}\]
where 
$\mu_H([-a_\eta,\,a_\eta])\geq 1-2\sigma^{2\beta+1}$ 
because of the first condition in \eqref{eq:diffprob}.
For $\lambda=a_\eta e^{-2a_\eta}\sigma^{4\beta+1}$, we have $\log(1/\lambda)\lesssim \log(1/\sigma)$.
Since $\{y:\,(f_Y/f_{0Y})(y)\leq\lambda\}\subseteq \{y:\,|y|\geq a_\eta\}$, 
\[
\begin{split}
P_{0Y}\hspace*{-0.1cm}\pt{\hspace*{-0.1cm}\pt{\log\frac{f_{0Y}}{f_Y}}\1_{\big\{\frac{f_Y}{f_{0Y}}\leq \lambda\big\}}}
&\lesssim\int_{|y|\geq a_\eta}\pt{\log\frac{f_{0Y}}{f_Y}(y)}f_{0Y}(y)\,\di y
\lesssim\frac{1}{\sigma^2}\int_{|y|\geq a_\eta}|y|^2f_{0Y}(y)\,\di y, 
\end{split}
\]
where
\[\begin{split}
\int_{|y|\geq a_\eta}y^2f_{0Y}(y)\,\di y
&\leq\pt{\int_{\mathbb R}e^{|x|}f_{0X}(x)\,\di x}
\int_{|y|\geq a_\eta}y^2f_\varepsilon(y)\,\di y\\
&\lesssim
\int_{|y|\geq a_\eta}e^{-|y|/2}\,\di y\lesssim e^{-a_\eta/2}\lesssim \sigma^6,
\end{split}
\]
see, \emph{e.g.}, Lemma A.7 in \cite{scricciolo:2011}, pp. 303--304,
provided that $a_0\geq\max\{12,\,2/(1+C_0)\}$. Consequently, 
$P_{0Y}(\log(f_{0Y}/f_Y)\1_{\{(f_Y/f_{0Y})\leq\lambda\}})\lesssim \sigma^4$.

\smallskip

\noindent
Thus, for every $1<\beta\leq 2$, we have $P_{0Y}(\log(f_{0Y}/f_Y)\1_{\{(f_Y/f_{0Y})\leq\lambda\}})\lesssim \sigma^{2\beta}$.
Lemma B2 of \cite{ghosal:shen:tokdar}, pp. 16--17, implies that $P_{0Y}(\log(f_{0Y}/f_Y))$ is bounded above by 
$\sigma^{2\beta}|\log \sigma|$. By Lemma 10 of \cite{ghosal2007}, p. 714, we have $\mathscr D_{H_0}(\mathscr P_\sigma)\gtrsim\exp{(-c_1K\log(1/\sigma))}
\gtrsim \exp{(-c_2(a_\eta/\sigma)\log^{3/2}(1/\sigma))}$ for constants $c_1,\,c_2>0$ that depend on $H_0(\mathbb R)$ and $a_1$.
Given $\sigma>0$, define $\mathscr S_\sigma:=\{\sigma':\,\sigma(1+\sigma^d)^{-1}\leq\sigma'\leq\sigma\}$ for a constant 
$0<d\leq s_1-1$. Then, $\Pi_\sigma(\mathscr S_\sigma)\gtrsim \exp{(-D_1\sigma^{-\gamma}\log^{t_1}(1/\sigma))}$.
Replace $\sigma$ at every occurrence with $\sigma'\in\mathscr S_\sigma$.
For $\xi:=\sigma^{\beta}\log^{1/2}(1/\sigma)$, noting that $\log(1/\sigma)\lesssim\log(1/\xi)$, 
since $\gamma\leq1$ we have
\[
\begin{split}
(\mathscr D_{H_0}\otimes \Pi_\sigma)(N_{\textrm{KL}}(P_{0Y};\,\xi^2))
&\gtrsim\exp{(-c_2(a_\eta/\sigma)\log^{3/2}(1/\sigma))}\exp{(-D_1\sigma^{-\gamma}\log^{t_1}(1/\sigma))}\\
&\gtrsim 
\exp{(-c_3(a_\eta/\sigma)\log^{(t_1\vee 3/2)}(1/\sigma))}\\
&\gtrsim 
\exp{(-c_4\xi^{-1/\beta}\log^{[1/(2\beta)+1+(t_1\vee 3/2)]}(1/\xi))}.
\end{split}
\]
Replacing $\xi$ with $\tilde\epsilon_n=n^{-\beta/(2\beta+1)}(\log n)^{1/2+\beta(t_1\vee 3/2)/(2\beta+1)}$,
for a suitable constant $C>0$ we have
$
(\mathscr D_{H_0}\otimes \Pi_\sigma)(N_{\textrm{KL}}(P_{0Y};\,\tilde\epsilon_n^2))\gtrsim \exp{(-C n\tilde\epsilon_n^2)}
$ and the proof is complete.
\end{proof} 


\section{Lemmas for Theorem \ref{thm:4} on adaptive
posterior contraction rates for Dirichlet Laplace-normal mixtures}


\smallskip

We introduce some more notation.
For $h=o(1)$, let $\delta=o(h)$. 
For $m\in\mathbb N$, $b=\mp1/2$ and $\sigma=o(1)$, we define the set
\begin{equation}\label{def:set}
A_{b,\sigma}:=\{x\in\mathbb R:\,\gamma h_{m,b,\sigma}(x)>-\bar h_{0,b}(x)/2\},
\end{equation}
with $\bar h_{0,b}$ and $h_{m,b,\sigma}$ as in \eqref{eq:barh} and \eqref{hm}, respectively,
as well as the function
\begin{equation}\label{def:functiong}
g_{b,\sigma}:=M_{0X}(b)e^{-b\cdot}\gamma h_{m,b,\sigma}\1_{A_{b,\sigma}}-\frac{1}{2}f_{0X}\1_{A_{b,\sigma}^c}.
\end{equation}

\smallskip

\allowdisplaybreaks

In the following lemma, we prove the existence of a compactly supported discrete mixing
probability measure, with a sufficiently small number of support points, 
such that the corresponding Laplace-normal mixture
has Hellinger distance of the order $O(\sigma^{\alpha+2})$ from the sampling density $f_{0Y}$
having an $\alpha$-Sobolev regular mixing density $f_{0X}$ with exponentially decaying tails.

\begin{lem}\label{lem:deltabound1}
Let $f_\varepsilon$ be the standard Laplace density.
Let $f_{0X}$ be a density 
satisfying Assumption \ref{ass:twicwtailcond}, Assumption \ref{ass:sobolevcond} for $\alpha>0$ and 
Assumption \ref{ass:smoothf0}.
For $\sigma>0$ small enough, there exist a constant $A_0>0$ and a discrete probability measure on $[-a_\sigma,\,a_\sigma]$, with 
$a_\sigma:=A_0|\log\sigma|$, having at most $N=O((a_\sigma/\sigma)|\log \sigma|^{1/2})$ support points, such that, for $f_Y:=f_\varepsilon\ast (\mu_H\ast\phi_\sigma)$ 
and $f_{0Y}:=f_\varepsilon\ast f_{0X}$,
$$d_{\mathrm H}(f_Y,\,f_{0Y})\lesssim \delta_0^{-1/2}e^{a_0/2}\sigma^{\alpha+2}$$
as soon as $P_{0X}(|X|\leq a_0)\geq \delta_0$ for some $0<a_0<a_\sigma$ and $0<\delta_0<1$.
\end{lem}

\begin{proof}
Reasoning as in Lemma \ref{lem:discrete}, for $a_0,\,\delta_0$ as in the statement, 
$d_{\mathrm H}^2(f_{Y},\,f_{0Y}) \leq 2\delta_0^{-1}e^{a_0}\|g_{Y}-g_{0Y}\|_2^2$, 
where $g_{Y}:=Uf_{Y}$ and $g_{0Y}:=Uf_{0Y}$, with $U$ defined in \eqref{eq:defU}.
Note that $(e^{|\cdot|/2}f_{0X})\in L^1(\mathbb R)\cap L^2(\mathbb R)$ by Assumption \ref{ass:twicwtailcond}. 
Also, $g_Y,\,g_{0Y}\in L^1(\mathbb R)\cap L^2(\mathbb R)$ so that, not only are 
the corresponding Fourier transforms
$\hat g_{Y},\,\hat g_{0Y}$ well defined, but $\|g_{Y}\|_2^2=(2\pi)^{-1}\|\hat g_{Y}\|_2^2$ and  
$\|g_{0Y}\|_2^2=(2\pi)^{-1}\|\hat g_{0Y}\|_2^2$. In order to bound $\|g_{Y}-g_{0Y}\|_2^2$, 
some definitions and preliminary facts are exposed. For $T\geq \lceil(\alpha+2)/\vartheta\rceil$, with $\vartheta$ suitably chosen later on,  we define the set $E_\sigma:=\{x\in\mathbb R:\,f_{0X}(x)>\sigma^T\}$. The tail condition on $f_{0X}$ of Assumption \ref{ass:twicwtailcond} implies that $E_\sigma \subset \{ |x| \leq  A_0 
|\log\sigma|\} $ for some $A_0>0$. Note that $A_0$ can be chosen arbitrarily large by choosing $T$ large enough because 
$A_0$ is proportional to $T/(1+C_0)$. Set $B_0:=\mathbb E_{0X}[f_{0X}^{-\vartheta}(X)]<\infty$, then 
\begin{equation}\label{eq:Ecompl}
P_{0X}(E_\sigma^c)\leq B_0 \sigma^{\vartheta T} \lesssim \sigma^{\alpha+2} 
\end{equation}
by definition of $T$. Introduced the densities
$$\bar h_{b,\sigma}:=\frac{f_{0X}+g_{b,\sigma}}{\|f_{0X}+g_{b,\sigma}\|_1}, \quad \text{and} \quad \frac{ \bar h_{b,\sigma}\1_{E_\sigma}}{\|\bar h_{b,\sigma}\1_{E_\sigma}\|_1},$$
where $g_{b,\sigma}$ is as defined in \eqref{def:functiong}, we consider the decomposition
\[\begin{split}
\|g_{0Y}-g_Y\|_2^2&\lesssim
\sum_{b=\mp 1/2}\|e^{b\cdot}\{f_\varepsilon\ast [f_{0X}-\phi_\sigma\ast (T_{m,b,\sigma}f_{0X})]\}\|_2^2
\\ &\quad +\sum_{b=\mp 1/2}\|e^{b\cdot}\{f_\varepsilon\ast \phi_\sigma\ast [(T_{m,b,\sigma}f_{0X})-(f_{0X}+g_{b,\sigma})]\}\|_2^2\\
&\quad+\sum_{b=\mp 1/2}\|e^{b\cdot}\{f_\varepsilon\ast \phi_\sigma\ast [(f_{0X}+g_{b,\sigma})-\bar h_{b,\sigma}]\}\|_2^2\\
&\quad+\sum_{b=\mp 1/2}\|e^{b\cdot}\{f_\varepsilon\ast \phi_\sigma\ast[\bar h_{b,\sigma}-(\bar h_{b,\sigma}\1_{E_\sigma}/\|\bar h_{b,\sigma}\1_{E_\sigma}\|_1)]\}\|_2^2\\
&\quad+\sum_{b=\mp 1/2}
\|e^{b\cdot}\{f_\varepsilon\ast \phi_\sigma\ast[(\bar h_{b,\sigma}\1_{E_\sigma}/\|\bar h_{b,\sigma}\1_{E_\sigma}\|_1)-\mu_H]\}\|_2^2
\\
& \quad =:\sum_{r=1}^5 V_r.
\end{split}
\]
We show that each term $V_1,\,\ldots,\,V_5$ is of the order $O(\sigma^{2(\alpha+2)})$.
First by inequality \eqref{bound:tildehm12} of Lemma \ref{lem:contapprox}, we have $V_1\lesssim \sigma^{2(\alpha+2)}$.


\noindent
\emph{Study of the term $V_2$}\\
We recall that $\varrho_b(t):=[1-\psi_b^2(t)]$, with $\psi_b(t):=-(\imath t+b)$, $t\in\mathbb R$, and that 
\begin{equation*}
h_{m,b,\sigma}:=\frac{1}{\gamma} \sum_{k=1}^{m-1}\frac{ (-1)^k \sigma^{2k}}{ 2^k k!}
\sum_{j=0}^{2k} \binom{2k}{j}(-b)^{2k-j}(\bar h_{0,b}\ast D^jH_\delta).
\end{equation*}
We write 
$$e^{b\cdot}\Big\{f_\varepsilon\ast \phi_\sigma\ast 
\Big[M_{0X}(b)e^{-b\cdot}\gamma h_{m,b,\sigma}\1_{A_{b,\sigma}^c}-\frac{1}{2}f_{0X}\1_{A_{b,\sigma}^c}\Big]\Big\} $$ 
as 
$$ M_{0X}(b) \Big\{ e^{b\cdot}[f_\varepsilon\ast \phi_\sigma]\ast 
\Big[\Big(\gamma h_{m,b,\sigma}-\frac{1}{2}\bar h_{0,b}\Big)\1_{A_{b,\sigma}^c}\Big]\Big\}$$
so that,
using the definition of $g_{b,\sigma}$ in \eqref{def:functiong},
\[\begin{split}
V_2&=\sum_{b=\mp 1/2}\Big\|e^{b\cdot}\Big\{f_\varepsilon\ast \phi_\sigma\ast 
\Big[M_{0X}(b)e^{-b\cdot}\gamma h_{m,b,\sigma}\1_{A_{b,\sigma}^c}-\frac{1}{2}f_{0X}\1_{A_{b,\sigma}^c}\Big]\Big\}\Big\|_2^2\\
&\lesssim\sum_{b=\mp 1/2}
\int_{\mathbb R}  \frac{|e^{\sigma^2 \psi_b(t)^2/2}|^2}{|\varrho_b(t)|^2 } 
| \mathcal F\{[2\gamma h_{m,b,\sigma}-\bar h_{0,b}]\1_{A_{b,\sigma}^c}\}(t)|^2\,\di t\\&\lesssim 
\sum_{b=\mp 1/2}(\|\gamma h_{m,b,\sigma}\1_{A_{b,\sigma}^c}\|_1+\|\bar h_{0,b}\1_{A_{b,\sigma}^c}\|_1 )^2,
\end{split}
\] 
where we have used the facts that 
$$\|\mathcal F\{[2\gamma h_{m,b,\sigma}-\bar h_{0,b}]\1_{A_{b,\sigma}^c}\}\|_\infty  \leq \|\gamma h_{m,b,\sigma}\1_{A_{b,\sigma}^c}\|_1+\|\bar h_{0,b}\1_{A_{b,\sigma}^c}\|_1 $$
and
$$  \frac{|e^{\sigma^2 \psi_b(t)^2/2}|^2}{|\varrho_b(t)|^2 } \lesssim \frac{ 1 }{1+t^4}.$$ 
Finally, using the inequalities in \eqref{cm12} of Lemma \ref{lem:Asigma}, we obtain that
\[
V_2 \lesssim \sigma^{2\upsilon R}\lesssim \sigma^{2(\alpha+2)}.
\] 
\medskip

\noindent
\emph{Study of the term $V_3$}\\
By the inequalities in \eqref{cm12} and \eqref{eq:norm1} of Lemma \ref{lem:Asigma}, noting that $\|\mathcal F\{\bar h_{0,b}\}\|_\infty\leq 1$, we have
\[\begin{split}
V_3&=
\sum_{b=\mp 1/2}\pt{1-\frac{1}{\|f_{0X}+g_{b,\sigma}\|_1}}^2
\|e^{b\cdot}[f_\varepsilon\ast \phi_\sigma\ast (f_{0X}+g_{b,\sigma})]\|_2^2\\
&\lesssim \sigma^{2\upsilon R}\sum_{b=\mp 1/2}\int_{\mathbb R}  \frac{  |e^{\sigma^2 \psi_b(t)^2/2}|^2 }{|\varrho_b(t)|^2} 
[|\mathcal F\{\bar h_{0,b}\}(t)|^2+|\mathcal F\{\gamma h_{m,b,\sigma}\}(t)|^2]\,\di t.
\end{split}
\]
Recalling that $|\mathcal F\{D^j H_\delta\}(t)|\leq |t|^j |\mathcal F \{H\} (\delta t)|\leq |t|^j$ for $j=0,\,\ldots,\,2k$, then
\[\begin{split}
\hspace*{-1cm}\int_{\mathbb R}  \frac{  |e^{\sigma^2 \psi_b(t)^2/2}|^2 }{|\varrho_b(t)|^2} 
|\mathcal F\{\gamma h_{m,b,\sigma}\}(t)|^2\,\di t &\lesssim
\sum_{k=1}^{m-1} \frac{e^{\sigma^2/4}}{(2^k k!)^2}
\int_{\mathbb R}\frac{[\sigma^2(t^2+1/4)]^{2k}}{e^{(\sigma t)^2}|\varrho_b(t)|^2}
|\mathcal F\{\bar h_{0,b}\}(t)|^2 \,\di t\\&
\lesssim 
\|\widehat{e^{b\cdot}f_{0X}}\|_2^2\lesssim\|e^{|\cdot|/2}f_{0X}\|_2^2<\infty, 
\end{split}
\]
which in turns implies that 
\[
V_3
\lesssim \sigma^{2\upsilon R}\lesssim \sigma^{2(\alpha+2)}.
\]
\medskip


\noindent
\emph{Study of the term $V_4$}\\
Taking into account that $\|f_{0X}+g_{b,\sigma}\|_1\geq 1$, see inequalities
\eqref{eq:norm1} of Lemma \ref{lem:Asigma}, we have
\[\begin{split}
V_4&\lesssim\sum_{b=\mp 1/2}\Big(\|\bar h_{b,\sigma}\1_{E_\sigma^c}\|_1^2\times
\|e^{b\cdot}\{f_\varepsilon\ast \phi_\sigma\ast(\bar h_{b,\sigma}\1_{E_\sigma}/\|\bar h_{b,\sigma}\1_{E_\sigma}\|_1)\}\|_2^2\\
&\hspace*{3cm}+
\|e^{b\cdot}\{f_\varepsilon\ast \phi_\sigma\ast(\bar h_{b,\sigma}\1_{E_\sigma^c})\}\|_2^2\Big)
\\
&\lesssim \sum_{b=\mp 1/2}\|(f_{0X}+g_{b,\sigma})\1_{E^c_\sigma}\|_1^2 \times
\int_{\mathbb R}  \frac{  |e^{\sigma^2 \psi_b(t)^2/2}|^2 }{|\varrho_b(t)|^2}|
\mathcal F\{e^{b\cdot}\bar h_{b,\sigma}\1_{E_\sigma}/\|\bar h_{b,\sigma}\1_{E_\sigma}\|_1\}(t)|^2 \,\di t\\
&\hspace*{3cm} +\sum_{b=\mp 1/2}
\int_{\mathbb R}  \frac{|e^{\sigma^2 \psi_b(t)^2/2}|^2 }{|\varrho_b(t)|^2}|\mathcal F\{e^{b\cdot}\bar h_{b,\sigma}\1_{E_\sigma^c}\}(t)|^2 \,\di t.
\end{split}
\]
Note that 
\[\begin{split}
\|(f_{0X}+g_{b,\sigma})\1_{E^c_\sigma}\|_1&\leq \frac{3}{2}P_{0X}(E_\sigma^c)+M_{0X}(b) \int_{A_{b,\sigma}\cap E_\sigma^c} e^{-bx}|\gamma h_{m,b,\sigma}(x)|\,\di x\\&
\leq \frac{3B_0}{2}\sigma^{\vartheta T}+M_{0X}(b) \int_{A_{b,\sigma}\cap E_\sigma^c} e^{-bx}|\gamma h_{m,b,\sigma}(x)|\,\di x,
\end{split}\] where, as hereafter shown, 
\begin{equation}\label{eq:setcomplem}
\int_{A_{b,\sigma}\cap E_\sigma^c} e^{-bx}|\gamma h_{m,b,\sigma}(x)|\,\di x\lesssim\sigma^{\alpha+2},
\end{equation}
and 
\begin{equation}\label{eq:fourierEcompl}
\|\mathcal F\{e^{b\cdot}\bar h_{b,\sigma}\1_{E_\sigma^c}\}\|_\infty\lesssim \sigma^{\alpha+2}.
\end{equation}
To prove inequality \eqref{eq:setcomplem}, note that, by Lemma \ref{lem:HolderHm} and inequality \eqref{eq:Ecompl}, for every $j=0,\,\ldots,\,2k$, since $R>2$, by H\"older's inequality,
\[\begin{split}
&\int_{A_{b,\sigma}\cap E_\sigma^c}\abs{\int_{\mathbb R}e^{-b u}f_{0X}(x-u)D^jH_\delta(u)\,\di u}\di x\\
&\hspace*{0.1cm}\leq
\int_{A_{b,\sigma}\cap E_\sigma^c}\abs{\int_{\mathbb R}[e^{-b u}f_{0X}(x-u)-f_{0X}(x)]D^jH_\delta(u)\,\di u}\di x \\
&\hspace*{3.5cm}+\int_{A_{b,\sigma}\cap E_\sigma^c}f_{0X}(x)\int_{\mathbb R}|D^jH_\delta(u)|\,\di u\,\di x \\
&\hspace*{0.1cm}\lesssim C_j\delta^{-j+\upsilon}\int_{A_{b,\sigma}\cap E_\sigma^c}[L_0(x)+f_{0X}(x)]\,\di x+
C_{0,j}\int_{A_{b,\sigma}\cap E_\sigma^c}f_{0X}(x)\,\di x\\
&\hspace*{0.1cm}\lesssim C_j\delta^{-j+\upsilon}\int_{A_{b,\sigma}\cap E_\sigma^c}\pt{\frac{L_0(x)}{f_{0X}(x)}}f_{0X}^{1/R}(x)
f_{0X}^{1-1/R}(x)\,\di x+[C_{0,j}+C_j\delta^{-j+\upsilon}] P_{0X}(E_\sigma^c)\\
&\hspace*{0.2cm}\lesssim \delta^{-j+\upsilon}
\pt{\int_{A_{b,\sigma}\cap E_\sigma^c}f_{0X}(x)\pt{\frac{L_0(x)}{f_{0X}(x)}}^R\,\di x}^{1/R}
 P_{0X}(E_\sigma^c)^{1-1/R}+ \sigma^{\vartheta T-j+\upsilon}\\
&\hspace*{0.1cm}\lesssim \delta^{-j+\upsilon}
\pt{\int_{\mathbb R}f_{0X}(x)\pt{\frac{L_0(x)}{f_{0X}(x)}+1}^R\,\di x}^{1/R}[P_{0X}(E_\sigma^c)]^{1-1/R}
+\sigma^{\vartheta T-j+\upsilon}\\
&\hspace*{0.1cm}\lesssim \delta^{-j+\upsilon}
\pt{\int_{\mathbb R}f_{0X}(x)\pt{\frac{L_0(x)}{f_{0X}(x)}+1}^R\,\di x}^{1/R}\sigma^{\vartheta T(1-1/R)}
+\sigma^{\vartheta T-j+\upsilon }.
\end{split}\]
Consequently,
\[
\int_{A_{b,\sigma}\cap E_\sigma^c} e^{-bx} |\gamma h_{m,b,\sigma}(x)|\,\di x\lesssim
\sigma^{\upsilon+\vartheta T(1-1/R)}+\sigma^{\vartheta T}\lesssim \sigma^{\alpha +2}
\]
by choosing $[(\alpha+2)/T]\leq \vartheta<[1\wedge(\upsilon R/T)]$. 
Thus  $\|(f_{0X}+g_{b,\sigma})\1_{E^c_\sigma}\|_1\lesssim\sigma^{\alpha+2}$.
Inequality \eqref{eq:fourierEcompl} essentially follows from the tail condition on $f_{0X}$ of Assumption \ref{ass:twicwtailcond} 
and the definition of the set $E_\sigma^c$.

\medskip


\noindent
\emph{Study of the term $V_5$}\\
Recalling that $\mathcal B$ stands for the bilateral transform operator, we have
\[\begin{split}
V_5=&\sum_{b=\mp 1/2}
\int_{\mathbb R}  \frac{|e^{\sigma^2 \psi_b(t)^2/2}|^2 }{|\varrho_b(t)|^2} |[\mathcal B\{\bar h_{b,\sigma}\1_{E_\sigma}/\|\bar h_{b,\sigma}\1_{E_\sigma}\|_1\}-\mathcal B\{\mu_H\}](\psi_b(t))|^2\,\di t.
\end{split}\]
For $M>0$, split the integral domain into $|t|\leq M$ and $|t|>M$
and let the corresponding terms be denoted by $V_5^{(1)}$ and $V_5^{(2)}$.
Let $\mu_H$ be a discrete probability measure on $E_\sigma\subseteq [-a_\sigma,\,a_\sigma]$ satisfying the constraints
\begin{equation*}\label{eq:c22}
\int_{E_\sigma} u^j\mu_H(\di u)=
\int_{E_\sigma} u^j\frac{\bar h_{b,\sigma}(u)}{\|\bar h_{b,\sigma}\1_{E_\sigma}\|_1}\,\di u, \quad\mbox{for $j=1,\,\ldots,\,J-1$,}
\end{equation*}
with $J=\lceil\eta ea M\rceil$ for some $\eta>1$, together with
\begin{equation}\label{eq:c23}
\int_{E_\sigma}e^{bu}\mu_H(\di u)=\int_{E_\sigma} e^{bu}
\frac{\bar h_{b,\sigma}(u)}{\|\bar h_{b,\sigma}\1_{E_\sigma}\|_1}\,\di u,
\end{equation}
where the integral on the right-hand side of \eqref{eq:c23} is finite because 
$\int_{E_\sigma} e^{bu}
\bar h_{b,\sigma}(u)\,\di u\leq \int_{\mathbb R} e^{bu}
\bar h_{b,\sigma}(u)\,\di u\lesssim 
M_{0X}(b)[1+\int_{\mathbb R}\gamma h_{m,b,\sigma}(u)\,\di u]=
M_{0X}(b)\{1+\gamma [1+O(\sigma^{m-2})]\}$
by relationship \eqref{eq:bound} of Lemma \ref{lem:contapprox}. Thus, 
$$\int_{E_\sigma} e^{bu}
\bar h_{b,\sigma}(u)\,\di u=O(1).$$
Also, by the lower bound inequality in \eqref{eq:norm1} of Lemma \ref{lem:Asigma} and the previously proven fact that 
$\|(f_{0X}-g_{b,\sigma})\1_{E_\sigma^c}\|_1\lesssim \sigma^{\alpha+2}$, we have
$\|\bar h_{b,\sigma}\1_{E_\sigma}\|_1=1-\|\bar h_{b,\sigma}\1_{E_\sigma^c}\|_1\gtrsim1-\|(f_{0X}-g_{b,\sigma})\1_{E_\sigma^c}\|_1
\gtrsim1-\sigma^{\alpha+2}$. Therefore,
$$\|\mathcal B\{\bar h_{b,\sigma}\1_{E_\sigma}/\|\bar h_{b,\sigma}\1_{E_\sigma}\|_1\}(\psi_b)\|_\infty\leq
(\|\bar h_{b,\sigma}\1_{E_\sigma}\|_1)^{-1}\|e^{b\cdot}\bar h_{b,\sigma}\1_{E_\sigma}\|_1
\lesssim \int_{E_\sigma} e^{bu}
\bar h_{b,\sigma}(u)\,\di u.$$
Then, using Lemma \ref{lem:diseg} with $r=J$, by the inequality $J!\geq (J/e)^J$, we have 
\[\begin{split}
V_5^{(1)}&:=\sum_{b=\mp 1/2}
\int_{|t|\leq M}  \frac{|e^{\sigma^2 \psi_b(t)^2/2}|^2 }{|\varrho_b(t)|^2} |[\mathcal B\{\bar h_{b,\sigma}\1_{E_\sigma}/\|\bar h_{b,\sigma}\1_{E_\sigma}\|_1\}-\mathcal B\{\mu_H\}](\psi_b(t))|^2\,\di t\\
&\lesssim \frac{a^{2J}}{(J!)^2}
\int_{0}^M t^{2(J-2)}\,\di t \\
&\lesssim M^{-2(\alpha+2)} \times \frac{a^{2J}}{(J!)^2}\times
\frac{M^{2(J+\alpha)+1}}{2J-3}\lesssim  M^{-2(\alpha+2)} \pt{\frac{eaM}{J}}^{2J+1} M^{2\alpha}\lesssim M^{-2(\alpha+2)}
\end{split}\]  
because $(eaM/J)^{2J+1} M^{2\alpha}<e^{-2(\log\eta)J}M^{2\alpha}<1$. Choosing $M$ so that  
$(\sigma M)^2\geq (2\alpha+1)|\log\sigma|$, equivalently, $M\geq\sigma^{-1}[(2\alpha+1)|\log\sigma|]^{1/2}$, and taking into account that $|e^{\sigma^2\psi_b^2(t)/2}|^2=O(e^{-\sigma^2t^2})$, we have
\[\begin{split}
V_5^{(2)}&:=\sum_{b=\mp 1/2}
\int_{|t|>M}  \frac{|e^{\sigma^2 \psi_b(t)^2/2}|^2 }{|\varrho_b(t)|^2} |[\mathcal B\{\bar h_{b,\sigma}\1_{E_\sigma}/\|\bar h_{b,\sigma}\1_{E_\sigma}\|_1\}-\mathcal B\{\mu_H\}](\psi_b(t))|^2\,\di t\\
&\lesssim
e^{-(\sigma M)^2}
\int_{|t|>M}t^{-4}
\,\di t
\lesssim
e^{-(\sigma M)^2}M^{-3}\lesssim 
\sigma^{2\alpha+1} M^{-3}
\lesssim \sigma^{2(\alpha+2)}.
\end{split}
\]
Therefore, $V_5\lesssim \sigma^{2(\alpha+2)}$.

Finally $\|g_Y-g_{0Y}\|_2^2\lesssim\sum_{r=1}^5 V_r\lesssim \sigma^{2(\alpha+2)}$ and the assertion follows.
\end{proof}


\section{Technical lemmas}

\begin{lem}\label{lem:H}
For $r\geq 0$, $a\in \mathbb R$ and $j \in \{0\}\cup\mathbb N$ fixed, there exists 
a constant $C_{r,j}< \infty$ such that, for $h=o(1)$  and $\delta=o(h)$,
\begin{equation}\label{eq:constant}
  \int_{\mathbb R} |x|^r e^{a\delta x} |D^jH(x) |\,\di x \leq C_{r,j} .
\end{equation}
\end{lem}

\begin{proof}
Recalling that $H(x)  = (2\pi)^{-1}\hat \tau(x) e^{-(h x)^2/2}$, $x\in\mathbb R$,
we have 
$$D^jH(x) =\frac{1}{2\pi}\sum_{i=0}^j\binom{j}{i}\hat\tau^{(i)}(x)D^{j-i}
e^{-(h x)^2/2},\quad x\in\mathbb R,$$
where $D^{j-i}e^{-(h x)^2/2}$ is a linear combination of terms of the form $e^{-(h x)^2/2}(-1)^{j_1}h^{2j_2}x^{j_3}$, where $0\leq j_1,\,j_2,\,j_3\leq (j-i)$. Note that $ e^{a\delta x}  e^{-(h x)^2/2}  = e^{ -(hx-a\delta/h)^2/2}e^{(a\delta/h)^2/2} \leq e^{(a\delta/h)^2/2}$,
where $e^{(a\delta/h)^2/2}= 1 + o(1)$ because $(\delta/h)=o(1)$. Then, by condition \eqref{eq:poldec}, for $\nu>(r+j+1)$ and $0\leq j_1,\,j_2,\,j_3\leq (j-i)$,
$$  |x|^{r} e^{a\delta x}  |\hat\tau^{(i)}(x)|e^{-(h x)^2/2}|x|^{j_3}\lesssim
|x|^{r+j_3} |\hat\tau^{(i)}(x)|,$$
where the function on the right-hand side is integrable. The assertion follows.
\end{proof}

\begin{lem}\label{lem:HolderHm}
Suppose that $f_{0X}$ satisfies the local H\"{o}lder condition \eqref{eq:holderf0} of Assumption \ref{ass:smoothf0} for $0<\upsilon\leq1$ and $L_0\in L^1(\mathbb R)$.
For every $b\in\mathbb R$ and $j\in\{0\}\cup\mathbb N$, if $h=o(1)$ and $\delta=o(h)$, then
\begin{equation}\label{eq:f+L}
\abs{\int_{\mathbb R}  [e^{-bu}f_{0X}(x-u) - f_{0X}(x) ] D^jH_\delta(u)\,\di u}
\leq \frac{C_j}{\delta^j}\delta^\upsilon [L_0(x)+f_{0X}(x)],\quad x\in\mathbb R,
\end{equation}
where $C_j:=\max\{3C_{\upsilon,j},\, C_{1,j}\}>0$, with $C_{\upsilon,j}$ as in \eqref{eq:constant}.
\end{lem}

\begin{proof}
For every $x\in\mathbb R$, by the local H\"{o}lder condition \eqref{eq:holderf0} of Assumption \ref{ass:smoothf0} and Lemma \ref{lem:diseg}, 
\[
\begin{split}
&\delta^j\abs{\int_{\mathbb R}  [e^{-bu}f_{0X}(x-u) - f_{0X}(x) ] D^jH_\delta(u)\,\di u}\\
&\hspace*{1cm}=
\abs{\int_{\mathbb R}  [e^{-b\delta z}f_{0X}(x-\delta z) - f_{0X}(x) ] D^jH(z)\,\di z}\\
&\hspace*{1cm}\leq 
\int_{\mathbb R}  |e^{-b\delta z}-1||f_{0X}(x-\delta z)-f_{0X}(x)||D^jH(z)|\,\di z \\
&\hspace*{7.1cm}+ f_{0X}(x)
\int_{\mathbb R}  |e^{-b\delta z}-1| |D^jH(z)|\,\di z\\
&\hspace*{7.1cm} + \int_{\mathbb R}|f_{0X}(x-\delta z) - f_{0X}(x)||D^jH(z)|\,\di z\\
 & \hspace*{1cm}\leq 
3 \int_{\mathbb R}|f_{0X}(x-\delta z) - f_{0X}(x)||D^jH(z)|\,\di z +
f_{0X}(x)
\int_{\mathbb R}  |e^{-b\delta z}-1| |D^jH(z)|\,\di z \\
&\hspace*{1cm}\leq3\delta^{\upsilon}L_0(x) \int_{\mathbb R}|z|^{\upsilon}|D^jH(z)|\,\di z
+
b\delta f_{0X}(x) \int_{\mathbb R}|z||D^jH(z)|\,\di z\\
 & \hspace*{1cm}
\leq 
3\delta^{\upsilon} C_{\upsilon,j}L_0(x) + b\delta C_{1,j} f_{0X}(x)
<C_j\delta^{\upsilon} [L_0(x)+f_{0X}(x)].
\end{split}
\]
Inequality \eqref{eq:f+L} follows.
\end{proof}

\begin{lem} \label{lem:Asigma}
For $h=o(1)$, let $\delta=o(h)$. 
For $m\in\mathbb N$, $b=\mp1/2$ and $\sigma=o(1)$, let the set $A_{b,\sigma}$ 
be defined as in \eqref{def:set}.
Under Assumptions \ref{ass:twicwtailcond} and 
\ref{ass:smoothf0} on $f_{0X}$, the latter with $0<\upsilon\leq1$, $L_0\in L^1(\mathbb R)$ and 
any $R>0$,
there exists a constant $\bar C_{m}>0$, depending on $m$ and $\upsilon$, such that, for $\sigma$ small enough,
\begin{equation}\label{eq:inclus}
\forall\,b=\mp1/2,\quad A_{b,\sigma}^c\subseteq B_\sigma,
\end{equation}
with $B_\sigma:=
\{ x:\,[L_0(x)+f_{0X}(x)]> \bar C_{m}^{-1}\sigma^{-\upsilon} f_{0X}(x)\}$.
Furthermore, there exist constants $C_R,\,S_R>0$, depending on $m$, $\upsilon$ and $R$, so that
\begin{equation}\label{cm12}
\|\bar h_{0,b}\1_{A_{b,\sigma}^c}\|_1< C_R \sigma^{\upsilon R},\quad
\|\gamma h_{m,b,\sigma}\1_{A_{b,\sigma}^c}\|_1 <2C_R\sigma^{\upsilon R}
\end{equation}
and the function $f_{0X}+g_{b,\sigma}$, with $g_{b,\sigma}$ as defined in \eqref{def:functiong}, which is non-negative, has 
\begin{equation}\label{eq:norm1}
1\leq\|f_{0X}+g_{b,\sigma}\|_1\leq1+S_R\sigma^{\upsilon R}.
\end{equation}
\end{lem}

\begin{proof}
Let $b$ be fixed. We begin by proving the inclusion in \eqref{eq:inclus}. 
Assume that $x \in  A_{b,\sigma}^c$, \emph{i.e.},  $\gamma h_{m,b,\sigma} \leq - \bar h_{0,b}/2$. Recall that 
$$\gamma h_{m,b,\sigma} = \sum_{k=1}^{m-1}\frac{ (-1)^k \sigma^{2k}}{ 2^k k!}  \sum_{j=0}^{2k} \binom{2k}{j}(-b)^{2k-j}
\int_{\mathbb R} \bar h_{0,b}(x-u) D^jH_\delta(u)\, \di u $$
and note that 
 $\int_{\mathbb R}D^jH_\delta(u)\,\di u=[\int_{\mathbb R}H(x)\,\di x]\1_{\{0\}}(j)=(\tau\ast \phi_h)(0)\1_{\{0\}}(j)\leq 1$. 
 Then
 \begin{equation*}\label{eq:sufcond}
\begin{split}
 \sum_{k=1}^{m-1} & \frac{ (-1)^k \sigma^{2k} }{ 2^k k!}  \sum_{j=0}^{2k} \binom{2k}{j}(-b)^{2k-j}
\int_{\mathbb R} [ \bar h_{0,b}(x-u) -  \bar h_{0,b}(x)] D^jH_\delta(u)\, \di u \\
& = \gamma h_{m,b,\sigma}(x) - \bar h_{0,b} (x)\sum_{k=1}^{m-1}\frac{ (-1)^k \sigma^{2k}}{ 2^k k!}  \sum_{j=0}^{2k} \binom{2k}{j}(-b)^{2k-j}\int_{\mathbb R}  D^jH_\delta(u)\, \di u \\
&\leq - \bar h_{0,b}(x)\left( \frac{ 1 }{ 2 } - (\tau\ast\phi_h)(0)\frac{(b\sigma)^2  }{  2 } \sum_{k=0}^{m-2}\frac{ (-1)^k (b\sigma)^{2k}}{ 2^k (k+1)!}\right)\\ &\leq  - \frac{ \bar h_{0,b}(x)  }{ 2 }\left( 1  -  b^2\sigma^2 \right)<  - \frac{ \bar h_{0,b} (x) }{ 4 }.
\end{split}
\end{equation*}
Moreover, for $\delta=c_\delta\sigma$, with $0<c_\delta<1$ as in Lemma \ref{lem:contapprox}, when $\sigma$ is small enough, by Lemma \ref{lem:HolderHm},
\[\begin{split}
&\hspace*{-0.1cm}\abs{\sum_{k=1}^{m-1}\frac{ (-1)^k \sigma^{2k}}{ 2^k k!}  \sum_{j=0}^{2k} \binom{2k}{j}(-b)^{2k-j} \int_{\mathbb R} [ \bar h_{0,b}(x-u) -  \bar h_{0,b}(x)]  D^jH_\delta(u)\, \di u}\\
&\hspace*{0.5cm}\leq\frac{1}{M_{0X}(b)}\sum_{k=1}^{m-1}\frac{\sigma^{2k}}{ 2^k k!}  \sum_{j=0}^{2k} \binom{2k}{j}|b|^{2k-j} e^{bx}\int_{\mathbb R} |e^{-bu} f_{0X}(x-u) - f_{0X}(x)||  D^jH_\delta(u)|\, \di u\\
&\hspace*{0.5cm}<\frac{1}{M_{0X}(b)}\pt{\sum_{k=1}^{m-1}\frac{(1+|b|c_\delta\sigma)^{2k} }{(2c_\delta^2)^k k!}\max_{0\leq j\leq 2k}C_j}
\sigma^{\upsilon}e^{bx}[L_0(x)+f_{0X}(x)]\\
&\hspace*{0.5cm}<\frac{1}{M_{0X}(b)}\underbrace{
\pt{ \sum_{k=1}^{m-1}\frac{2^k\max_{0\leq j\leq 2k}C_j}{c_\delta^{2k} k!}}
}_{=:\tilde C_{m}}
\sigma^{\upsilon}e^{bx}[L_0(x)+f_{0X}(x)]\\
\end{split}\]
where $0<\tilde C_{m}<\infty$. 
Then, for $\bar C_m:=4\tilde C_m$, we have $A_{b,\sigma}^c  \subseteq B_\sigma$.

We now prove the inequalities in \eqref{cm12}. Concerning the first one,
$$\int_{A_{b,\sigma}^c} \bar h_{0,b}(x)\,\di x < \sigma^{\upsilon R}
\frac{\bar C_{m}^R}{M_{0X}(b)} \int_{B_\sigma} e^{bx} f_{0X}(x) \pt{\frac{ L_0(x) }{ f_{0X}(x) } +1}^R\di x \leq C_R\sigma^{\upsilon R},$$
where 
$$C_R = \sigma^{\upsilon R}\underbrace{\frac{\bar C_{m}^R}{M_{0X}(-1/2)\wedge M_{0X}(1/2)} \int_{\mathbb R} e^{|x|/2} f_{0X}(x) \pt{\frac{ L_0(x) }{ f_{0X}(x) } +1}^R\di x}_{=:C_R} > \infty$$
 because of condition \eqref{eq:ratiofo} and Assumption \ref{ass:twicwtailcond}.
Concerning the second inequality, from previous computations,
\begin{equation*}
\begin{split} 
\int_{A_{b,\sigma}^c}| (\bar h_{0,b}\ast D^jH_\delta)(x)|\,\di x& \leq   \delta^{-j}
\int_{A_{b,\sigma}^c}\int_{\mathbb R}|\bar h_{0,b}(x-\delta u) -\bar h_{0,b}(x)||D^jH(u)|\,\di u
+ C_{0,j} C_R \sigma^{\upsilon R}\\
& \leq \frac{C_j}{M_{0X}(b)}\delta^{-j+\upsilon}\int_{A_{b,\sigma}^c} e^{bx}f_{0X}(x)  \left(\frac{ L_0(x) }{ f_{0X}(x)} +1\right)\, \di x+  C_{0,j}C_R\sigma^{\upsilon R} \\
&< \pt{\frac{\bar C_{m}^{R-1}C_j}{M_{0X}(b)}(c_\delta\sigma)^{-j}+ C_{0,j}C_R} \sigma^{\upsilon R},\quad j\in\{0\}\cup\mathbb N,
\end{split}
\end{equation*}
which implies that 
$\|\gamma h_{m,b,\sigma}\1_{A_{b,\sigma}^c}\|_1<2C_R\sigma^{\upsilon R}$.

\smallskip

To prove the last part of the lemma, we begin by noting that 
$$f_{0X}+g_{b,\sigma}=[f_{0X}+M_{0X}(b)e^{-bx}\gamma h_{m,b,\sigma}]\1_{A_{b,\sigma}}+\frac{1}{2}f_{0X}\1_{A_{b,\sigma}^c}>
\frac{1}{2}f_{0X}\geq0$$
and
\[M_{0X}(b)\int_{\mathbb R}e^{-bx}\gamma h_{m,b,\sigma}(x)\,\di x=0.\]
In fact, since $\int_{\mathbb R}e^{-b\delta x}H(x)\,\di x<\infty$ because $(\delta/h)=o(1)$, we have
\[\begin{split}
M_{0X}(b)\int_{\mathbb R}e^{-bx}\gamma h_{m,b,\sigma}(x)\,\di x&
=\sum_{k=1}^{m-1}
\frac{(-1)^k\sigma^{2k}}{2^kk!}\sum_{j=0}^{2k}\binom{2k}{j}(-b)^{2k-j}\int_{\mathbb R}e^{-bu}D^jH_\delta(u)\,\di u\\
&=\sum_{k=1}^{m-1}
\frac{(-1)^k\sigma^{2k}}{2^kk!}\underbrace{\sum_{j=0}^{2k}\binom{2k}{j}(-b)^{2k-j}b^j}_{(-b+b)^{2k}=0}\int_{\mathbb R}e^{-b\delta x}H(x)\,\di x=0.
\end{split}\]
Then, since
$g_{b,\sigma}=M_{0X}(b)e^{-b\cdot}\gamma h_{m,b,\sigma}-[M_{0X}(b)e^{-b\cdot}\gamma h_{m,b,\sigma}+(f_{0X}/2)]\1_{A_{b,\sigma}^c}$,
we have
\[
\begin{split}
\int_{\mathbb R}(f_{0X}+g_{b,\sigma})(x)\,\di x&=1+\underbrace{\int_{\mathbb R}M_{0X}(b)e^{-bx}\gamma h_{m,b,\sigma}(x)\,\di x}_{=0}\\
&\hspace*{2.5cm}
-\int_{A_{b,\sigma}^c}
\pq{M_{0X}(b)e^{-bx}\gamma h_{m,b,\sigma}(x)+\frac{1}{2}f_{0X}(x)}\,\di x\\&= 1-\int_{A_{b,\sigma}^c}
\Big[\underbrace{M_{0X}(b)e^{-bx}\gamma h_{m,b,\sigma}(x)}_{\leq -\frac{1}{2}f_{0X}(x)}+\frac{1}{2}f_{0X}(x)\Big]\,\di x\geq 1.
\end{split}
\]
On the other side, using Lemma \ref{lem:HolderHm} and reasoning as in the first part of the present lemma,
\[
\begin{split}
\int_{\mathbb R}(f_{0X}+g_{b,\sigma})(x)\,\di x&=1
-\int_{A_{b,\sigma}^c}
\pq{M_{0X}(b)e^{-bx}\gamma h_{m,b,\sigma}(x)+\frac{1}{2}f_{0X}(x)}\,\di x\\&\leq
1+
\int_{A_{b,\sigma}^c}
\pq{M_{0X}(b)e^{-bx}|\gamma h_{m,b,\sigma}(x)|+\frac{1}{2}f_{0X}(x)}\,\di x\\
&\leq 1+\underbrace{2[M_{0X}(-1/2)\vee M_{0X}(1/2)]C_R}_{=:S_R}\sigma^{\upsilon R}.
\end{split}
\]
Conclude that $1\leq\|f_{0X}+g_{b,\sigma}\|_1\leq1+S_R\sigma^{\upsilon R}$. The proof is thus complete.
\end{proof}

\begin{rmk}
\emph{
Although in condition \eqref{eq:holderf0} of Assumption \ref{ass:smoothf0} the constant $R$ 
is such that $R\geq m/\upsilon$, with the smallest integer $m\geq(\alpha+2)$, in Lemma \ref{lem:Asigma}, indeed,
$R$ can be any positive real.
}
\end{rmk}


\section{\emph{Lemma for Theorem \ref{thm:5} on adaptive
posterior contraction rates for $L^1$-Wasserstein deconvolution of Laplace mixtures}}

The following lemma assesses the order of the bias of the distribution function
corresponding to a Gaussian mixture, where the mixing distribution 
is any probability measure on the real line and the scale 
parameter is equal to the kernel bandwidth times a logarithmic factor.
It shows that condition \eqref{eq:ass1} of Theorem \ref{theo:1} is verified 
for a universal constant $C_1$.

\begin{lem}\label{lem:biasmixgaus}
Let $F_{X}$ be the distribution function of $\mu_X=\mu_H\ast \phi_\sigma$, 
with $\mu_{H}\in\mathscr P$ and $\sigma>0$. 
Let $K\in L^1(\mathbb R)\cap L^2(\mathbb R)$ be symmetric, with 
$\int_{\mathbb R}|z||K(z)|\,\di z<\infty$, and $\hat K\in L^1(\mathbb R)$ such that
$\hat K\equiv 1$ on $[-1,\,1]$.
Given $h>0$, for $\sigma=O(\sqrt{2}h|\log h^{2\alpha+1}|^{1/2})$, we have
\begin{equation}\label{eq:derivative21}
\|F_{X}\ast K_h-F_{X}\|_1=O(h^{\alpha+1}).
\end{equation}
\end{lem}

\begin{proof}
Let $b_{F_X}:=(F_X\ast K_h-F_X)$. 
Defined $\hat f(t):=[1-\hat K(ht)][\hat\phi(\sigma t/\sqrt{2})\1_{[-1,\,1]^c}(ht)/t]$, $t\in\mathbb R$,
since $t\mapsto [\hat\mu_H(t)\hat\phi(\sigma t/\sqrt{2})]\hat f(t)$ is in $L^1(\mathbb R)$, 
arguing as for $G_{2,h}$ in \cite{dedecker2015}, pp. 251--252, we have
\[\begin{split}
\|b_{F_X}\|_1&=
\int_{\mathbb R}\bigg|\frac{1}{2\pi}\int_{|t|>1/h}
\exp{(-\imath t x)}\hat \mu_H(t)\hat\phi(\sigma t/\sqrt{2})\hat f(t)\,\di t
\bigg|\,\di x\\
&=\|\mu_H\ast\phi_{\sigma/\sqrt{2}}\ast f\|_1\leq \|\mu_H\ast\phi_{\sigma/\sqrt{2}}\|_1\times\| f\|_1=\| f\|_1
\leq (\|\hat f\|_2^2+\|\hat f^{(1)}\|_2^2)^{1/2}/\sqrt{2},
\end{split}\]
see, \emph{e.g.}, \cite{Bobkov:2016}, p. 1031, for the last inequality. 
Since $\int_{1/h}^\infty\hat\phi(\sigma t)\,\di t\lesssim (h/\sigma^2)\hat \phi(\sigma/h)$, we have
$\|\hat f\|_2^2\lesssim\hat\phi(\sigma/h)\int_{|t|>1/h}
[1+|\hat K(ht)|^2]t^{-2}\,\di t\lesssim h^{2(\alpha+1)}$ because $\|\hat K\|_\infty\leq \|K\|_1<\infty$.
Write
\[
\hat f^{(1)}(t)=-\pg{h\hat K^{(1)}(ht)+\pt{\frac{1}{t
}+\frac{\sigma^2}{2} t
}[1-\hat K(ht)]
}
\frac{\hat\phi(\sigma t/\sqrt{2})}{t}\1_{[-1,\,1]^c}(ht),\quad t\in\mathbb R.
\]
Since $K\in L^1(\mathbb R)$ and $zK(z)\in L^1(\mathbb R)$ jointly imply that $\hat K$ is 
continuously differentiable with $|\hat K^{(1)}(t)|\rightarrow 0$ as $|t|\rightarrow \infty$, 
so that $\hat K^{(1)}\in C_b(\mathbb R)$, we have
$\|\hat f^{(1)}\|_2^2\lesssim \int_{|t|>1/h}
\{h^4|\hat K^{(1)}(ht)|^2+(h^4+\sigma^4)[1+|\hat K(ht)|^2]\}\hat\phi(\sigma t)\,\di t \lesssim 
h^{2(\alpha+1)}$. Hence, $\|\hat f^{(1)}\|_2^2\lesssim h^{2(\alpha+1)}$. The assertion follows.
\end{proof}

\begin{rmk}
\emph{Due to the exponentially decaying tails of the Gaussian density and 
a suitable choice of the scale parameter $\sigma$ as a 
multiple of the kernel bandwidth $h$, times a logarithmic factor, 
a different argument than that used in Lemma \ref{lem:sob} is used to bound the bias of $F_X$. 
}
\end{rmk}

\bibliographystyle{imsart-number}
\bibliography{biblio}


\end{document}